\documentclass[10pt,reqno]{amsart}
\usepackage[margin=1in]{geometry}
\usepackage{dsfont}
\usepackage{enumitem} 
\usepackage{amsmath,amssymb,amsthm,graphicx,amsxtra, setspace}
\usepackage[utf8]{inputenc}
\usepackage{mathrsfs}
\usepackage{alltt}
\usepackage{relsize}
\usepackage{hyperref}
\usepackage{aliascnt}
\usepackage{tikz}
\usepackage{mathtools}
\usepackage{multicol}
\usepackage{upgreek}
\usepackage{multirow}
\usepackage{subfigure}
\usepackage{subfigmat}
\usepackage{graphicx,type1cm,eso-pic,color}
\allowdisplaybreaks
\usepackage{lipsum}
\usepackage{natbib}

\newtheorem{theorem}{Theorem}[section]

\newaliascnt{lemma}{theorem}
\newtheorem{lemma}[lemma]{Lemma}
\aliascntresetthe{lemma} 

\newaliascnt{proposition}{theorem}

\aliascntresetthe{proposition}

\newaliascnt{assumption}{theorem}

\aliascntresetthe{assumption}

\newaliascnt{corollary}{theorem}

\aliascntresetthe{corollary}

\newaliascnt{definition}{theorem}

\aliascntresetthe{definition}

\newaliascnt{example}{theorem}

\aliascntresetthe{example}

\newaliascnt{remark}{theorem}
\newtheorem{remark}[remark]{Remark}
\aliascntresetthe{remark}

\newaliascnt{hypothesis}{theorem}

\aliascntresetthe{hypothesis}

\newaliascnt{property}{theorem}

\aliascntresetthe{property}

\let\originalleft\left
\let\originalright\right
\renewcommand{\left}{\mathopen{}\mathclose\bgroup\originalleft}
\renewcommand{\right}{\aftergroup\egroup\originalright}

\def\w{\textbf{W}^{\vareXilon}_{{\theta}^{\vareXilon}}}

\def\L{\mathrm{L}}

\def\I{\mathrm{I}}

\def\C{\mathrm{C}}

\def\x{\mathbf{x}}

\def\w{\mathbf{w}}
\def\W{\mathrm{W}}

\def\no{\nonumber}

\def\P{\mathbb{P}}
\def\u{\mathbf{u}}
\def\H{\mathrm{H}}
\def\n{\mathbf{n}}

\mathtoolsset{showonlyrefs} 
\def\bl{\textcolor{black}}

\newcommand{\Addresses}{{
		\footnote{
			\noindent \textsuperscript{1}Department of Mathematics, Indian Institute of Technology Roorkee (IITR), Roorkee, Uttarakhand 247667, India.\par\nopagebreak
			\noindent  \textit{e-mails:} Arbaz Khan:  \texttt{arbaz@ma.iitr.ac.in,} Sumit Mahajan: \texttt{sumit{\_}m@ma.iitr.ac.in, sumit.mahajan@inria.fr.}
			
			\noindent \textsuperscript{*}Corresponding author.\\
			\textit{Keywords: Burgers-Huxley equation, A posteriori error analysis, Discontinous Galerkin finite element method, Weakly singular kernel, $\L^2$ error estimates.} 
			Mathematics Subject Classification (2020):  65J15, 65N30, 65N50.
}}}
\begin{document}
	\title[A posteriori error estimates for the GBHE with memory]{A posteriori error estimates for the Generalized Burgers-Huxley equation with weakly singular kernels}
	\title[A posteriori error estimates for the GBHE with memory]{A posteriori error estimates for the Generalized Burgers-Huxley equation with weakly singular kernels\Addresses}
	\author[S. Mahajan and A. Khan ]{Sumit Mahajan\textsuperscript{1} and Arbaz Khan\textsuperscript{1*}}
	\begin{abstract}
			This paper explores the residual based a posteriori error estimations for the generalized Burgers-Huxley equation (GBHE) featuring weakly singular kernels. Initially, we present a reliable and efficient error estimator for both the stationary GBHE and the semi-discrete GBHE with memory, utilizing the discontinuous Galerkin finite element method (DGFEM) in spatial dimensions. Additionally, employing backward Euler and Crank–Nicolson discretization in the temporal domain and DGFEM in spatial dimensions, we introduce an estimator for the fully discrete GBHE, taking into account the influence of past history. The paper also establishes optimal $\L^2$ error estimates for both the stationary GBHE and GBHE. Ultimately, we validate the effectiveness of the proposed error estimator through numerical results, demonstrating its efficacy in an adaptive refinement strategy.	
	\end{abstract}
	\maketitle
\section{Introduction}\setcounter{equation}{0} 
\bl{Nonlinear partial differential equations (PDEs) arise in modeling complex physical processes, introducing significant challenges in deriving accurate numerical solutions. Finite element methods (FEM) are widely used to approximate such problems, with the accuracy of these approximations measured through error analysis. While a priori error estimates provide theoretical bounds based on solution regularity, a posteriori error estimates are indispensable for adaptive strategies, offering computable and localized error indicators to enhance solution accuracy. This study focuses on a posteriori error analysis for time-dependent nonlinear integro-differential equations, addressing the challenges introduced by nonlinearity and providing guidance for effective mesh refinement in practical computations.}

The generalized Burgers–Huxley equation (GBHE) stands out as a significant partial differential equation, incorporating advection, diffusion, and a nonlinear reaction term. Originally devised to model nerve pulse propagation in nerve fibers and wall motion in liquid crystals \cite{XYW}, the GBHE has found applications in diverse fields such as traffic flow problems \cite{NAG}, nuclear waste disposal \cite{KC}, fish population mobility \cite{HRT}, and the movement of domain walls in ferroelectric materials under electric fields \cite{YK}. While these applications often adopt the GBHE or its variants in a phenomenological sense, the equation remains a valuable test model for investigating nonlinear advection–reaction–diffusion systems and for evaluating advanced numerical methods.

In the realm of mathematical modeling for physical systems, the conventional approach focuses on representing dynamics at a specific moment. However, this method often overlooks the importance of past influences, potentially introducing inaccuracies. This gap becomes particularly noticeable in precision-dependent domains like nuclear reactor dynamics \cite{KCP}, heat transfer, and thermoelasticity.
Our study delves into the generalized Burgers-Huxley equation (GBHE) with a consideration for memory effects. Through the introduction of a weakly singular kernel, our approach transcends traditional partial differential equations, embracing the realm of partial integro-differential equations. This nuanced perspective enables us to capture the essential impact of historical influences on system behavior, addressing a crucial aspect often omitted in conventional analyses.
\subsection{Model problem}
Consider the domain $\Omega_T = \Omega \times [0,T]$ where $\Omega\subset \mathbb{R}^d \ (d = 2, 3)$ is a convex domain with Lipschitz boundary $\partial\Omega$ and $T$ is the final time. The GBHE with a weakly singular kernel is defined as:
\begin{equation}\label{3.GBHE}
	\begin{aligned}
		\frac{\partial u(\x,t)}{\partial t}&+\alpha u(\x,t)^{\delta}\sum_{i=1}^d\frac{\partial u(\x,t)}{\partial x_i}-\nu\Delta u(\x,t)-\eta\int_{0}^{t} K(t-\tau)\Delta u(\x,\tau) \mathrm{~d}\tau\\&=\beta u(\x,t)(1-u(\x,t)^{\delta})(u(\x,t)^{\delta}-\gamma)+f(\x,t),  \ (\x,t)\in\Omega\times(0,T],\\ u(\x,t)&=0, \ \x \in \partial\Omega ,\ t\in(0,T],\\
		u(\x,0)&=u_0(\x), \ \x\in{\Omega}.
	\end{aligned}
\end{equation}
\bl{Here, $u(\x,t)$ denotes the scalar-valued solution of the GBHE with a weakly singular kernel, where $\x = (x_1,\cdots,x_d)$.} The function $f(\cdot,\cdot)$  represents the external forcing and $K(\cdot)$, regarded as the weakly singular kernel with coefficient  $\eta\geq 0,$ amounts to the relaxation time. The parameters  $\alpha,\nu , \beta\geq 0$ represent the advection, diffusion and  reaction term, respectively, \bl{$\delta$ is a positive integer}
and $\gamma\in(0,1)$. One of the primary examples of a weakly singular kernel is \bl{\begin{align}\label{3.wk1}K(t) =\frac{1}{\Gamma( \tau)}\frac{1}{t^{1-\tau}}, \quad 0<\tau<1,\end{align}}
\bl{where  $\Gamma(\tau)=\int_{0}^{\infty}t^{\tau-1}e^{-t}\mathrm{d}t$} is the \emph{Euler Gamma function}. For different choices of the parameters, this equation represents well-known mathematical models, the Burgers equation and the Huxley equation.
\subsection{Literature survey}
\bl{The numerical approximation of the generalized Burgers-Huxley equation (GBHE) has been widely studied in 1D using various numerical techniques \cite{HNS, KHA, KSM, SGZ}. A priori error estimates have been rigorously analyzed through conforming finite element methods (FEM) \cite{MKH}. For $\alpha = 1$ and $\eta = \beta = 0$, the GBHE simplifies to the nonlinear Burgers’ equation, a widely studied model that describes numerous physical phenomena such as shock flows, gas dynamics, nonlinear wave propagation, turbulence, traffic dynamics, convection-diffusion processes, boundary layer behavior, and acoustic attenuation \cite{Rad,ZLi}. Variants of the Burgers' equation, including the Burgers-Fisher \cite{YJi}, Fisher–Kolmogorov \cite{CBD}, and KdV-Burgers \cite{KRS} equations, extend its scope to diverse scientific applications. Significant progress has been made in analyzing the 2D Burgers’ equation, with studies addressing both analytical and numerical aspects \cite{Elt, KRS, ZYa, HAb, LWZ, SSA}. Additionally, coupled (system) versions of the Burgers' equation have been explored in the literature \cite{Fle, Bah, KRS, ZOW, LCS}, offering insights into more complex interactions and suggesting a promising direction for future research. Despite these advances, studies on the GBHE in higher dimensions remain limited \cite{EMR, HSI, LWH}. Notably, the existence of an exact traveling wave solution for the 2D GBHE, as reported in \cite{EMR}, highlights its importance as a benchmark model for validating numerical methods.}

\bl{The discussions on conforming, non-conforming, and discontinuous FEM for the stationary GBHE in $d$-dimensions $(d=2,3)$ can be found in \cite{KMR}.
	Recent contributions have shifted their focus towards establishing well-posedness and a priori error estimates for the GBHE with a weakly singular kernel \eqref{3.GBHE}. Initially explored through conforming FEM \cite{GBHE}, these analyses have been extended to incorporate the Crouzeix-Raviart element and the discontinuous Galerkin (DG) element in spatial discretization, coupled with backward Euler for time discretization \cite{GBHE2}. The literature also encompasses additional investigations regarding a priori error estimates for analogous models, as evidenced by works such as \cite{ZNy,MTW,CTs, LTW,MMh,MBM} and references therein, offering a comprehensive exploration of the existing literature on this subject.}

In recent decades, a posteriori error estimation has emerged as a crucial component in crafting efficient adaptive meshing procedures to strategically control and minimize errors in numerical simulations. The literature dedicated to a posteriori error analysis for both elliptic and parabolic problems has attained a high level of maturity, as documented by numerous works over this period \cite{KJC,PMa,FNJ,MNo,Rcal,NSo,LMC,GPa,EVM,EOJ,SNP,CSG,HAR}. This list, while extensive, is not exhaustive, and for a more detailed discussion, readers are directed to \cite{AJT} and \cite{VERB}.

However, the field of hyperbolic and integro-differential equations is still witnessing ongoing advancements. The GBHE with memory serves as a bridge between hyperbolic and parabolic equations. Specifically, differentiation leads to a hyperbolic equation in the presence of a smooth kernel, such as $K(t)=1$ or setting $\tau = 1$ in \eqref{3.wk1}. Conversely, singular kernels (taking $\tau\rightarrow 0$ in \eqref{3.wk1}) yield equations of the parabolic type. A significant focus in the literature has been on a posteriori error estimates for linear parabolic integro-differential equations \cite{WLD, GMMRIMA, GMMRMC, GMMRJSc, JCA}, with various authors exploring this area, presenting error estimates under different norms and employing diverse methodologies.

The a posteriori error analysis presented in this work adopts a residual-type approach, drawing on abstract results related to spatial estimators. Specifically, these concepts are applied to discontinuous Galerkin schemes when addressing parabolic problems \cite{GOC}. We decompose the discrete solution into conforming and non-conforming contributions, incorporating a reconstruction technique, as done in \bl{\cite{BKM}}. Our work is organized into three main parts: firstly, we present the error estimator for the stationary GBHE. In the subsequent section, we propose an estimator for the semi-discrete method for GBHE with a weakly singular kernel using the discontinuous Galerkin finite element method (DGFEM) in space. Finally, we discuss the a posteriori error estimation for the time discretization using backward Euler (BE) and Crank-Nicolson (CN) methods with DGFEM in space.

This thorough exploration of a posteriori error estimation for GBHE, with a specific focus on weakly singular kernels, contributes to the evolving landscape of research in the field, offering valuable insights and methodologies for future investigations.
\subsection{Main Contribution}
As far as we are aware, this study marks the pioneering effort in presenting posteriori error assessments for nonlinear partial differential equations (PDEs) featuring a weakly singular kernel. Moreover, there is a notable absence of literature discussing posteriori error estimates for both stationary and non-stationary GBHE. Furthermore, we have successfully demonstrated the optimal $\L^2$ error estimate for both stationary and non-stationary GBHE. The primary advancements of this investigation include:
\begin{itemize}
	\item \textbf{$\L^2$ error estimates: } 
	This article expands upon the error estimates in the ${\L}^2$ norm for SGBHE and GBHE with memory. This extension builds on existing results for error estimates in the energy norm, as established in \cite{KMR} and \cite{GBHE2}. Through the utilization of the duality argument and the $A_h$ projection defined in \eqref{3.Aproj}, we present rigorous proofs for these ${\L}^2$ error estimates. Significantly, our analysis reveals the optimality of these error estimates for both stationary and non-stationary models.
	\item \textbf{A posteriori error estimates: }
	\begin{itemize}
		\item \textbf{Stationary Generalized Burgers-Huxley equation:} We introduce an error-based estimator, demonstrating its efficiency and reliability. 
		The Burgers equation serves as a foundational model for equations akin to the Navier-Stokes equation, thereby rendering our proposed estimator versatile and applicable to a wide range of related equations.
		\item \textbf{Semi-discrete scheme:} 
		When dealing with the GBHE featuring a weakly singular kernel, this research introduces a residual-based error estimator tailored for a semi-discrete scheme employing DGFEM. This unique estimator incorporates historical data, accommodating the impact of the memory term, and is substantiated as a reliable tool in the analysis.
		\item \textbf{Fully discrete scheme:} We explore posterior error estimates using backward Euler and Crank-Nicolson schemes in the fully discrete case. Our proposed estimator, applicable to both schemes, captures the influence of past history and accommodates mesh changes at different times, as defined in equation \eqref{3.timeeff}.
		\item \textbf{Numerical Validation:} To substantiate our theoretical discoveries through numerical validation, we present data demonstrating the efficacy of the suggested indicator in establishing a dependable limit for error during uniform refinement. An adaptive strategy has been employed to identify singularities in the solution, thereby reinforcing the practical viability of our approach.
	\end{itemize}
\end{itemize}
\subsection{Outline}
The paper is organized as follows:
Section \ref{sec2} establishes preliminary notations and introduces the weak formulation for the model problem \eqref{3.GBHE}, laying the foundation for our analysis. We also define the necessary projections to support the verification of $\L^2$ error estimates.
\bl{In Section \ref{sec3}, the focus is on the stationary case (SGBHE). Initially, we establish optimal error estimates in the $\L^2$-norm and subsequently present a posteriori error analysis to showcase the reliability and efficiency of the proposed estimator for the spatial discretization of SGBHE using DGFEM.}
Transitioning to the semi-discrete scheme, Section \ref{sec4} employs DGFEM in space to derive $\L^2$-error estimates and a posteriori error estimates.
The fully discrete scheme is introduced in Section \ref{sec5}, utilizing backward Euler and Crank-Nicolson methods in time coupled with DGFEM in space. The reliability of the proposed estimators is verified for both schemes.
Section \ref{sec6} delves into numerical computations, validating the reliability of the proposed estimators for the total error under both uniform and adaptive refinement for various types of domains and solutions. This validation encompasses three distinct cases: SGBHE, GBHE with memory using BE in time, and GBHE with memory using CN in time.

\section{Problem Formulation}\setcounter{equation}{0} \label{sec2}
\subsection{Notations} Let $W^{m,p}(\Omega)$ represents the standard Sobolev space with the corresponding norm $\|\cdot\|_{\W^{m,p}}$ (The domain specification is omitted whenever clear from the context). In particular for $m= 0,$ we have the Lebesgue spaces $\L^p(\Omega)$ with norm $\|\cdot\|_{\L^p}$ and for $p=2,$ we have $W^{m,2}(\Omega) = \H^m(\Omega)$ which denotes the Hilbert space with norm represented by $\|\cdot\|_{\H^k}$. The dual space of $\L^p(\Omega)$ is given by the $\L^{\frac{p}{p-1}}(\Omega)$. The space $\H_0^1(\Omega)$ represents the closure of $C_0^{\infty}(\Omega)$ (set of infinitely differentiable function with compact support) with respect to $\H^{1}(\Omega)$ norm and its dual space is given by $\H^{-1}(\Omega)$. In other words $\H_0^1(\Omega) := \{u \in \H^1(\Omega) : u|_{\partial\Omega} = 0$ a.e.\}.
\subsection{Weak formulation}
To derive the weak formulation, we multiply equation \eqref{3.GBHE} by a test function $v$ and employ integration by parts, incorporating boundary conditions. This leads to the following weak formulation: Given initial data $u_0 \in \L^2(\Omega)$ and external forcing $f\in \L^2(0,T;\L^{2}(\Omega))$, seek $u\in \mathrm{L}^{\infty}(0,T; \L^{2}(\Omega))\cap\L^{2}(0,T; \H_0^1(\Omega))\cap\L^{2(\delta+1)}(0,T; \L^{2(\delta+1)}(\Omega)),$ with $\partial_tu\in\L^{\frac{2(\delta+1)}{2\delta+1}}(0,T;\H^{-1}(\Omega)+\L^{\frac{2(\delta+1)}{2\delta+1}}(\Omega)),$ such that
\begin{equation}\label{3.weakGBHE}
	\left\{
	\begin{aligned}
		\langle\partial_tu(t),v\rangle+ \mathcal{A}(u(t),v)&+\bl{\eta(( K*\nabla u)(t),\nabla v)}=\langle f(t),v\rangle, \\
		(u(0),v)&=(u_0,v), \quad  \forall~~ v\in \H_0^1(\Omega)\cap\L^{2(\delta+1)}(\Omega),
	\end{aligned}
	\right.
\end{equation}
where \begin{align}\label{3.weakcont}
	&\mathcal{A}(u,v) = \nu (\nabla u,\nabla v)+\alpha b(u,u,v)-\beta\langle c(u),v\rangle,
	\\&b(u,v,w)=\int_{\Omega}u^{\delta}\sum_{i=1}^{d}\frac{\partial v}{\partial x_i}w\mathrm{~d}x,\quad 
	c(u)=u(1-u^{\delta})(u^{\delta}-\gamma).
\end{align}
Throughout this work, the kernel $K(\cdot)$ is assumed to be \emph{weakly singular positive kernel}. This means that for any positive constant $T >0$, we have 
\bl{\begin{align}\label{pk}
		\int_0^T \int_0^t K(t-\tau)u(\tau)u(t) \mathrm{~d}\tau \mathrm{~d}t\geq 0,\  \ \forall ~~u\in \L^2(0, T). 
\end{align}}
The well-posedness of \eqref{3.weakGBHE} and a priori error estimates using both conforming and DG finite element methods have been previously addressed in \cite{GBHE} and \cite{GBHE2}, respectively. The subsequent regularity is established for the solution of the GBHE:
\begin{theorem}\label{2M.thm33}
	Let $\Omega\subset \mathbb{R}^d$, where $d = 2,3$, be either convex or a domain with $C^2$ boundary. \bl{For $ \delta \in \mathbb{N}$ for $d=2$, and $ \delta = 1, 2$ for $d=3$}, if $u_0\in\H^2(\Omega)\cap\H_0^1(\Omega)$, then the following properties hold:
	\begin{enumerate}[label=(\roman*)]\label{reg}
		\item If $f\in\H^1(0,T;\L^2(\Omega))$, then $\partial_tu\in\L^{\infty}(0,T;\L^2(\Omega))\cap\L^2(0,T;\H_0^1(\Omega))$.
		\item If $f\in\H^1(0,T;\L^2(\Omega))\cap \L^2(0,T;\H^1(\Omega))$, then $u \in \L^{\infty}(0,T;\H^2(\Omega))$.
	\end{enumerate}
\end{theorem}
\subsection{Domain Discretization}
The domain $\Omega$ is partitioned into shape-regular meshes, comprised of triangles or rectangles in 2D and tetrahedra in 3D, denoted by $\mathcal{T}_h$. The symbols $\mathcal{E}_h$, $\mathcal{E}^i_h$, and $\mathcal{E}^{\partial}_h$ represent the sets of all edges, interior edges, and boundary edges of the triangulation, respectively. The diameter of an element $K$ is denoted as $h_K$, and the length of an edge is represented by $h_E$. Within the context of a given $\mathcal{T}_h$, $C^{0}(\mathcal{T}_h)$ and $\H^s(\mathcal{T}_h)$ refer to the broken spaces associated with continuous and differentiable function spaces, respectively.

The shared edge between two mesh cells $K_\pm$ is denoted as $E=K_+\cap K_-\in\mathcal{E}^i_h$. Furthermore, the traces of functions $w\in C^0(\mathcal{T}_h)$ on edge $E$ of $K_\pm$ are designated as $w_{\pm}$, respectively. The average operator $\{\!\{\cdot\}\!\}$ and the jump operator $[\![\cdot]\!]$ on edge $E$ are defined as:
\begin{align*}
	\{\!\{w\}\!\}=\frac{1}{2}(w_+ + w_-) \quad  \text{and}	\quad[\![w]\!]={w_+\mathbf{{n}_+} + w_-\mathbf{{n}_-}},
\end{align*}
respectively. 
If ${w}\in C^1(\mathcal{T}_h),$ we define the jump of the normal derivative as $[\![\partial{w}/\partial\mathbf{{n}}]\!]=\nabla({w}_+ -{w}_-)\cdot\mathbf{{n}_+}, $ where $\mathbf{{n}\pm}$ denotes the unit outward normal vectors for the respective mesh cells $K_\pm$. In the case of $E\in K_+\cap\partial\Omega$, the jump is ${[\![{w}]\!]={w}_+\mathbf{{n}_+}}$ and the average is $\{\!\{w\}\!\}=w_+$. Let $u_e$ represent the exterior trace of the function $u$, and for the boundary edges, we set $u^e = 0.$ The local gradient on each $K\in\mathcal{T}_h$ is denoted by $\nabla_h$, with $(\nabla_h{w})|_K = \nabla({w}|_K).$ \bl{Throughout this work, the constant $C$ represents a generic constant independent of the mesh size $h$}. The discrete space for the DG formulation is defined as
\begin{align}\label{2.dgsubspace1}
	V_h = {V}^{DG}_h=\{{v}\in \L^2(\Omega): \forall\  K\in \mathcal{T}_h : {v}|_K \in \mathcal{P}_k(K)\},
\end{align}
where $\mathcal{P}_k(K)$ denotes the space of polynomials of degree $k$ on $K.$
The semi-discrete weak formulation for the \eqref{3.GBHE} corresponding to DGFEM is given as:  Find $u_h\in V_h$, for all $t\in(0,T)$ such that
\begin{align}\label{2.dgweakform}
	\nonumber(\partial_tu_h(t), \chi(t))+\mathcal{A}_{DG}(u_h(t),\chi(t))+\eta( K*a_{DG}( u_h(s),  \chi(t)))&=( f(t),\chi(t)),\\
	(u_h(0),\chi(t))&=(u_h^0,\chi(t)), 
\end{align}
$\forall$ $\chi \in V_h$, where
\begin{align}\label{3.ADGf} 
	\mathcal{A}_{DG}(u,v) = \nu a_{DG}(u,v) + \alpha b_{DG}(u,u,v)-\beta(c(u),v).
\end{align}
The operator $\mathcal{A}_{DG}$ consists of three terms, where the first term is the discrete diffusion term given as
\begin{align}\label{3.adg}
	\no a_{DG}({u},{v})&=(\nabla_h {u}, \nabla_h {v})-\sum_{E\in\mathcal{E}_h}\int_E {\{\!\{\nabla_h {u}\}\!\}}\!\cdot\![\![{v}]\!]\mathrm{~d}{s}\\&\quad-\sum_{E\in\mathcal{E}_h}\int_E {\{\!\{\nabla_h {v}\}\!\}}\!\cdot\![\![{u}]\!]\mathrm{~d}{s}
	+\sum_{E\in\mathcal{E}_h}\int_E\mathfrak{K}_h [\![{u}]\!]\!\cdot\![\![{v}]\!]\mathrm{~d}{s},
\end{align}
where the parameter \bl{$\mathfrak{K}_h=\frac{\mathfrak{K}}{h_E}$ is the penality term and $\mathfrak{K}$} is chosen sufficiently large to ensure the stability of the formulation (see, e.g.,  \cite{Arn}). The discrete advection term is given as 
\bl{\begin{align}\label{3.bdg}
		\no	b_{DG}(\mathbf{w},u_h,v)=\frac{1}{\delta+2}\Bigg(& \sum_{K\in \mathcal{T}_h} \int_{K} \mathbf{w}\cdot \nabla u_h v \mathrm{~d}x 
		+\sum_{K\in \mathcal{T}_h} \int_{\partial K} \hat{\mathbf{w}}_{u_h}^{up}{v} \mathrm{~d}s\\& -\sum_{K\in \mathcal{T}_h} \int_{K} \mathbf{w}\cdot \nabla v {u_h} \mathrm{~d}x 
		-\sum_{K\in \mathcal{T}_h} \int_{\partial K} \hat{\mathbf{w}}_{v}^{up}{u_h} \mathrm{~d}s\Bigg),
\end{align}}
where the upwind flux is given as 
\begin{align*}
	\hat{\mathbf{w}}_{u}^{up}=\frac{1}{2}\left[\mathbf{w}\cdot\mathbf{n}_K -|\mathbf{w}\cdot \mathbf{n}_K|\right](u^e\!-\!u),
\end{align*}
\bl{with $\mathbf{w} = (w, w)^T$, in particular, $\mathbf{w} = (u^\delta, u^\delta)^T$.} The reaction term is defined similarly to the continuous case and is expressed as $c(u)=u(1-u^{\delta})(u^{\delta}-\gamma)$. 
The discrete (DG) norm in this setting is given as
\begin{align}\label{3.dgnorm}
	|\!|\!| v|\!|\!|^2:= \sum_{K\in\mathcal{T}_h}\|\nabla_h v\|_{\L^2(K)}^2 + \sum_{E\in\mathcal{E}_h}\mathfrak{K}_h\|[\![v]\!]\|_{\L^2(E)}^2.
\end{align}
\begin{remark}
	The nonlinear operators, satisfies the following properties:
	\begin{enumerate}
		\item \bl{$b_{DG}(\u,u,u) = 0$},
		\item $b_{DG}(\w,u,v) = -b_{DG}(\w,v,u)$. 
		\item $(c(u),u) = (1+\gamma)(u^{\delta+1},u) -\gamma \|u\|_{\L^2}^2 - \|u\|^{2(\delta+1)}_{\L^{2(\delta+1)}}. $
	\end{enumerate}
\end{remark}

In the upcoming sections, we systematically present $\L^2$ and a posteriori error estimates, ordered by increasing complexity of the problem, facilitating a more accessible understanding.
\section{Stationary generalised Burgers'-Huxley equation}\setcounter{equation}{0}\label{sec3}
The weak formulation of the stationary generalised Burgers'-Huxley equation (SGBHE) is expressed as follows: Find $u\in \H_0^1(\Omega),$ such that:
\begin{align}\label{3.SGBHEweak}
	\mathcal{A} (u,v)= \langle f,v\rangle, \qquad \quad \forall~~ v\in \H_0^1(\Omega).
\end{align}
where the \bl{{semilinear form} $\mathcal{A}(\cdot,\cdot)$} is as defined in \eqref{3.weakcont}. Using the DGFEM, the discrete counterpart of the SGBHE is given as: Find $ u_h \in V_h$ such that 
\begin{align}\label{3.SGBHE}
	\mathcal{A}_{DG} (u_h,\chi)= \langle f_h,\chi\rangle, \qquad \forall~ \chi\in V_h,
\end{align}
where $\mathcal{A}_{DG}$ is as defined in \eqref{3.ADGf} and \bl{$f_h$ denote the piecewise polynomial approximations of $f$ in $V_h$}. 
A thorough examination of the existence, uniqueness, and regularity outcomes for both weak formulations, as specified in \eqref{3.SGBHEweak} and \eqref{3.SGBHE}, along with a priori error estimates in the energy norm, is exhaustively discussed in \cite{KMR}.
\subsection{Technical results}
Initially, we establish the necessary projections essential for validating the $\L^2$ error estimates. Subsequently, we present technical Lemmas that play a pivotal role in the subsequent analysis.

For $v\in \L^2(\Omega)$, we denote the $\L^2(\Omega)$ projection  of $v$ on $V_{h}$:
\begin{align}
	\Pi_h v \in V_{h},\quad (\Pi_hv - v, \phi_h) = 0 \quad \forall~~ \phi_h\in V_{h},
\end{align}
and for $K\in \mathcal{T}_h$, the function $\Pi_h v|_K$ is the $\L^2(K)-$projection of $v|_K$ on $P^k(K)$. Let $p\in [1,k]$,  then the following estimates 
\begin{align}
	\|\Pi_hv-v\|_{\L^2(K)} &\leq C h^{p+1}|v|_{\H^{p+1}(K)},\label{3.aproj1}\\
	|\Pi_hv-v|_{\H^1(K)} &\leq C h^{p}|v|_{\H^{p+1}(K)},\label{3.aproj2}
\end{align}
for all $v\in \H^{p+1}(K)$, $K\in\mathcal{T}_h.$ 
Let us define a function \bl{$ u_h^*$ as the $`A_h-$projection' of $u$}, as defined in  \cite{DFK}, on $V_{h}$ satisfying
\begin{align}\label{3.Aproj}
	\bl{	u_h^* \in V_{h}\cap \L^{\infty}(\Omega), \qquad a_{DG}(u_h^*,\phi_h ) =a_{DG}(u, \phi_h) \quad \forall ~~ \phi_h \in V_{h}\cap \L^{\infty}(\Omega),}
\end{align}
\bl{where $a_{DG}(\cdot,\cdot)$ is defined in \eqref{3.adg}.} 
Let us recall that the operator $a_{DG}(\cdot,\cdot)$ is coercive, such that 
\begin{align}\label{3.acor}
	a_{DG}(\phi_h,\phi_h) \geq C_{c}	|\!|\!| \phi_h|\!|\!| ^2, 
\end{align}
where $C_c$ is the coercivity constant. 
In the subsequent analysis, we decompose the error as $e = u - u_h = u - u_h^* + u_h^* - u_h := \omega + \xi$. Initially, we focus on bounding the functions $\omega = u - u_h^*$ in the $|\!|\!| \cdot|\!|\!| $ and $\L^2(\Omega)$ norms.
\begin{lemma}\label{3.Lemma3.1}
	For $u\in \H^{p+1}(\Omega)$ with $1\leq p \leq k$, there exists a constant $C>0$ such that 
	\begin{align}
		|\!|\!| \omega	|\!|\!|  &\leq Ch^p|u|_{\H^{p+1}}.
	\end{align}
	\begin{proof}
		Let us set $\hat{u} = \Pi_h u$ and $\phi_h = \hat{u} - u_h^*$. By the coercivity of $a_{DG}(\cdot,\cdot),$ \eqref{3.Aproj} and further using \eqref{3.aproj2}, we have 
		\begin{align}
			C_c	|\!|\!| \phi_h	|\!|\!| ^2 \leq a_{DG}(\phi_h,\phi_h) &= a_{DG}(\phi_h,\phi_h) + a_{DG}(u_h^* - u,\phi_h) \\&= a_{DG}(\hat{u} - u,\phi_h)\\&\leq C 	|\!|\!| \hat{u} - u	|\!|\!| 	|\!|\!| \phi_h	|\!|\!| \\&\leq C h^p |u|_{\H^{p+1}}	|\!|\!| \phi_h	|\!|\!| \label{3.55}
		\end{align}
		Now, using the triangle inequality, \eqref{3.55} and the estimate \eqref{3.aproj2}, we attain
		\begin{align}
			|\!|\!| \omega	|\!|\!|  = 	|\!|\!| u - u_h^*	|\!|\!|  \leq	|\!|\!| u-\hat{u}	|\!|\!| +	|\!|\!| \hat{u} - u_h^*	|\!|\!|   \leq Ch^p|u|_{\H^{p+1}}.
		\end{align}
	\end{proof}
\end{lemma}
\begin{lemma}\label{3.Lemma 3.5}
	Under analogous assumptions to those in Lemma \ref{3.Lemma3.1}, there exists a constant $C>0$ such that
	\begin{align}
		\|\omega\|_{\L^2}&\leq Ch^{p+1}|u|_{\H^{p+1}}.
	\end{align}
	\begin{proof}
		The proof closely follows the approach outlined in \cite[Lemma 4.2]{DFK}.
	\end{proof}
\end{lemma}
The subsequent lemma plays a pivotal role in ensuring the reliability of the estimator, particularly for the estimation of nonlinear terms.
\begin{lemma}\label{2.crnclem111}
	There holds:
	\begin{align*}
		-\alpha[b_{DG}(u_1,u_1,w)-b_{DG}(u_2,u_2,w)]&\le  \frac{\nu}{2}|\!|\!| w|\!|\!|^2+C(\alpha,\nu)\left(\|u_1\|^{\frac{8\delta}{4-d}}_{\L^{4\delta}}+\|u_2\|^{\frac{8\delta}{4-d}}_{\L^{4\delta}}\right)\|w\|_{\L^2}^2,\\
		\text{or ~~}
		-\alpha[b_{DG}(u_1,u_1,w)-b_{DG}(u_2,u_2,w)]&\le  \frac{\nu}{2}|\!|\!| w|\!|\!|^2+\frac{2^{2\delta}\alpha^2}{4\nu(\delta+1)^2}\Big(\|u_1^{\delta}w\|_{\L^2}  +\|u_2^{\delta}w\|_{\L^2}\Big),
	\end{align*}
	where $u_1,u_2\in V_{h}$, $w=u_1-u_2$, \bl{$b_{DG}$ is defined in \eqref{3.bdg},} and $C(\alpha, \nu) = \left(\frac{4+d}{4\nu}\right)^{\frac{4+d}{4-d}}\left(\frac{4-d}{8}\right)(\frac{2^{\delta-1}C\alpha}{(\delta+2)(\delta+1)})^{\frac{4-d}{8}}$ 
	is a positive constant depending on the parameters.
\end{lemma}
\begin{proof}
	The first estimate has been established in \cite[Lemma 2.4]{GBHE2}. We slightly modify the first estimates by employing Taylor’s formula, Cauchy-Schwarz, Hölder's, and Young’s inequality, resulting in:
	\begin{align}\label{2.30}
		-\alpha[b_{DG}(u_1,u_1,w)-b_{DG}(u_2,u_2,w)]
		&\leq \frac{2^{\delta-1}\alpha C}{\delta+1}\sum_{K\in \mathcal{T}_h}(\|u_1^{\delta}w\|_{\L^{2}}+\|u_2^{\delta}w\|_{\L^{2}})\|\nabla_h w\|_{\L^2(K)}\nonumber\\
		&\leq  \frac{\nu}{2}|\!|\!| w|\!|\!|^2+\frac{2^{2\delta}\alpha^2}{4\nu(\delta+1)^2}\|u_1^{\delta}w\|_{\L^2}  +\frac{2^{2\delta}\alpha^2}{4\nu(\delta+1)^2}\|u_2^{\delta}w\|_{\L^2}, 
	\end{align}
	yields the desired result.
\end{proof}

An application of the above Lemma \ref{2.crnclem111} leads to the following outcome:
\begin{lemma}\label{lemma3.2}
	For $u_1,u_2\in V_{h}$, $w=u_1-u_2$
	\bl{\begin{align}
			\nonumber{\mathcal{A}}_{DG}(u_1,w)-{\mathcal{A}}_{DG}(u_2,w)&\ge  \frac{\nu}{2}|\!|\!| w|\!|\!|^2 +\frac{\beta}{4}(\|{u}^{\delta}_1w\|_{\L^2}^2+\|u_2^{\delta}w\|_{\L^2}^2)\nonumber\\
			&\quad+\left(\beta\gamma-C(\beta,\gamma,\delta) - C(\alpha,\nu)\Big(\|u_1\|^{\frac{8\delta}{4-d}}_{\L^{4\delta}}+\|u_2\|^{\frac{8\delta}{4-d}}_{\L^{4\delta}}\Big)\right)\|w\|_{\L^2}^2,\label{3.renon1}\\\nonumber
			\text{or~~} 
			{\mathcal{A}}_{DG}(u_1,w) - {\mathcal{A}}_{DG}(u_2,w)&\geq   \frac{\nu}{2}|\!|\!| w|\!|\!|^2+ \bigg(\frac{\beta}{4}-\frac{2^{2\delta}\alpha^2}{4\nu(\delta+1)^2}\bigg)(\|u_1^{\delta}w\|_{\L^2}^2+\|u_2^{\delta}w\|_{\L^2}^2)\\&\quad + \bigg(\beta\gamma -\frac{\beta}{2}2^{2\delta}(1+\gamma)^2(\delta+1)^2\bigg) \|w\|_{\L^2}^2,\label{3.renon2}
	\end{align}}
	where  $C (\beta,\gamma,\delta)= \frac{\beta}{2}2^{2\delta}(1+\gamma)^2(\delta+1)^2.$  
\end{lemma}
\begin{proof}
	The proof closely parallels the approach outlined in \cite[Lemma 2.4]{GBHE2}.
\end{proof}
\bl{
	\begin{remark}\label{rhs}
		\begin{enumerate}
			\item Using the Sobolev embedding \cite[cf. Sec 4]{BOr} and the stability estimate \cite[Lemma 2.12]{GBHE2}, we obtain, for $1 \leq q < \infty$ in $2D$ and for $1 \leq q \leq 6$ in $3D$:
			\begin{align}\label{1qinf}
				\|u_h\|_{\L^q} \leq C |\!|\!| u_h|\!|\!| \leq C \|f\|_{\L^2}, \quad \|u_h^*\|_{\L^q} \leq C|\!|\!| u_h^*|\!|\!| \leq C|\!|\!| u|\!|\!| \leq C \|f\|_{\L^2}.
			\end{align}
			\item For $q = \infty$, using the triangle inequality, an approach similar to \cite[Lemma 14]{VAC}, the interpolation result \cite[Lemma 4.4.1]{BSc}, and the Sobolev embedding, we obtain:
			\begin{align}\label{qinf}
				\no\|u_h^*\|_{\L^{\infty}} \leq \|u_h^* -\Pi_h u \|_{\L^{\infty}} + \|\Pi_hu - u\|_{\L^{\infty}} + \|u\|_{\L^{\infty}} &\leq C ( h^{k+1-\frac{d}{2}}\|u\|_{\H^{k+1}} + \|u\|_{\H^2})\\&\leq C( h^{k+1-\frac{d}{2}}\|u\|_{\H^{k+1}} + \|f\|_{\L^2}),
			\end{align}
			where $C$ is a constant independent of $h$.
			\item A similar estimate to $\eqref{1qinf}$ can be extended for $1 \leq q \leq \infty$ in case of $3D$ by employing an approach analogous to the one used for $\eqref{qinf}$ in a bounded domain. This extension follows from the embedding $\L^\infty(\Omega) \subset \L^p(\Omega)$ for $1 \leq p < \infty$.
		\end{enumerate}
	\end{remark}
}
\subsection{$\L^2$ Error estimates for SGBHE}
In establishing optimal $\L^2$ error estimates for the SGBHE, the pivotal role is played by the following lemma.
\bl{\begin{lemma}\label{3.A1}
		Let $u$ be the solution of \eqref{3.SGBHEweak} and $u_h\in V_h$ be the discrete solution of \eqref{3.SGBHE}, and set $\xi = u_h^*-u_h$, where $u_h^*$ is the $A_h$ projection defined in \eqref{3.Aproj}. Then the following estimate holds
		\bl{\begin{align}
				\left(\frac{\nu}{2C_{\Omega}}+\beta\gamma-C(\beta,\alpha,\delta) - C(\alpha,\nu)\|f\|^{\frac{8\delta}{4-d}}_{\L^{2}}\right)\|\xi\|_{\L^2}^2\leq C(\alpha,\delta,\nu,\gamma,\|f\|_{\L^{2}})\|\omega\|_{\L^2}^2 ,
		\end{align}}
		where, $C(\alpha,\delta,\nu,\gamma,\|f\|_{\L^{2}}) = \frac{2^{2\delta}\alpha^2\delta^2}{\nu(\delta+2)^2}\|{f}\|_{\L^{2}}^{2\delta}+ \frac{1}{\nu}2^{2\delta}\beta^2(1+\gamma)^2(\delta+1)^2\|f\|^{2\delta}_{\L^{2}}+  2\beta\gamma +\frac{2^{4\delta}C_{\Omega}^2(2\delta+1)^2\beta^2}{\nu}\|f\|_{\L^{2}}^{4\delta},$ and $C_{\Omega}$ is the Poincar\'e constant. 
\end{lemma}}
\begin{proof}
	Subtract \eqref{3.SGBHEweak} from \eqref{3.SGBHE}, we have 
	\begin{align*}
		&\mathcal{A}_{DG} (u,\chi) - 	\mathcal{A}_{DG} (u_h,\chi)=\nu a_{DG}(u-u_h,\chi) + \alpha (b_{DG}(u,u,\chi) - b_{DG}(u_h,u_h,\chi))-\beta(c(u) - c(u_h),\chi)= 0
	\end{align*} 
	By rearranging the preceding equation and utilizing \eqref{3.Aproj}, while setting $\chi= \xi = u_h-u_h^*$, we obtain:
	\begin{align}
		\nu a_{DG}(\xi,\xi) + &\alpha (b_{DG}(u_h^*,u_h^*,\xi) - b_{DG}(u_h,u_h,\xi))-\beta(c(u_h^*)- c(u_h) ,\xi)\\&=\alpha (b_{DG}(u_h^*,u_h^*,\xi) - b_{DG}(u,u,\xi))-\beta(c(u_h^*)- c(u),\xi).
	\end{align} 
	Using the estimate \eqref{3.renon1} obtained in Lemma \ref{lemma3.2}, we achieve 
	\begin{align}\label{3.ap4}
		&\frac{\nu}{2}|\!|\!| \xi|\!|\!|^2 +\frac{\beta}{4}(\|{u}^{\delta}_h\xi\|_{\L^2}^2+\|(u_h^*)^{\delta}\xi\|_{\L^2}^2)\nonumber
		+\left(\beta\gamma-C(\beta,\alpha,\delta) - C(\alpha,\nu)\Big(\|u_h\|^{\frac{8\delta}{4-d}}_{\L^{4\delta}}+\|u_h^*\|^{\frac{8\delta}{4-d}}_{\L^{4\delta}}\Big)\right)\|\xi\|_{\L^2}^2\\&\leq \alpha (b_{DG}(u_h^*,u_h^*,\xi) - b_{DG}(u,u,\xi))-\beta(c(u_h^*) - c(u),\xi).
	\end{align}
	Let us  now estimate $J_1 = b_{DG}(u_h^*,u_h^*,\xi) - b_{DG}(u,u,\xi)$, using Taylor's formula, Hölder's inequality, and the bound on the discrete upwind term \cite{TKM} as
	\begin{align*}
		\nonumber	|J_1|
		&= -\frac{2\alpha}{\delta+2}\sum_{K\in \mathcal{T}_h}\sum_{i=1}^d  \left(\int_K\left((u_h^*)^{\delta}\frac{\partial u_h^*}{\partial{x_i}}-u^{\delta}\frac{\partial u}{\partial{x_i}}\right) \xi \mathrm{~d}x- \int_K((u_h^*)^{\delta+1}-u^{\delta+1})\frac{\partial \xi}{\partial{x_i}} \mathrm{d}x\right) \\&\quad+ \sum_{K\in \mathcal{T}_h}\left( \int_{\partial K} \left(\hat{\mathbf{u^*}}_{h,u^*_h}^{up}-\hat{\mathbf{u}}_{u}^{up}\right){\xi} \mathrm{~d}s - \int_{\partial K} \hat{\mathbf{\xi}}_{\xi}^{up}(u_h^*-u) \mathrm{~d}s \right)
		\\\nonumber
		& \leq \frac{2C\alpha\delta}{(\delta+2)}\Big(\|u_h^*\|^{\delta}_{\L^{\infty}} + \|u\|^{\delta}_{\L^{\infty}} \Big)\|\omega\|_{\L^{2}}\|\nabla_h \xi\|_{\L^2(\mathcal{T}_h)}
		\\& \leq{\frac{2^{2\delta}C\alpha^2\delta^2}{\nu(\delta+2)^2}\left(\|{u_h^*}\|_{\L^{\infty}}^{2\delta}+\|{u}\|_{\L^{\infty}}^{2\delta}\right)\|\omega\|_{\L^2}^2+\frac{\nu}{8}|\!|\!| \xi|\!|\!|^2,}
	\end{align*}
	The term $J_2$ can be rewritten as 
	\begin{align}
		&J_2 = \beta\left[(u_h^*(1-(u_h^*)^{\delta})((u_h^*)^{\delta}-\gamma)-u(1-u^{\delta})(u^{\delta}-\gamma),\xi)\right] \\&=  2\beta(1+\gamma)({(u_h^*)}^{\delta+1}-u^{\delta+1},\xi) -2\beta\gamma(\omega,\xi) -2\beta({(u_h^*)}^{2\delta+1}-u^{2\delta+1},\xi).
	\end{align} 
	Let us rewrite,	$J_2= J_3 + J_4 + J_5,$ where 
	\begin{align*}
		J_3 = 2\beta(1+\gamma)((u_h^*)^{\delta+1}-u^{\delta+1},\xi),\quad  J_4 = -2\beta\gamma(\omega,\xi),\quad J_5 = -2\beta((u_h^*)^{2\delta+1}-u^{2\delta+1},\xi).
	\end{align*}
	The term $J_3$ can be initially estimated through Taylor's formula, followed by H\"older's inequality, and ultimately Young's inequalities, resulting in:
	\begin{align}
		|J_3|&=2\beta(1+\gamma)(\delta+1)((\theta u_h^*+(1-\theta)u)^{\delta}\omega,\xi) 
		\nonumber\\&\leq \frac{1}{\nu}2^{2\delta}\beta^2(1+\gamma)^2(\delta+1)^2\left(\|{u_h^*}\|^{2\delta}_{\L^{4\delta}}+\|{u}\|^{2\delta}_{\L^{4\delta}}\right)\|\omega\|^2_{\L^2} +\frac{\nu}{8}|\!|\!| \xi|\!|\!|^2.
	\end{align}
	To estimate $J_4$, we employ the Cauchy-Schwarz inequality, followed by Young's inequality, yielding:
	\begin{align}
		|J_4|&\leq 2\beta\gamma\|\omega\|_{\L^2}\|\xi\|_{\L^2}\leq 2\beta\gamma\|\omega\|_{\L^2}^2+\frac{\beta\gamma}{2}\|\xi\|_{\L^2}^2.
	\end{align}
	Utilizing Taylor's formula, H\"older's inequality, Young's inequality, and the discrete Sobolev embedding, we estimate $J_5$ as:
	\begin{align}\label{2.7p11}
		|J_5|&=-2(2\delta+1)\beta\left((\theta u_h^*+(1-\theta)u)^{2\delta}\omega,\xi\right)
		\nonumber\\&\leq  2^{2\delta}(2\delta+1)\beta\left(\|u_h^*\|_{\L^{\infty}}^{2\delta}+\|u\|_{\L^{\infty}}^{2\delta}\right)\|\omega\|_{\L^{2}}\|\xi\|_{\L^{2}}	
		\nonumber\\&\leq{\frac{\nu}{8}|\!|\!| \xi|\!|\!|^2
			+ \frac{2^{4\delta}C_{\Omega}^2(2\delta+1)^2\beta^2}{\nu}\left(\|u_h^*\|_{\L^{\infty}}^{4\delta}+\|u\|_{\L^{\infty}}^{4\delta}\right)\|\omega\|_{\L^2}^2.}
	\end{align}
	Consolidating the aforementioned estimates, we obtain:
	\begin{align}
		&\frac{\nu}{2}|\!|\!| \xi|\!|\!|^2_{DG} +\frac{\beta}{4}(\|{u}^{\delta}_h\xi\|_{\L^2}^2+\|(u^*)^{\delta}\xi\|_{\L^2}^2)\nonumber
		+
		\left(\beta\gamma-C(\beta,\alpha,\delta) - C(\alpha,\nu)\Big(\|u_h\|^{\frac{8\delta}{4-d}}_{\L^{4\delta}}+\|u_h^*\|^{\frac{8\delta}{4-d}}_{\L^{4\delta}}\Big)\right)\|\xi\|_{\L^2}^2\\&\leq\bigg( \frac{2^{2\delta}\alpha^2\delta^2}{\nu(\delta+2)^2}\left(\|{u_h^*}\|_{\L^{\infty}}^{2\delta}+\|{u}\|_{\L^{\infty}}^{2\delta}\right) + \frac{1}{\nu}2^{2\delta}\beta^2(1+\gamma)^2(\delta+1)^2\left(\|{u_h^*}\|^{2\delta}_{\L^{4\delta}}+\|{u}\|^{2\delta}_{\L^{4\delta}}\right) \\&\quad+  2\beta\gamma +\frac{2^{4\delta}C_{\Omega}^2(2\delta+1)^2\beta^2}{\nu}\left(\|u_h^*\|_{\L^{\infty}}^{4\delta}+\|u\|_{\L^{\infty}}^{4\delta}\right)\bigg)\|\omega\|_{\L^2}^2 .
	\end{align}
	Applying the Poincar\'{e} inequality to the first term and \bl{using Remark \ref{rhs}} yields the desired outcome.
\end{proof}
\bl{	\begin{remark}
		\begin{enumerate}
			\item 
			The estimates discussed in Lemma \ref{3.A1} have been obtained under the stringent condition on the regularity of the given data $f$. In the estimate \eqref{3.ap4}, instead of using \eqref{3.renon2}, applying \eqref{3.renon1} transforms the condition outlined in Lemma \ref{3.A1}, making it dependent solely on the parameters. This can be expressed as: For 
			\begin{align}\label{3.appara}
				\nu \geq \max\left\{ \frac{4^{\delta}\alpha^2}{\beta(\delta+1)^2}, C_{\Omega}\left(2C_1 + \beta\left[ 4^{\delta}(1+\gamma)^2(1+\delta)^2 - 2\gamma \right] \right) \right\},
			\end{align} we have 
			\begin{align}\label{5.2}
				\|u_h^*-u_h\|_{\L^2}^2\leq C\|\omega\|_{\L^2}^2,
			\end{align}
			which is analogous to the condition established for the well posedness and error estimates in energy norm  of the SGBHE as discussed in \cite[cf. (2.13)]{KMR}.
			\item In the case of time dependence, the $\L^2$ error estimates can be demonstrated without imposing any constraints on the parameters, as established in Theorem \ref{3.L2sd}.
		\end{enumerate}
\end{remark}}
\bl{Finally, the $\L^2$ error estimates for the SGBHE are presented in the following theorem.
	\begin{theorem}\label{3.apth1}
		Assuming \eqref{3.appara} holds, for $u\in \H^{k+1}(\Omega)$ with $k\geq 1$ the error estimates under $\L^2$-norm is given by
		\begin{align}
			\|u-u_h\|_{\L^2}\leq C h^{k+1}\|u\|_{\H^{k+1}},
		\end{align}
		where $C$ is a constant independent of $h$.
\end{theorem}}
\begin{proof}
	By employing the triangle inequality, along with \eqref{5.2} and Lemma \ref{3.Lemma 3.5} , the proof
	readily follows.
\end{proof}

\subsection{A posteriori error estimates}
The residuals linked to each element $K \in \mathcal{T}_h$ and edge $E \in \mathcal{E}_h$ in the a posteriori error estimates for the SGBHE provided in \eqref{3.SGBHE} are precisely specified as follows:
\begin{itemize}
	\item Element-wise residual, 
	$
	R_K := \{f_h+\nu\Delta u_h- \alpha u_h^{\delta}\sum\limits_{i=1}^d\frac{\partial u_h}{\partial x_i}+\beta u_h(1-u_h^{\delta})(u_h^{\delta}-\gamma)\}|_K,
	$
	\item Edge-wise residual, 
	$R_E:= \begin{cases}\frac{1}{2} [\![(\nu\nabla_h u_h)\cdot\n]\!] & \text { for } E \in \mathcal{E}^i_h, \\
		0 & \text { for } E \in \mathcal{E}_h^{\partial}.\end{cases}$
\end{itemize}
Moreover, we introduce the element-wise error  estimator as
$\zeta_K^2=\zeta_{R_K}^2+\zeta_{E_K}^2+\zeta_{J_K}^2$, where 
\begin{align}\label{3.sgbheest}
	\zeta_{R_K}^2:=h_K^2\|R_K\|_{\L^2(K)}^2,\quad
	\zeta_{E_K}^2:=\sum\limits_{E \in \mathcal{E}_h} h_E\left\|R_E\right\|_{\L^2(E)}^2,\quad
	\zeta_{J_K}^2:=\sum\limits_{E\in  \mathcal{E}_h} \mathfrak{K}_h\left\|[\![ u_h]\!]\right\|_{\L^2(E)}^2.
\end{align}
Note that if $V_h$ is the conforming space, then $\zeta_{J_K}^2 = 0$. Consequently, the global a posteriori error estimator and the data approximation term for the system \eqref{3.SGBHE} are expressed as follows:
\begin{align}\label{3.estimator}
	\zeta=\left(\sum_{K \in \mathcal{T}_h} \zeta_K^2\right)^{1 / 2}, \quad \mathcal{F} = \left(\sum_{K \in \mathcal{T}_h} h_K^2\|f-f_h\|^2_{\L^2(K)}\right)^{1 / 2}. 
\end{align}
\subsubsection{Reliability}
Let us rewrite the discrete operator $a_{DG}$ as, $a_{DG}= a^c_{DG} + N_{DG}$, where 
\begin{align}
	\label{3.adg1.1} a^c_{DG}({u},{v})=(\nabla_h {u}, \nabla_h {v})
	+\sum_{E\in\mathcal{E}_h}\int_E\mathfrak{K}_h [\![{u}]\!]\!\cdot\![\![{v}]\!]\mathrm{~d}{s},\\
	\label{3.adg1.2}N_{DG}(u,v) = -\bigg(\sum_{E\in\mathcal{E}_h}\int_E {\{\!\{\nabla_h {u}\}\!\}}\!\cdot\![\![{v}]\!]\mathrm{~d}{s}+\sum_{E\in\mathcal{E}_h}\int_E {\{\!\{\nabla_h {v}\}\!\}}\!\cdot\![\![{u}]\!]\mathrm{~d}{s}\bigg),
\end{align}
for $u, v \in V_h.$  Further, we define
\begin{align}\label{3.form3}
	\tilde{\mathcal{A}}_{DG}(u,v) := \nu a^c_{DG}(u,v) + \alpha b_{DG}(u,v) -\beta(c(u),v). 
\end{align}
\begin{remark}
	\begin{enumerate}
		Based on the formulations defined in \eqref{3.ADGf} and \eqref{3.form3}, we observe that
		\item For any $u, v \in V_h$, it holds that $ \mathcal{A}_{DG}(u,v) = \tilde{\mathcal{A}}_{DG}(u,v) + N_{DG}(u,v).$ 
		\item If we assume $u, v \in \H_0^1(\Omega)$, we find that $ \tilde{\mathcal{A}}_{DG}(u,v)= \mathcal{A}(u,v). $
	\end{enumerate}
\end{remark}
\bl{To establish the reliability}, we employ Cl\'ement interpolation~\cite[Lemma 4.6]{SZh}, defined as:
\begin{align} \label{3.clm}
	I_h: \H_0^1(\Omega)\rightarrow \big\{\phi\in C(\bar{\Omega}): \phi|_K\in \P_1(K),~ \forall~~ K\in \mathcal{T}_h, ~ \phi = 0 \text{~~on} ~~~\partial\Omega\big\},
\end{align}
with 	$|\!|\!| I_h\chi|\!|\!|\leq C	|\!|\!| \chi|\!|\!|$.
Additionally, the following estimates hold:
\begin{align}
	\bigg(\sum_{K\in \mathcal{T}_h}h_K^{-2}\|\chi-\I_h\chi\|^2_{\L^2(K)}\bigg)^{\frac{1}{2}} \leq C 	|\!|\!| \chi|\!|\!|, \label{3.clmes1}\\
	\bigg(\sum_{E\in \mathcal{E}}h_E^{-1}\|\chi-\I_h\chi\|^2_{\L^2(K)}\bigg)^{\frac{1}{2}} \leq C 	|\!|\!| \chi|\!|\!| \label{3.clmes2}.
\end{align}
\bl{We decompose the DG approximate solution $u_h$ as:  
	\begin{align}\label{3.confsol}  
		u_h = u_h^c + u_h^r,  
	\end{align}  
	where $u_h^c \in V_h^c = V_h \cap \H_0^1(\Omega)$ is the conforming component, and $u_h^r = u_h - u_h^c$ is the remainder.  The conforming part $u_h^c$ is obtained through the nodal averaging operator $\mathfrak{A}_h: V_h \to V_h^c$, such that $\mathfrak{A}u_h = u_h^c$, as defined in \cite[Theorem 2.2 and 2.3]{KPa}. The residual $u_h^r$ satisfies $u_h^r \in V_h$.  This decomposition separates $u_h$ into a conforming contribution $u_h^c$ and a non-conforming residual $u_h^r$.}
\begin{lemma}\label{lemma3.3}
	The remainder term $u_h^r$ can  be estimated as 
	\bl{\begin{align}
			|\!|\!| u_h^r|\!|\!|\leq C \bigg(\sum_{K\in \mathcal{T}_h}\zeta_{J_K}^2\bigg)^{\frac{1}{2}},
	\end{align}}
	where $C$ is independent of $h$.
\end{lemma}
\begin{proof}
	The conclusion directly arises from the decomposition $u_h = u_h^c + u_h^r$, utilizing the approximation results provided in \cite[Lemma 4.5]{SZh}, and considering the edge residual.
\end{proof}
\begin{lemma}\label{lemma3.4}
	Let  $u$ be the unique solution to equation \eqref{3.SGBHEweak}, and $u_h^c$ be the conforming solution as defined in \eqref{3.confsol} with $\chi = u-u_h^c$ such that \eqref{3.appara} holds.  Denoting $\zeta$ and $\mathcal{F}$ from \eqref{3.estimator}, the subsequent assertion holds:
	\begin{align}\label{3.renon}
		|\!|\!|{ u-u^c_h}|\!|\!|^2\leq C(\zeta + \mathcal{F} ) 	|\!|\!|{ u-u^c_h}|\!|\!| +\tilde{\mathcal{A}}_{DG}(u_h,\chi) - \tilde{\mathcal{A}}_{DG}(u_h^c,\chi). 
	\end{align}
\end{lemma}
\begin{proof}
	Utilizing the interpolation described in \eqref{3.clm}, we obtain:
	\begin{align}\label{3.15}
		( f,I_h \chi) =\mathcal{A}_{DG}(u_h, I_h\chi) = \tilde{\mathcal{A}}_{DG}(u_h, I_h\chi) + \nu  N_{DG}(u_h,I_h\chi),
	\end{align}
	where $N_{DG}(\cdot,\cdot)$ is defined in \eqref{3.adg1.2}. Applying Lemma \ref{lemma3.2} with $u_1 = u$, $u_2 = u_h^c$, and $\chi = u-u_h^c$, along with \eqref{3.15}, yields:
	\begin{align}\label{3.10}
		\nonumber	|\!|\!|{ u-u_h^c}|\!|\!|^2&\leq \tilde{\mathcal{A}}_{DG}(u,\chi) -   \tilde{\mathcal{A}}_{DG}(u_h^c,\chi)\\&\nonumber=  (f,\chi-I_h \chi)-   \tilde{\mathcal{A}}_{DG}(u_h,\chi-I_h \chi)+ \nu N_{DG}(u_h,I_h\chi)+\tilde{\mathcal{A}}_{DG}(u_h,\chi) - \tilde{\mathcal{A}}_{DG}(u_h^c,\chi)  \\&\nonumber= (f,\chi-I_h \chi) -\nu a_{DG}^c(u_h,\chi-I_h \chi) -\alpha b_{DG}(u_h,u_h,\chi-I_h \chi) +\beta(c(u_h),\chi-I_h \chi)\\&\quad + \nu N_{DG}(u_h,I_h\chi)+\tilde{\mathcal{A}}_{DG}(u_h,\chi) - \tilde{\mathcal{A}}_{DG}(u_h^c,\chi).
	\end{align}
	By employing integration by parts for the term $a_{DG}^c(u_h^c, \chi - I_h \chi)$, we attain:
	\begin{align}\label{3.intpar}
		\nonumber a_{DG}^c(u_h,\chi-I_h \chi) & = \sum_{K\in \mathcal{T}_h}\int_{K}\nabla_h u_h\nabla_h (\chi-I_h \chi) \mathrm{~d}K\\&=  -\sum_{K\in \mathcal{T}_h}\int_{K}\Delta u_h( \chi-I_h \chi) \mathrm{~d}K  +  \sum_{K\in \mathcal{T}_h}\int_{\partial K}\frac{\partial u_h}{\partial \nu} (\chi-I_h \chi) \mathrm{~d}S.
	\end{align}
	Upon substitution back into \eqref{3.10}, we derive:
	\begin{align*}
		|\!|\!|{ u-u_h^c}|\!|\!|^2&\leq \tilde{\mathcal{A}}_{DG}(u_h,\chi) - \tilde{\mathcal{A}}_{DG}(u_h^c,\chi) + (f,\chi-I_h \chi) + \nu N_{DG}(u_h,I_h\chi)\\&\quad- \bigg(\sum_{K\in \mathcal{T}_h}\int_{\partial K}\frac{\partial u_h}{\partial \nu} (\chi-I_h \chi) \mathrm{~d}S-\nu\sum_{K\in \mathcal{T}_h}\int_{K}\Delta u_h (\chi-I_h \chi) \mathrm{~d}K \\& \quad+ \alpha\sum_{K\in \mathcal{T}_h}\int_{K}u_h^{\delta}\sum_{i=1}^d\frac{\partial u_h}{\partial x_i}(\chi-I_h \chi)  \mathrm{~d}K- \sum_{K\in \mathcal{T}_h}\beta\int_{K}u_h(1-u_h^{\delta})(u_h^{\delta}-\gamma)(\chi-I_h \chi)  \mathrm{~d}K  \bigg)\\& = \tilde{\mathcal{A}}_{DG}(u_h,\chi) - \tilde{\mathcal{A}}_{DG}(u_h^c,\chi)+ (f-f_h,\chi-I_h \chi)+ \nu N_{DG}(u_h,I_h\chi)- \sum_{K\in \mathcal{T}_h}\nu\int_{\partial K}\frac{\partial u_h}{\partial \nu} (\chi-I_h \chi) \mathrm{~d}S\\&\quad+\sum_{K\in \mathcal{T}_h}\int_{K}\bigg(\nu\Delta u_h- \alpha u_h^{\delta}\sum_{i=1}^d\frac{\partial u_h}{\partial x_i}+\beta u_h(1-u_h^{\delta})(u_h^{\delta}-\gamma) + f_h\bigg)(\chi-I_h \chi) \mathrm{~d}K \\& =\tilde{\mathcal{A}}_{DG}(u_h,\chi) - \tilde{\mathcal{A}}_{DG}(u_h^c,\chi) + (f-f_h,\chi-I_h \chi)+ \nu N_{DG}(u_h,I_h\chi) \\&\quad - \sum_{K\in \mathcal{T}_h}\int_{\partial K}\nu\frac{\partial u_h}{\partial \nu} (\chi-I_h \chi) \mathrm{~d}S+\sum_{K\in \mathcal{T}_h}\int_{K}R_K (\chi-I_h \chi) \mathrm{~d}K\\& : = \tilde{\mathcal{A}}_{DG}(u_h,\chi) - \tilde{\mathcal{A}}_{DG}(u_h^c,\chi)+ J_1 + J_2 + J_3 +J_4,
	\end{align*}
	where the terms $J_i's$  
	are defined as follows 
	\begin{align*}
		J_1 &= (f-f_h,\chi-I_h \chi), \quad\quad J_2=\nu N_{DG}(u_h,I_h\chi),\\ J_3 &= -\nu\sum_{K\in \mathcal{T}_h}\int_{\partial K}\frac{\partial u_h}{\partial \nu} (\chi-I_h \chi) \mathrm{~d}S,\quad J_4= \sum_{K\in \mathcal{T}_h}\int_{K}R_K (\chi-I_h \chi) \mathrm{~d}K.
	\end{align*}
	Firstly, the $J_1$ term can be estimated using Cauchy-Schwarz inequality as 
	\begin{align}\label{3.11}
		|J_1|\leq\left(\sum_{K \in \mathcal{T}_h} h_K^2\|f-f_h\|^2_{\L^2(K)}\right)^{1 / 2}\bigg( \sum_{K\in \mathcal{T}_h}h_K^{-2}\| \chi-I_h \chi\|_{\L^2(K)}^2\bigg)^{\frac{1}{2}}.
	\end{align}
	Secondly, for $J_2$, employing the inverse inequality akin to \cite[Lemma 4.3]{SZh}, we obtain:
	\begin{align}\label{3.12}
		|J_2| \leq C \bigg(\sum_{K\in \mathcal{T}_h} \zeta_{J_K}^2\bigg)^{\frac{1}{2}}|\!|\!| \chi|\!|\!|.
	\end{align}
	Again, using Cauchy-Schwarz inequality and the estimate \eqref{3.clmes2}, we have 
	\begin{align}\label{3.13}
		\nonumber |J_3| &=\bigg| \sum_{K\in \mathcal{T}_h}\int_{\partial K}\frac{\partial u_h}{\partial \nu} (\chi-I_h \chi) \mathrm{~d}S\bigg| = \bigg|\sum_{E\in \mathcal{E}_h} \int_E [\![\nabla_h u_h]\!](\chi-I_h \chi) \mathrm{~d}S\bigg|\\&\nonumber\leq C \bigg(\sum_{E\in \mathcal{E}_h}h_E\| [\![\nabla_h u_h]\!]\|_{\L^2(E)}^2\bigg)^{\frac{1}{2}}\bigg(\sum_{E\in \mathcal{E}_h}h_E^{-1} \| \chi-I_h \chi\|_{\L^2(E)}^2\bigg)^{\frac{1}{2}}\\& \leq C \bigg(\sum_{E\in \mathcal{E}_h} \zeta_{e_K}^2\bigg)^{\frac{1}{2}}|\!|\!| \chi|\!|\!|.
	\end{align}
	Finally, $J_4$ can be estimated using \eqref{3.clmes1} as 
	\begin{align}\label{3.14}
		\nonumber |J_4| =  \bigg|\sum_{K\in \mathcal{T}_h}\int_{K}R_K (\chi-I_h \chi) \mathrm{~d}K\bigg| &\leq \bigg( \sum_{K\in \mathcal{T}_h}h_K^2\|R_K\|_{\L^2(K)}^2\bigg)^{\frac{1}{2}}\bigg( \sum_{K\in \mathcal{T}_h}h_K^{-2}\| \chi-I_h \chi\|_{\L^2(K)}^2\bigg)^{\frac{1}{2}}\\&\leq C \bigg(\sum_{K\in \mathcal{T}_h}\zeta_{R_K}^2\bigg)^{\frac{1}{2}}|\!|\!| \chi|\!|\!|.
	\end{align}
	Combining \eqref{3.11}-\eqref{3.14} gives the required result.
\end{proof}

Further, it is required to bound $\tilde{\mathcal{A}}_{DG}(u_h,\chi) - \tilde{\mathcal{A}}_{DG}(u_h^c,\chi)$ in terms of the residual, which is discussed in the following lemma.
\begin{lemma}\label{lemma3.6}
	\bl{Assume that $\delta \in \mathbb{N}$ for $d = 2,$ and $\delta =1,2,$ for $d=3$, and $u_h^r = u_h - u_h^c,$ we have}
	$$\tilde{\mathcal{A}}_{DG}(u_h,\chi) - \tilde{\mathcal{A}}_{DG}(u_h^c,\chi) \leq C  |\!|\!| u_h^r|\!|\!||\!|\!| \chi|\!|\!|.$$
\end{lemma}
\begin{proof} 
	Considering the definition of $\mathcal{A}_{DG}$ as provided in \eqref{3.ADGf}, it is evident that:
	\begin{align*}
		\tilde{\mathcal{A}}_{DG}(u_h,\chi) - \tilde{\mathcal{A}}_{DG}(u_h^c,\chi)&= \nu a_{DG}(u_h-u_h^c, \chi) + \alpha \big(b_{DG}(u_h,u_h,\chi)-b_{DG}(u_h^c,u_h^c,\chi)\big)\\&\quad-\beta\big((c(u_h),\chi)-(c(u_h^c),\chi)\big),
	\end{align*}
	By applying the Cauchy-Schwarz inequality, we obtain:
	\begin{align}
		a_{DG}(u_h-u_h^c, \chi)\leq |\!|\!| u_h^r|\!|\!||\!|\!| \chi|\!|\!|.
	\end{align}
	For the second term, we consider $	J'_1 =\alpha( b_{DG}(u_h,u_h,\chi)-b_{DG}(u_h^c,u_h^c,\chi))$,
	which can be estimated using an integration by parts and inverse estimate in the first term. Then applying Taylor's formula, H\"older's  inequality and the Sobolev embedding, $\L^{p}(\Omega)\subset \H_0^1(\Omega),$ for, $1\leq p < \infty,$ for $d=2,$ and $1\leq p\leq 6,$ for $d=3$ and the stability estimates leads to 
	\begin{align}\label{3.b1}
		\nonumber	|J'_1|
		& \leq \frac{2\alpha\delta}{(\delta+2)(\delta+1)}\sum_{i=1}^d \left(u_h^{\delta+1}-(u_h^c)^{\delta+1}, \frac{\partial \chi}{\partial{x_i}}\right)
		\\\nonumber
		&\leq  \frac{2^{\delta}\alpha\delta}{(\delta+2)}\bigg(\|u_h\|^{\delta}_{\L^{2(\delta+1)}}+\|u_h^c\|^{\delta}_{\L^{2(\delta+1)}}\bigg)\|u_h-u_h^c\|_{\L^{2(\delta+1)}}\|\nabla_h \chi\|_{\L^2(\mathcal{T}_h)}
		\\
		&\leq C |\!|\!| u_h^r|\!|\!||\!|\!| \chi|\!|\!|.
	\end{align}
	The reaction term $c(u)$ can be rewritten as  
	\begin{align*}
		J_2'& = (c(u_h),\chi)-(c(u_h),\chi) \\&= \underbrace{ -2\beta\gamma(u_h-u_h^c,\chi)}_{J'_3}+\underbrace{2\beta(1+\gamma)({u_h}^{\delta+1}-(u_h^c)^{\delta+1},\chi) }_{J'_4}\underbrace{-2\beta({u_h}^{2\delta+1}-(u_h^c)^{2\delta+1},\chi)}_{J'_5},
	\end{align*} 
	Using Cauchy-Schwarz, we have 
	\begin{align}\label{3.rc1}
		|J'_3|&\leq C |\!|\!| u_h^r|\!|\!||\!|\!| \chi|\!|\!|.
	\end{align}
	An application of Taylor's formula yields
	\begin{align}\label{3.rc2}
		|J_4'|&=2\beta(1+\gamma)(\delta+1)((\theta u_h+(1-\theta)u_h^c)^{\delta}(u_h-u_h^c),\chi) 
		\nonumber\\&\leq 2^{\delta}\beta(1+\gamma)(\delta+1)\left(\|u_h\|^{\delta}_{\L^{2(\delta+1)}}+\|u_h^c\|^{\delta}_{\L^{2(\delta+1)}}\right)\|u_h-u_h^c\|_{\L^{2(\delta+1)}}\|\chi\|_{\L^2}
		\nonumber\\&\leq  C|\!|\!| u_h^r|\!|\!||\!|\!| \chi|\!|\!|.
	\end{align}
	Finally, using Taylor's formula, H\"older inequality and Sobolev embedding, we have
	\begin{align}\label{3.rc3}
		|J_5'|&=2(2\delta+1)\beta\left((\theta u_h+(1-\theta)u_h^c)^{2\delta}(u_h-u_h^c),\chi\right)
		\nonumber\\&\leq  2^{2\delta}(2\delta+1)\beta\left(\|u_h\|_{\L^{2(\delta+1)}}^{2\delta}+\|u_h^c\|_{\L^{2(\delta+1)}}^{2\delta}\right)\|u_h-u_h^c\|_{\L^{2(\delta+1)}}\|\chi\|_{\L^{2(\delta+1)}}
		\nonumber\\&\leq C |\!|\!| u_h^r|\!|\!||\!|\!| \chi|\!|\!|.
	\end{align}
	Upon substituting back, the desired result seamlessly follows.
\end{proof}
\begin{theorem}\label{3.thm3.7}
	Let $u$ be the unique solution of \eqref{3.SGBHEweak} and $u_h$ be the DG approximated solution. Let $\zeta$ be the a posteriori error estimator defined in  \eqref{3.estimator}, then the following estimator holds:
	\begin{align}
		|\!|\!|{ u-u_h}|\!|\!|\leq C(\zeta +\mathcal{F} ).
	\end{align}
\end{theorem}
\begin{proof}
	Using triangle inequality, we have
	$$|\!|\!|{ u-u_h}|\!|\!| = |\!|\!|{ u-u^c_h-u_h^r}|\!|\!|\leq |\!|\!|{ u-u^c_h}|\!|\!| + |\!|\!|{ u^r_{h}}|\!|\!|.$$
	An application of Lemma \ref{lemma3.3} -- \ref{lemma3.6} yields the desired outcome.
\end{proof}
\subsection{Efficiency}
The efficiency of the estimator can be evaluated through the conventional polynomial bubble functions technique \cite[p.~1771]{VerS}. To accomplish this, we introduce an interior bubble function $b_K$ defined and supported on an element $K$. Let $E$ be a common internal edge between two elements $K$ and $K'$, and let $\omega_E$ denote the patch $K \cup K'$. Consequently, an edge bubble function $b_E$ is defined on $E$ in a manner that it is positive in the interior of the patch and zero on the patch's boundary. It is noteworthy that these functions satisfy the conditions:
$$b_K\in \H_0^1(K),\quad b_E\in \H_0^1(\omega_E), \quad \text{and} \quad \|b_K\|_{\L^{\infty}(K)} =  \|b_E\|_{\L^{\infty}(E)} =1.$$ 
Further, the following result holds true \cite[Lemma 4.10]{SZh}.
\begin{lemma}\label{lemma3.12}
	For every element $K$ and edge $E$, the subsequent outcomes are valid:
	\begin{align}
		\|b_Kv\|_{\L^2(K)}\leq C \|v\|_{\L^2(K)},\quad 
		\|v\|_{\L^2(K)}^2&\leq C (v, b_Kv)_K,\label{3.3.12.1}\\
		|\!|\!|b_K v|\!|\!| \leq h_K^{-1}\|v\|_{\L^2(K)},\quad
		\|\sigma\|_{\L^2(E)}^2 &\leq (\sigma,b_E\sigma)_E,\label{3.3.12.2}\\
		\|b_E\sigma\|_{\L^2(\omega_E)} \leq h_E^{\frac{1}{2}}\|\sigma\|_{\L^2(E)},\quad
		|\!|\!|b_E\sigma|\!|\!|_{\omega_E} &\leq h_E^{\frac{-1}{2}}\|\sigma\|_{\L^2(E)}\label{3.3.12.3}, 	
	\end{align}
	where $v$ and $\sigma$ are defined on elements and faces, respectively.
\end{lemma}
For the element-wise residual, the following estimates holds:
\begin{lemma}\label{lemma3.13}
	There holds:
	\begin{align}
		\Big(\sum_{K \in \mathcal{T}_h}	\zeta_{R_K}^2\Big)^{\frac{1}{2}}\leq |\!|\!|u_h-u|\!|\!| + \mathcal{F}
	\end{align}
\end{lemma}
\begin{proof}
	Recall that 
	$$R_K := \bigg\{f_h+\nu\Delta u_h- \alpha u_h^{\delta}\sum\limits_{i=1}^d\frac{\partial u_h}{\partial x_i}+\beta u_h(1-u_h^{\delta})(u_h^{\delta}-\gamma)\bigg\}\bigg|_K,$$ 
	we set $B|_K = h_K^2R_Kb_k. $ Using estimate \eqref{3.3.12.1}, we get
	\begin{align}
		\sum_{K \in \mathcal{T}_h}	\zeta_{R_K}^2 &= \sum_{K \in \mathcal{T}_h} h_K^2\|R_K\|_{\L^2(K)} \leq C\sum_{K \in \mathcal{T}_h}	(R_K,h_K^2b_KR_K)_K \leq C\sum_{K \in \mathcal{T}_h}	(R_K,B)_K \\&=  \sum_{K \in \mathcal{T}_h} \Big(f_h+\nu\Delta u_h- \alpha u_h^{\delta}\sum\limits_{i=1}^d\frac{\partial u_h}{\partial x_i}+\beta u_h(1-u_h^{\delta})(u_h^{\delta}-\gamma),B\Big)_K.
	\end{align}
	Using the exact solution $\Big(f+\nu\Delta u- \alpha u^{\delta}\sum\limits_{i=1}^d\frac{\partial u}{\partial x_i}+\beta u(1-u^{\delta})(u^{\delta}-\gamma)\Big)\Big|_K = 0,$ and integration by parts, we achieve
	\begin{align}
		\sum_{K \in \mathcal{T}_h}	\zeta_{R_K}^2 &\leq C \sum_{K \in \mathcal{T}_h} \Big(\nu(\nabla_h(u_h-u),\nabla_hB)_K + \frac{\alpha}{\delta+1} (u_h^{\delta+1}-u^{\delta+1},\nabla_hB)_K\Big)\\&\quad+ \sum_{K \in \mathcal{T}_h} \Big((c(u_h))-c(u) +(f_h-f), B\Big)_K
	\end{align}
	by an application of Cauchy-Schwarz inequality and estimate similar to \eqref{lemma3.6}, we get 
	\begin{align}
		\sum_{K \in \mathcal{T}_h}	\zeta_{R_K}^2 &\leq C \left(|\!|\!|u_h-u|\!|\!| +  \Big(\sum_{K \in \mathcal{T}_h} h_K^2\|f-f_h\|^2_{\L^2(K)}\Big)^{1 / 2} \right)\Big(\sum_{K \in \mathcal{T}_h}|\!|\!|B|\!|\!|^2_K + h_K^{-2}\|B\|_{\L^2(K)}^2\Big)^{\frac{1}{2}}.
	\end{align}
	Finally, using estimates in \eqref{3.3.12.1}, \eqref{3.3.12.2} obtained in Lemma \eqref{lemma3.12}, we have 
	\begin{align}
		|\!|\!|B|\!|\!|^2_K \leq C h_K^2\|R_K\|_{\L^2(K)}^2 \qquad \text{and }\quad h_K^{-2}\|B\|_{\L^2(K)}^2 \leq C h_K^2\|R_K\|_{\L^2(K)},
	\end{align}
	yields the required proof.
\end{proof} 
\begin{lemma}\label{lemma3.14}
	The following estimate holds 
	\begin{align}
		\Big(\sum_{E \in \mathcal{E}_h}	\zeta_{R_E}^2\Big)^{\frac{1}{2}}\leq |\!|\!|u-u_h|\!|\!| + \mathcal{F}.
	\end{align}
\end{lemma}
\begin{proof}
	For an interior edge $E\in\mathcal{E}_h^i$, we define the edge bubble function as
	\begin{align}\label{tau}
		\sigma = \frac{1}{2}\sum_{E\in \mathcal{E}_h^i} [\![(\nu\nabla u_h)\cdot\n]\!]b_E. 
	\end{align}  
	Using \eqref{3.3.12.2} and that $[\![(\nu\nabla_h u)\cdot\n]\!] = 0$ on the interior edge, we have 
	\begin{align}
		\sum_{K \in \mathcal{T}_h}	\zeta_{R_E}^2 \leq C \sum_{E\in \mathcal{E}_h^i} ([\![(\nu\nabla_h u_h)\cdot\n]\!], \sigma)_E = \sum_{E\in \mathcal{E}_h^i} \Big([\![(\nu\nabla_h (u_h-u))\cdot\n]\!], \sigma\Big)_E.
	\end{align}
	After integration by parts over each element of the patch $\omega_E$, we obtain 
	\begin{align}
		\big([\![(\nu\nabla_h (u_h-u))\cdot\n]\!], \sigma\big)_E = \sum_{K\in \omega_E} \int_K \nu \Delta(u_h-u) \sigma \mathrm{~d}x+ \int_K\nu\nabla_h(u_h-u)\cdot\nabla_h \sigma\mathrm{~d}x.
	\end{align}
	As $u$ solve the differential equation \eqref{3.SGBHEweak}, we achieve
	\begin{align}
		\big([\![(\nu\nabla_h (u_h-u))\n]\!], \sigma\big)_E &= \sum_{K\in \omega_E} \int_K\big(f_h+\nu\Delta u_h- \alpha u_h^{\delta}\sum\limits_{i=1}^d\frac{\partial u_h}{\partial x_i}+\beta u_h(1-u_h^{\delta})(u_h^{\delta}-\gamma)\big) \sigma\mathrm{~d}x\\&\quad+
		\sum_{K\in \omega_E} \int_K \Big( \alpha u_h^{\delta}\sum\limits_{i=1}^d\frac{\partial u_h}{\partial x_i}- \alpha u^{\delta}\sum\limits_{i=1}^d\frac{\partial u}{\partial x_i}\Big)\sigma\mathrm{~d}x \\&\quad +	\sum_{K\in \omega_E} \int_K \Big( -\beta u_h(1-u_h^{\delta})(u_h^{\delta}-\gamma)+\beta u(1-u^{\delta})(u^{\delta}-\gamma)\Big)\sigma \mathrm{~d}x\\&\quad+ 
		\sum_{K\in \omega_E} \int_K (f-f_h) \sigma\mathrm{~d}x := J_1 + J_2+ J_3 + J_4.
	\end{align}
	To estimate $J_1$, using Cauchy-Schwarz inequality and Lemma \ref{lemma3.13}, we obtain 
	\begin{align}
		J_1 &= \sum_{K\in \omega_E} \int_K\big(f_h+\nu\Delta u_h- \alpha u_h^{\delta}\sum\limits_{i=1}^d\frac{\partial u_h}{\partial x_i}+\beta u_h(1-u_h^{\delta})(u_h^{\delta}-\gamma)\big) \sigma\mathrm{~d}x =  \sum_{K\in \omega_E} \int_K R_K\sigma \mathrm{~d}x \\&\leq \Big(\sum_{K\in\omega_E}h_K^2\|R_K\|_{\L^2(K)}^2\Big)^{\frac{1}{2}}\Big(\sum_{K\in\omega_E}h_K^{-2}\|\sigma\|_{\L^2(K)}^2\Big)^{\frac{1}{2}}\\&\leq \left(|\!|\!|u_h-u|\!|\!| + \Big(\sum_{K \in \mathcal{T}_h} h_K^2\|f-f_h\|^2_{\L^2(K)}\Big)^{1 / 2}\right)\Big(\sum_{E \in \mathcal{E}_h}	\zeta_{R_E}^2\Big)^{\frac{1}{2}},
	\end{align}
	where, we have used \eqref{tau} and \eqref{3.3.12.3}. Using integration by parts over $\omega_E$, Taylor's formula for the advection term as done in \eqref{3.b1} yields
	\begin{align}
		J_2 = 	\sum_{K\in \omega_E} \int_K \Big( \alpha u_h^{\delta}\sum\limits_{i=1}^d\frac{\partial u_h}{\partial x_i}- \alpha u^{\delta}\sum\limits_{i=1}^d\frac{\partial u}{\partial x_i}\Big)\sigma\mathrm{~d}x&= 	\sum_{K\in \omega_E} \int_K \Big( \frac{\alpha}{\delta +1} (u_h^{\delta+1}- u^{\delta+1})\Big)\nabla_h\sigma\mathrm{~d}x \\&\leq C|\!|\!|u_h-u|\!|\!| \Big(\sum_{K \in \omega_E }|\!|\!|\sigma|\!|\!|_{\omega_E}^2\Big)^{\frac{1}{2}}
		\\&\leq C|\!|\!|u_h-u|\!|\!| \Big(\sum_{E \in \mathcal{E}_h}	\zeta_{R_E}^2\Big)^{\frac{1}{2}}.
	\end{align}
	Using estimates similar to \eqref{3.rc1}-\eqref{3.rc3} in $J_3$, we attain
	\begin{align}
		J_3 = \sum_{K\in \omega_E} \int_K \Big( -\beta u_h(1-u_h^{\delta})(u_h^{\delta}-\gamma)+\beta u(1-u^{\delta})(u^{\delta}-\gamma)\Big)\sigma \mathrm{~d}x \leq C  |\!|\!|u_h-u|\!|\!| \Big(\sum_{E \in \mathcal{E}_h}	\zeta_{R_E}^2\Big)^{\frac{1}{2}}.
	\end{align}
	and finally using Cauchy-Schwarz inequality, $J_4$ can be estimated as 
	\begin{align}
		J_4 = 	\sum_{K\in \omega_E} \int_K (f-f_h) \sigma\mathrm{~d}x & \leq \Big(\sum_{K \in \mathcal{T}_h} h_K^2\|f-f_h\|^2_{\L^2(K)}\Big)^{1 / 2}\Big(\sum_{K\in\omega_E}h_K^{-2}\|\sigma\|_{\L^2(K)}^2\Big)^{\frac{1}{2}}\\&\leq \Big(\sum_{K \in \mathcal{T}_h} h_K^2\|f-f_h\|^2_{\L^2(K)}\Big)^{1 / 2}\Big(\sum_{E \in \mathcal{E}_h}	\zeta_{R_E}^2\Big)^{\frac{1}{2}}.
	\end{align} 
	By consolidating these estimates, we attain the intended outcome.
\end{proof}

Finally, we state the main theorem for the efficiency of our estimator.
\begin{theorem}
	Let $u$ and $u_h$ denote the solution of \eqref{3.SGBHEweak} and its Discontinuous Galerkin (DG) approximation defined in \eqref{3.SGBHE}, respectively. The error estimator $\zeta$ from \eqref{3.estimator} is bounded as follows:
	\begin{align}
		\zeta \leq C(|\!|\!|u-u_h|\!|\!| +\mathcal{F}).
	\end{align}
\end{theorem}
\begin{proof}
	First, note that $[\![ u]\!]= 0,$ which implies that
	$$	\zeta_{J_K} \leq C|\!|\!|u-u_h|\!|\!| .$$
	and then combining the Lemma \ref{lemma3.13} and \ref{lemma3.14}, we attain the required result. 
\end{proof}
\section{ Semi-discrete scheme for GBHE with weakly singular kernel}\setcounter{equation}{0}\label{sec4}
In this section, we examine the GBHE with weakly singular kernel \eqref{3.weakGBHE}. Initially, we compute the $\L^2$ error estimates for the semi-discrete DGFEM. In this section for the simplicity of notation, we write $(u_h)_t^*(t) = u_{h,t}^*(t)$. 
\subsection{$\L^2$-Error estimates for GBHE with weakly singular kernel }
Let $u(t)$ represent the weak solution of the GBHE with a weakly singular kernel as defined in equation \eqref{3.weakGBHE}. Let $u_h(t)$ denote the dG solution outlined in equation \eqref{2.dgweakform}. Utilizing the projection $u^*_h(t)$ defined in \eqref{3.Aproj} for any $t$ within the interval $[0,T]$, we can decompose our error term $e(t) = u(t)-u_h(t)$ into two components:  $e(t) = u(t)-u_h(t) = (u(t)-u_h^*(t))+ (u_h^*(t)-u_h(t))$.
The following estimates holds for the first term.
\begin{lemma}\label{3.lemma4.1}
	Let $u(t), u_t(t) \in \H^{p+1}(\Omega)$, for $p\in[1,k]$ and $\omega(t) = u(t)-u_h^*(t)$, for any $t\in [0,t]$. There exists a constant $C>0$, such that 
	\begin{align}
		|\!|\!|\omega_t(t)	|\!|\!|&\leq Ch^p|u_t(t)|_{\H^{p+1}},\\
		\|\omega_t(t)\|_{\L^2}&\leq  Ch^{p+1}|u_t(t)|_{\H^{p+1}}.
	\end{align}
\end{lemma}
\begin{proof}
	\bl{For any $t \in [0, T]$, let $\omega_t(t) = u_t(t) - u_{h,t}^*(t)$. Recall, the $A_h$-projection \eqref{3.Aproj} is defined as:  
		$$
		a_{DG}(u_h^*, \phi_h) = a_{DG}(u, \phi_h) \quad \forall~~ \phi_h \in V_h \cap L^\infty(\Omega),
		$$ 
		which leads to  
		\begin{align}  
			0 = \frac{\mathrm{d}}{\mathrm{d}t} (a_{DG}(u(t) - u_h^*(t), \phi_h(t))) = a_{DG}\left(\frac{\partial (u(t) - u_h^*(t))}{\partial t}, \phi_h(t)\right).  
		\end{align}  
		This establishes the time evolution of the error $\omega_t(t)$ in terms of the $A_h$-projection.
	}
	So, we have 
	$$a_{DG}(\omega_t(t), \phi_h(t)) = 0 \quad\forall ~~\phi_h(t)\in V_h.$$
	Now, for any $t\in[0,T]$, denote the $\L^2$ projection of $u_t$ as
	$$ \Pi_h u_t(t) = \big(\Pi_hu(t))_t = \hat{u}_t(t) \in V_h.$$
	Using triangle inequality and the method  similar to Lemma \ref{3.Lemma3.1}, we attain  
	$$|\!|\!|\omega_t|\!|\!| = |\!|\!|u_t-u_{h,t}^*|\!|\!| \leq |\!|\!|u_t-\hat{u}_t|\!|\!| + |\!|\!|\hat{u}_t - u_{h,t}^*|\!|\!|,$$
	where the first term is estimated using $\L^2-$projection and the subsequent term using \eqref{3.55}. The estimate of $\omega_t(t)$ have been proved in \cite[Lemma 4.2]{DFK}.
\end{proof}		
\begin{theorem}\label{3.L2sd}
	Let $u$ be the exact solution and $u_h$ be the approximated solution. For $u(t)\in \H^{p+1}(\Omega)$ and $u_t(t)\in \H^{p+1}(\Omega),$ the error $e = u-u_h$ satisfies
	$$\|e(t)\|_{\L^{\infty}(0,T;\L^2(\Omega))} + h \int_0^t|\!|\!| e(\tau)|\!|\!| \mathrm{~d}\tau \leq \|e(0)\|_{\L^2}^2+ Ch^{p+1}\Big(|u(t)|_{\H^{p+1}}^2+|u_t(t)|_{\H^{p+1}}^2\Big),$$
	\bl{where $C$ is a constant independent of $h$, $k$ is the degree of the approximating polynomial, and $p \in [1, k]$.}
\end{theorem}
\begin{proof}
	Subtracting \eqref{3.weakGBHE} from \eqref{2.dgweakform}, we attain
	\begin{align}\label{7p2}
		&	\langle\partial_t(u(t)-u_h(t)),\chi\rangle+\bl{\nu a_{DG}(u(t)-u_h(t), \chi)} + \eta( K*a_{DG}( u(t)-u_h(t)),  \chi) \\&+\alpha[b(u(t),u(t),\chi)-b_{DG}(u_h(t),u_h(t),\chi)]-\beta[c(u(t),\chi)-c(u_h(t),\chi)]=0,
	\end{align}
	using $\omega = u-u_h^*, ~\xi = u_h^*-u_h$, we have 
	\begin{align}
		&	(\partial_t\xi, \chi) +\nu a_{DG}(\xi, \chi) + \eta( K*a_{DG}(\xi,  \chi))+\alpha (b_{DG}(u_h^*,u_h^*,\chi)-b_{DG}(u_h,u_h,\chi))\\&\quad - \beta((c(u_h^*),\chi)+ (c(u_h),\chi))\nonumber\\&= -\alpha[b_{DG}(u_h(t),u_h(t),\chi)-b(u(t),u(t),\chi)]+\beta[c(u_h(t),\chi)-c(u(t),\chi)]- (\partial_t\omega,\xi)- a_{DG}(\omega, \xi)\nonumber\\&= -\alpha[b(u_h^*(t),u_h^*(t),\chi)-b(u(t),u(t),\chi)]+\beta[c(u_h^*(t),\chi)-c(u(t),\chi)]- (\partial_t\omega,\xi)
	\end{align}
	Setting $\chi=\xi$ along with an estimate akin to that in Lemma \ref{3.A1}, we obtain
	\begin{align}
		&	\frac{1}{2}\frac{d}{dt}\|\xi\|_{\L^2}^2+\frac{\nu}{2}|\!|\!| \xi|\!|\!|^2 +\frac{\beta}{4}(\|{u}^{\delta}_h\xi\|_{\L^2}^2+\|(u_h^*)^{\delta}\xi\|_{\L^2}^2)+ \eta( K*a_{DG}(\xi,  \xi))\nonumber
		\\&\quad +\left(\beta\gamma-C(\beta,\alpha,\delta) - C(\alpha,\nu)\Big(\|u_h\|^{\frac{8\delta}{4-d}}_{\L^{4\delta}}+\|u_h^*\|^{\frac{8\delta}{4-d}}_{\L^{4\delta}}\Big)\right)\|\xi\|_{\L^2}^2\\&\leq\bigg( \frac{2^{2\delta}\alpha^2\delta^2}{\nu(\delta+2)^2}\left(\|{u_h^*}\|_{\L^{\infty}}^{2\delta}+\|{u}\|_{\L^{\infty}}^{2\delta}\right) + \frac{1}{\nu}2^{2\delta}\beta^2(1+\gamma)^2(\delta+1)^2\left(\|{u_h^*}\|^{2\delta}_{\L^{4\delta}}+\|{u}\|^{2\delta}_{\L^{4\delta}}\right) \\&\quad+  2\beta\gamma +\frac{2^{4\delta}(2\delta+1)^2\beta^2}{\nu}\left(\|u_h^*\|_{\L^{\infty}}^{4\delta}+\|u\|_{\L^{\infty}}^{4\delta}\right)\bigg)\|\omega\|_{\L^2}^2 +\|\omega_t\|_{\L^2}^2+\|\xi\|_{\L^2(\Omega)}^2 .
	\end{align}
	Moreover, we have 
	\begin{align*}
		\frac{d}{dt}\|\xi\|_{\L^2}^2+ \nu|\!|\!| \xi|\!|\!|^2+ \eta( K*a_{DG}(\xi,  \xi))&\leq C(\alpha,\delta,\nu,\gamma,\|u_h^*\|_{\L^{\infty}},\|u\|_{\L^{\infty}}) \|\omega\|_{\L^2}^2 +\|\omega_t\|_{\L^2}^2\\&\quad+\left(C(\beta,\alpha,\delta) + C(\alpha,\nu)\Big(\|u_h\|^{\frac{8\delta}{4-d}}_{\L^{4\delta}}+\|(u_h^*)\|^{\frac{8\delta}{4-d}}_{\L^{4\delta}}\Big)\right)\|\xi\|_{\L^2}^2
	\end{align*}
	Integrating from $0$ to $T$, using Gronwall inequality and Lemma \ref{3.Lemma 3.5} with the positivity of the kernel \eqref{pk}, we achieve
	\begin{align*}
		&\|\xi\|_{\L^2}^2+ \int_0^t|\!|\!| \xi(\tau)|\!|\!| ^2\mathrm{d}\tau \\&\leq \exp\bigg\{C(\alpha,\delta,\nu,\gamma,\|u_h^*\|_{\L^{\infty}},\|u\|_{\L^{\infty}})+C(\beta,\alpha,\delta) + C(\alpha,\nu)\Big(\|u_h\|^{\frac{8\delta}{4-d}}_{\L^{4\delta}}+\|u_h^*\|^{\frac{8\delta}{4-d}}_{\L^{4\delta}}\Big) \bigg\}\\&\quad \times\bigg(\|\xi(0)\|_{\L^2} +  h^{2(p+1)}\Big(|u|_{\H^{p+1}(\Omega)}^2+|u_t|_{\H^{p+1}(\Omega)}^2\Big) \bigg).
	\end{align*}
	Finally, by applying the triangle inequality, Lemma \ref{3.Lemma 3.5}, Lemma \ref{3.lemma4.1}, and Remark \ref{rhs} we achieve the desired outcome.
\end{proof}
\subsection{A posteriori error estimates }
In this section, we investigate a posteriori error estimates for the semi-discrete approximation of the GBHE with a weakly singular kernel. We employ the DGFEM for spatial discretization. \bl{For $ \tilde{u}$ defined in \eqref{3.4.1}, the error term is decomposed as follows: $e = u - u_h = u - \tilde{u} + \tilde{u} - u_h := \rho + \theta$}, where $\rho = u - \tilde{u}$ denotes the parabolic error, and $\theta = \tilde{u} - u_h$ signifies the elliptic error. To address the impact of the history arising from the presence of the memory term (weakly singular kernel), we introduce the error estimator in \eqref{3.sdesti}. In Lemma \ref{lemma4.1}, we delve into the reliability corresponding to the elliptic part. Finally, we define the estimator for the semi-discrete case in \eqref{4.theta} and establish the upper bound of the error in Theorem \ref{3.theorem4.3}.

\bl{The auxiliary problem, a form of elliptic reconstruction as discussed in \cite[Section 2]{MNo}, is defined as follows: Find $\tilde{u}\in \H_0^1(\Omega)$ such that 
	\begin{align}\label{3.4.1}
		\mathcal{A}(\tilde{u}(t),v) = (g_h(t),v), \quad \forall~~ v\in \H_0^1(\Omega).
\end{align}}
where, 
\begin{align}\label{3.g}
	(g_h(t),v) &=  (f(t),v)-(\partial_tu_h(t),v) - \eta\int_0^t K(t-\tau)a_{DG}^c(u_h(\tau), v) \mathrm{~d}\tau,
\end{align}
for all, $t\in[0,T]$, \bl{recall that $a_{DG}^c(\cdot,\cdot)$ have been defined in \eqref{3.adg1.1}}.
Furthermore, the semi-discrete weak formulation of the GBHE with a weakly singular kernel, represented by \eqref{2.dgweakform}, can be expressed as follows: Find $u_h\in \C(0,T; V_h)$ such that
\begin{align}\label{3.4.2}
	\mathcal{A}_{DG}(u_h,\chi) \bl{+ \eta\int_0^t K(t-\tau)N_{DG}(u_h(\tau), \chi) \mathrm{~d}\tau}= (g_h(t),\chi)  , \quad \forall~~ \chi\in V_h,
\end{align}
for each $t\in [0,T]$. 
Then, the residuals corresponding to the each element $K\in \mathcal{T}_h$ and edge $E\in \mathcal{E}_h$ are defined as 
\begin{itemize}
	\item Element-wise residual, 
	$
	R_K := \big\{f_h - \partial_tu_h+\eta \int\limits_0^tK(t-\tau)\Delta u_h(\tau) \mathrm{~d}\tau+\nu\Delta u_h- \alpha u_h^{\delta}\sum\limits_{i=1}^d\frac{\partial u_h}{\partial x_i}
	\\+\beta u_h(1-u_h^{\delta})(u_h^{\delta}-\gamma)\big\}\big|_K,\label{3.sder}
	$
	\item Edge-wise residual for diffusion term, 
	$R_{E_1} := \begin{cases}\frac{\nu}{2} [\![(\nabla_h u_h)\cdot\n]\!]  & \text { for } E \in \mathcal{E}^i_h,\\
		0 & \text { for } E \in\mathcal{E}^{\partial}_h,\end{cases}$
	\item Edge-wise residual for memory term, 
	$R_{E_2}:= \begin{cases}\frac{\eta}{2} \Big(\int\limits_0^tK(t-\tau)[\![(\nabla_h u_h(\tau))\cdot\n]\!]\mathrm{~d}\tau\Big) & \text { for } E \in \mathcal{E}_h^i,\\
		0 & \text { for } E \in \mathcal{E}^{\partial}_h.\end{cases}$
\end{itemize}
We introduce the element-wise error  estimator for the time dependent system as, 
$\zeta_K^2=\zeta_{R_K}^2+\zeta_{E_K}^2+\zeta_{J_K}^2$, where 
\begin{align}\label{3.sdesti}
	\nonumber&\zeta_{R_K}^2:=h_K^2\|R_K\|_{\L^2(K)}^2,\quad
	\zeta_{E_K}^2:=\sum\limits_{E \in \partial K} h_E(\left\|R_{E_1}\right\|_{\L^2(E)}^2+\left\|R_{E_2}\right\|_{\L^2(E)}^2),\\&\quad
	\zeta_{J_K}^2:=\sum\limits_{E \in \partial K} \mathfrak{K}_h\left(\Big\|[\![ u_h]\!]\Big\|_{\L^2(E)}^2 + \Big\| \int_0^t K(t-\tau)[\![u_h(\tau)]\!]\mathrm{~d}\tau\Big\|_{\L^2(E)}^2\right).
\end{align}
\begin{remark}
	\bl{The memory term encapsulates information from the past history, and the estimator effectively retains and incorporates this information continuously over time.}
\end{remark}
\begin{lemma}\label{lemma4.1}
	Let $\tilde{u}$ be the exact solution of \eqref{3.4.1}, $u_h^c$ be the conforming part of approximated solution as defined in \eqref{3.confsol} and \eqref{3.appara} holds. Define 
	$\theta^c =  \tilde{u}-u_h^c$, there holds
	$$ |\!|\!| \theta^c|\!|\!|\leq C(\zeta+ \mathcal{F}), $$
	where, $\zeta$
	and $\mathcal{F}$ are the global estimator and data approximation term, respectively, as defined in \eqref{3.estimator}.
\end{lemma}
\begin{proof}
	An application of Lemma \ref{lemma3.2}, for $u_1 = \tilde{u}$, $u_2 = u_h^c$  
	with  $\chi = \tilde{u} -u_h^c,$ and the interpolation defined in \eqref{3.clm}, given as $( g_h,I_h \chi) - \eta\int_0^t K(t-\tau)N_{DG}(u_h(\tau), I_h\chi) \mathrm{~d}\tau= \mathcal{A}_{DG}(u_h, I_h\chi),$ yields 
	\begin{align}\label{3.sd7}
		\nonumber|\!|\!|{ \tilde{u}-u_h^c}|\!|\!|^2&\leq \tilde{\mathcal{A}}_{DG}(\tilde{u},\chi) -   \tilde{\mathcal{A}}_{DG}(u_h^c,\chi)\\&= \nonumber (g_h,\chi-I_h \chi)-   \tilde{\mathcal{A}}_{DG}(u_h,\chi-I_h \chi)+ N_{DG}(u_h,I_h\chi)\\&\quad\nonumber+ \eta\int_0^t K(t-\tau)N_{DG}(u_h(\tau),I_h \chi) \mathrm{~d}\tau +\tilde{\mathcal{A}}_{DG}(u_h,\chi) - \tilde{\mathcal{A}}_{DG}(u_h^c,\chi)  \\&\nonumber=(f-f_h,\chi-I_h \chi)+ (f_h,\chi-I_h \chi)-(\partial_t u_h, \chi-I_h \chi) -  \eta\int\limits_0^tK(t-\tau)(\nabla_h u_h, \nabla_h(\chi-I_h \chi))\mathrm{~d}\tau \\&\nonumber\quad - \nu a^c_{DG}(u_h,\chi-I_h \chi) - b_{DG}(u_h,u_h,\chi-I_h \chi) + (c(u_h),\chi-I_h \chi)+ N_{DG}(u_h,I_h\chi) \\&\quad+ \eta\int_0^t K(t-\tau)N_{DG}(u_h(\tau), I_h\chi) \mathrm{~d}\tau+\tilde{\mathcal{A}}_{DG}(u_h,\chi) - \tilde{\mathcal{A}}_{DG}(u_h^c,\chi).
	\end{align}
	Now, for the memory term $\int\limits_0^t K(t-\tau)(\nabla_h u_h(\tau),\nabla_h (\chi-I_h \chi)) \mathrm{~d}\tau$, using integration by parts, we have 
	\begin{align}\label{3.sd8}
		&\nonumber -\sum_{K\in \mathcal{T}_h}\int_0^t\int_{K}K(t-\tau)\nabla_h u_h\nabla_h (\chi-I_h \chi) \mathrm{~d}\tau\mathrm{~d}K\\&=  \sum_{K\in \mathcal{T}_h}\int_0^t\int_{K}K(t-\tau)\Delta u_h( \chi-I_h \chi) \mathrm{~d}\tau\mathrm{~d}K -  \sum_{K\in \mathcal{T}_h}\int_0^t\int_{\partial K} K(t-\tau)\frac{\partial u_h}{\partial \nu} (\chi-I_h \chi)  \mathrm{~d}\tau \mathrm{~d}S.
	\end{align}
	Substituting \eqref{3.sd8} back into \eqref{3.sd7} and using integration by parts for $a_{DG}^c(u_h,\chi-I_h\chi)$ term as done in \eqref{3.intpar}, yields
	\begin{align}\label{3.sd10}
		\nonumber&|\!|\!|{ \tilde{u}-u_h^c}|\!|\!|^2\\&\nonumber\leq (f-f_h,\chi-I_h \chi)+ (f_h,\chi-I_h \chi)-(\partial_t u_h, \chi-I_h \chi) +  \eta\int\limits_0^tK(t-\tau)(\Delta u_h, (\chi-I_h \chi))\mathrm{~d}\tau \\&\nonumber\quad+\nu (\Delta u_h, (\chi-I_h \chi))-  \sum_{K\in \mathcal{T}_h}\int_{\partial K}\frac{\partial u_h}{\partial \nu} (\chi-I_h \chi) \mathrm{~d}S 	-  \sum_{K\in \mathcal{T}_h}\int_0^t\int_{\partial K} K(t-\tau)\frac{\partial u_h}{\partial \nu} (\chi-I_h \chi)  \mathrm{~d}\tau \mathrm{~d}S \\&\nonumber\quad+ b_{DG}(u_h,u_h,\chi-I_h \chi) - (c(u_h),\chi-I_h \chi)+ N_{DG}(u_h,I_h\chi)+ \eta\int_0^t K(t-\tau)N_{DG}(u_h(\tau), I_h\chi) \mathrm{~d}\tau\\&\nonumber\quad+\tilde{\mathcal{A}}_{DG}(u_h,\chi) - \tilde{\mathcal{A}}_{DG}(u_h^c,\chi)
		\\&\nonumber=\sum_{K\in \mathcal{T}_h}\int_{K}\bigg(f_h - \partial_tu_h+\eta \int\limits_0^tK(t-\tau)\Delta u_h(\tau) \mathrm{~d}\tau+\nu\Delta u_h- \alpha u_h^{\delta}\sum\limits_{i=1}^d\frac{\partial u_h}{\partial x_i}
		+\beta u_h(1-u_h^{\delta})(u_h^{\delta}-\gamma)\bigg)\\&\nonumber\quad(\chi-I_h \chi) \mathrm{~d}K	+ (f-f_h,\chi-I_h \chi)  -\nu\sum_{K\in \mathcal{T}_h}\int_{\partial K}\frac{\partial u_h}{\partial \nu} (\chi-I_h \chi) \mathrm{~d}S + \nu N_{DG}(u_h,I_h\chi)	 \\&\quad\nonumber - \eta \sum_{K\in \mathcal{T}_h}\int_0^t\int_{\partial K} K(t-\tau)\frac{\partial u_h}{\partial \nu} (\chi-I_h \chi)  \mathrm{~d}\tau \mathrm{~d}S+ \eta\int_0^t K(t-\tau)N_{DG}(u_h(\tau), I_h\chi) \mathrm{~d}\tau\\&\quad+\tilde{\mathcal{A}}_{DG}(u_h,\chi) - \tilde{\mathcal{A}}_{DG}(u_h^c,\chi)
	\end{align}
	The initial component corresponds to the element-wise residual, denoted as $R_K$, while the memory term in \eqref{3.sd10} can be assessed utilizing the Cauchy-Schwarz inequality and the Cl\'ement interpolation \eqref{3.clmes2}, represented as
	\begin{align}\label{3.sd9}
		\nonumber&\bigg| \sum_{K\in \mathcal{T}_h}\int_0^t\int_{\partial K} K(t-\tau)\frac{\partial u_h}{\partial \nu} (\chi-I_h \chi)  \mathrm{~d}\tau \mathrm{~d}S\bigg| \\&= \bigg|\sum_{E\in \mathcal{E}}\int_0^t \int_EK(t-\tau) [\![\nabla_h u_h(\tau)]\!](\chi-I_h \chi) \mathrm{~d}\tau\mathrm{~d}S\bigg|\\&\nonumber\leq C \bigg(\sum_{E\in \mathcal{E}}h_E\| [\![K*\nabla_h u_h(s)]\!]\|_{\L^2(E)}^2\bigg)^{\frac{1}{2}}\bigg(\sum_{E\in \mathcal{E}}h_E^{-1} \| \chi-I_h \chi\|_{\L^2(E)}^2\bigg)^{\frac{1}{2}}\\& \leq C\bigg(\sum_{E\in \mathcal{E}} \zeta_{E_K}^2\bigg)^{\frac{1}{2}}|\!|\!| \chi|\!|\!|,
	\end{align}
	and  the term $\eta\int_0^t K(t-\tau)N_{DG}(u_h(\tau), \chi) \mathrm{~d}\tau$ can be estimated similar to \eqref{3.12} as
	\begin{align}
		\eta\int_0^t K(t-\tau)N_{DG}(u_h(\tau), I_h\chi) \mathrm{~d}\tau \leq C \bigg(\sum_{K\in \mathcal{T}_h} \zeta_{J_K}^2\bigg)^{\frac{1}{2}}|\!|\!| \chi|\!|\!|.
	\end{align}
	The estimation for the remaining terms can be directly derived from Lemma \ref{lemma3.6}.
\end{proof}
Let us now introduce the error indicator $\Theta$ for the semi-discrete scheme as
\begin{align}\label{4.theta}
	\Theta^2 = \|e(0)\|_{\L^2}^2 + \int_0^T\zeta^2 + \int_0^T\Theta_2^2 + \max\limits_{0\leq t\leq T}\Theta_3^2, 
\end{align}
where, 
$$\Theta_2^2 = \sum_{E\in\mathcal{E}_h}\|[\![\partial_t u_h]\!]\|_{\L^2(E)}, \qquad\Theta_3^2 = \sum_{E\in\mathcal{E}_h}\|[\![u_h]\!]\|_{\L^2(E)}.$$
and $\zeta$ is the global error estimator defined in the above lemma.
\begin{theorem}\label{3.theorem4.3}
	Let $u$ be the exact solution defined in \eqref{3.weakGBHE} and $u_h$ be the semi-discrete defined in \eqref{3.4.2}, then there exist a constant $C>0$ independent of $h,$ such that 
	\begin{align}
		\|e\|_{\L^{\infty}(0,T;\L^2(\Omega))}^2 \mathrm{~d}s + \int_0^T |\!|\!| e(t)|\!|\!|^2  \mathrm{~d}t \leq C \Theta,
	\end{align}
	where, $\Theta$ is defined in \eqref{4.theta}.
\end{theorem}
\begin{proof}
	Deriving from \eqref{3.weakGBHE} and \eqref{3.4.1}, we obtain:
	\begin{align}
		&\langle\partial_t(u(t)-u_h(t)),v\rangle +	\nu (\nabla (u(t)-\tilde{u}(t)),\nabla v)+\alpha \Big(b(u(t),u(t),v)-b(\tilde{u}(t),\tilde{u}(t),v)\Big) \\&\quad+ \eta\Big((K*\nabla u(\tau),\nabla v)- \int_0^t K(t-\tau)a_{DG}^c( u_h(\tau), v) \mathrm{~d}\tau\Big)-\beta\Big(\langle c(u(t)),v\rangle  - \langle c(\tilde{u}(t)),v\rangle\Big) =0,
	\end{align}
	for all $v\in \H_0^1(\Omega)$. 
	Consider $\rho = u - \tilde{u}$ and $\theta = \tilde{u} - u_h$. Utilizing the decomposition \eqref{3.confsol} and choosing $v = e^c = u - u_h^c$, we then proceed with:
	\begin{align}
		&(\partial_t e^c,e^c) +	\mathcal{\tilde{A}}_{DG}(u,e^c) - \mathcal{\tilde{A}}_{DG}(u^c_h,e^c) + \eta(K*\nabla_h e^c(\tau),\nabla_h e^c) \\&=(\partial_t u_h^r,e^c)+ \nu(\nabla_h \theta, \nabla_h e^c)+ \eta(K*\nabla_h u_h^r(\tau),\nabla_h e^c) 
		+ \alpha b(\tilde{u},\tilde{u},e^c) - \alpha b_{DG}(u_h,u_h,e^c)  \\&\quad+ \beta (c(\tilde{u}),e^c) - \beta ((c(u_h)),e^c) +
		a_{DG}^c(u_h-u^c_h,e^c) \\&\quad+ \alpha (b(u_h,u_h,e^c) -b(u_h^c,u_h^c,e^c) ) - (c(u_h) - c(u_h),e^c).
	\end{align}
	Using Lemma \ref{lemma3.2}, \ref{lemma3.6} and Cauchy-Schwarz inequality, we achieve
	\begin{align}\label{3.31}
		\nonumber&\frac{1}{2}\frac{\mathrm{d}}{\mathrm{d}t}\|e^c\|_{\L^2}^2 +  |\!|\!| e^c|\!|\!|^2 +  \eta(K*\nabla e^c(\tau),\nabla e^c)\\&\nonumber \leq \|\partial_t u_h^r\|_{\L^2}\|e^c\|_{\L^2}  + \nu \|\nabla_h \theta\|_{\L^2(\mathcal{T}_h)}\|\nabla_h e^c\|_{\L^2(\mathcal{T}_h)}+ \eta(K*\nabla_h u_h^r(\tau),\nabla_h e^c)\\&\quad 
		+ \alpha b(\tilde{u},\tilde{u},e^c) - \alpha b_{DG}(u_h,u_h,e^c)+ \beta (c(\tilde{u}),e^c) - \beta ((c(u_h)),e^c) + |\!|\!| u_h^r|\!|\!||\!|\!| e^c|\!|\!|.
	\end{align}
	Estimating the term $\alpha b(\tilde{u},\tilde{u},e^c) - \alpha b_{DG}(u_h,u_h,e^c)$ in a manner akin to \eqref{3.b1}, through the application of integration by parts and Taylor's formula, we obtain:
	\begin{align}
		\alpha	b(\tilde{u},\tilde{u},e^c) - \alpha b_{DG}(u_h,u_h,e^c) 
		& \leq \frac{2\alpha\delta}{(\delta+2)(\delta+1)}\sum_{i=1}^d \left(\tilde{u}^{\delta+1}-(u_h)^{\delta+1}, \frac{\partial e^c}{\partial{x_i}}\right)
		\\\nonumber
		& \leq \frac{2\alpha\delta}{(\delta+2)}\sum_{i=1}^d \left((\theta \tilde{u} + (1-\theta)u_h)^{\delta}(\tilde{u}-u_h), \frac{\partial e^c}{\partial{x_i}}\right)
		\\\nonumber
		&\leq  \frac{2^{\delta}\alpha\delta}{(\delta+2)}\bigg(\|\tilde{u}\|^{\delta}_{\L^{2(\delta+1)}}+\|u_h\|^{\delta}_{\L^{2(\delta+1)}}\bigg)\|\theta\|_{\L^{2(\delta+1)}}\|\nabla_h e^c\|_{\L^2(\mathcal{T}_h)}.
	\end{align}
	The reaction term can be approximated in a manner analogous to $J_2'$, as illustrated in equations \eqref{3.rc1} to \eqref{3.rc3}:
	\begin{align}
		&\beta (c(\tilde{u}),e^c) - \beta ((c(u_h)),e^c) \\&= -2\beta\gamma(\tilde{u}-u_h,e^c)+2\beta(1+\gamma)({\tilde{u}}^{\delta+1}-u_h^{\delta+1},e^c) -2\beta({\tilde{u}}^{2\delta+1}-u_h^{2\delta+1},e^c) \\& \leq 2\beta\gamma\|\theta\|_{\L^2}\|e^c\|_{\L^2} + 2^{\delta}\beta(1+\gamma)(\delta+1)\left(\|\tilde{u}\|^{(\delta+1)}_{\L^{2(\delta+1)}}+\|u_h\|^{(\delta+1)}_{\L^{2(\delta+1)}}\right)\|\theta\|_{\L^{2(\delta+1)}}\|e^c\|_{\L^2}\\&\quad+2^{2\delta}(2\delta+1)\beta\left(\|\tilde{u}\|_{\L^{2(\delta+1)}}^{\delta}+\|u_h\|_{\L^{2(\delta+1)}}^{\delta}\right)\|\theta\|_{\L^{2(\delta+1)}}\|e^c\|_{\L^{2(\delta+1)}}.
	\end{align}
	Substituting the above estimates in \eqref{3.31}, we achieve 
	\begin{align}\label{3.32}
		&\frac{1}{2}\frac{\mathrm{d}}{\mathrm{d}t}\|e^c\|_{\L^2}^2 +  |\!|\!| e^c|\!|\!|^2 +  \eta(K*\nabla e^c(\tau),\nabla e^c)\\& \leq \|\partial_t u_h^r\|_{\L^2}\|e^c\|_{\L^2} + \nu \|\nabla_h \theta\|_{\L^2(\mathcal{T}_h)}\|\nabla_h e^c\|_{\L^2(\mathcal{T}_h)}+ \eta(K*\nabla_h u_h^r(\tau),\nabla_h e^c)\\&\quad + \alpha b(\tilde{u},\tilde{u},e^c) - \alpha b_{DG}(u_h,u_h,e^c) + \beta (c(\tilde{u}),e^c) - \beta ((c(u_h)),e^c) + |\!|\!| u_h^r|\!|\!||\!|\!| e^c|\!|\!|\\& \leq \|\partial_t u_h^r\|_{\L^2}\|e^c\|_{\L^2}  + \nu \|\nabla_h \theta\|_{\L^2(\mathcal{T}_h)}\|\nabla_h e^c\|_{\L^2(\mathcal{T}_h)}+ \eta(K*\nabla_h u_h^r(\tau),\nabla_h e^c)
		\\&\quad + \frac{2^{\delta}\alpha\delta}{(\delta+2)}\bigg(\|\tilde{u}\|^{\delta}_{\L^{2(\delta+1)}}+\|u_h\|^{\delta}_{\L^{2(\delta+1)}}\bigg)\|\theta\|_{\L^{2(\delta+1)}}\|\nabla_h e^c\|_{\L^2(\mathcal{T}_h)}+ 2\beta\gamma\|\theta\|_{\L^2}\|e^c\|_{\L^2} \\&\quad+ 2^{\delta}\beta(1+\gamma)(\delta+1)\left(\|\tilde{u}\|^{\delta}_{\L^{2(\delta+1)}}+\|u_h\|^{\delta}_{\L^{2(\delta+1)}}\right)\|\theta\|_{\L^{2(\delta+1)}}\|e^c\|_{\L^2}\\&\quad+2^{2\delta}(2\delta+1)\beta\left(\|\tilde{u}\|_{\L^{2(\delta+1)}}^{\delta}+\|u_h\|_{\L^{2(\delta+1)}}^{\delta}\right)\|\theta\|_{\L^{2(\delta+1)}}\|e^c\|_{\L^{2(\delta+1)}}.
	\end{align}
	Integrating from $0$ to $T$, using positivity of the kernel,Young's inequality and the stability estimate yields 
	\begin{align}\label{3.33}
		&\|e^c(T)\|_{\L^2}^2  + C\int_0^T |\!|\!| e^c(t)|\!|\!|^2  \mathrm{~d}t\\& \leq \|e^c(0)\|_{\L^2}^2 + \int_0^T\|\partial_t u_h^r(t)\|_{\L^2}^2\mathrm{~d}t + \nu \int_0^T\|\nabla_h \theta(t)\|_{\L^2(\mathcal{T}_h)}^2\mathrm{~d}t+ \eta \int_0^T \|(K*\nabla_h u_h^r)(t)\|^2_{\L^2(\mathcal{T}_h)}\mathrm{~d}t
		\\&\quad + \frac{2^{2\delta}C\alpha^2\delta^2}{(\delta+2)^2}\int_0^T\bigg(\|\tilde{u}(t)\|^{2\delta}_{\L^{2(\delta+1)}}+\|u_h(t)\|^{2\delta}_{\L^{2(\delta+1)}}\bigg)\|\nabla_h\theta(t)\|_{\L^2(\mathcal{T}_h)}^2 \mathrm{~d}t\\&\quad + C\int_0^T\|\theta(t)\|_{\L^2}^2\mathrm{~d}t + \int_0^T 2^{4\delta}(2\delta+1)^2\beta^2\left(\|\tilde{u}\|_{\L^{2(\delta+1)}}^{2\delta} +\|u_h(t)\|_{\L^{2(\delta+1)}}^{2\delta}\right)\|\nabla_h\theta(t)\|_{\L^2(\mathcal{T}_h)}^2 \mathrm{~d}t\\&\quad + \int_0^T\bigg(1+\beta^2\gamma^2 + 2^{2\delta}\beta^2(1+\gamma)^2(\delta+1)^2\bigg(\|\tilde{u}(t)\|^{2(\delta+1)}_{\L^{2(\delta+1)}}+\|u_h(t)\|^{2(\delta+1)}_{\L^{2(\delta+1)}}\bigg)\bigg)\|e^c(t)\|_{\L^2}^2 \mathrm{~d}t
	\end{align}
	The memory term can be assessed by applying Young's convolution inequality as follows:
	\begin{align}
		\int_0^T \|(K*\nabla_h u_h^r)(t)\|^2_{\L^2(\mathcal{T}_h)}\mathrm{~d}s \leq \left(\int_0^T |K(t)| \mathrm{~d}t\right)^2\int_0^T|\!|\!| u_h^r(t)|\!|\!|^2 \mathrm{~d}t.
	\end{align}
	Once more, employing Young's and Gronwall inequalities, we obtain:
	\begin{align}\label{3.34}
		&\|e^c(T)\|_{\L^2}^2 + C\int_0^T |\!|\!| e^c(t)|\!|\!|^2  \mathrm{~d}t\\& \leq C\bigg(\|e^c(0)\|_{\L^2}^2 + \int_0^T\|\partial_t u_h^r(t)\|_{\L^2}^2\mathrm{~d}t + \int_0^T|\!|\!|\theta(t)|\!|\!|^2\mathrm{~d}t + \int_0^T|\!|\!| u_h^r(t)|\!|\!|^2 \mathrm{~d}t\bigg).
	\end{align}
	By applying the triangle inequality along with Theorem 3.2 and Lemma 4.2, we achieve the desired outcome.
\end{proof}
\section{A posteriori error estimates for the Fully discrete scheme}\setcounter{equation}{0}\label{sec5}
This section is dedicated to the fully discrete a posteriori estimation of the GBHE with memory. Two distinct time discretization approaches will be explored: one utilizing backward Euler discretization, and the other employing the Crank-Nicolson (CN) scheme for temporal discretization. Both methods will be coupled with the discontinuous Galerkin Finite Element Method (DGFEM) for spatial discretization. Our methodology closely follows that of the semi-discrete case, with additional emphasis on discrete time considerations and the possibility of mesh changes.

The time domain $[0,T]$ is divided into intervals, $0= t_0< t_1<\cdots t_k< \cdots < t_N =T,$ incorporating time stepping, $ \tau_k = t_k -t_{k-1}$. Here, $u_h^k$ represents the solution $u(\x,t)$ at time $t =t_k$, and $V_h^k$ denotes the discrete space corresponding to the partition $\mathcal{T}_{h,k}$. It is important to note that this discrete space may differ from the space $\mathcal{T}_{h,k-1}$ for $k\geq 1$, highlighting the potential for mesh changes.
\subsection{Backward Euler}
The fully discrete weak formulation using the backward Euler discretization for the time derivative term of the system (\ref{3.GBHE}) reads as: Given $u_h^{k-1}$, find $ u_h^k\in V_h^k$ such that
\begin{align}\label{2.ncweakformfd}
	\nonumber	\left(\frac{u_h^k-u_h^{k-1}}{\tau_k}, \chi\right) + \mathcal{A}_{DG}(u_h^k,\chi)+\eta \left(\sum_{j=1}^{k}\omega_{kj}\tau_ja_{DG}(u_h^{j}, \chi) \right)= (f^k,\chi),\\
	(u_h(0),\chi)=(u_h^0,\chi), \quad \forall ~~\chi \in V^k_h.
\end{align}
where, $\omega_{k j}=\frac{1}{\tau_k\tau_j } \int_{t_{k-1}}^{t_k} \int_{t_{j-1}}^{\min \left(t, t_j\right)} K(t-s) \mathrm{~d}s \mathrm{~d}t$, for $1\leq k\leq N$ and $f^k = (\Delta t)^{-1}\int_{t_{k-1}}^{t_k} f(s)  \mathrm{d}s$ .
It is important to note that the partition $\mathcal{T}_{h,k}$ may undergo either coarsening or refinement compared to $\mathcal{T}_{h,k-1}$, resulting in the discrete space $V_h^k$ having a different known term $u_h^{k-1}\in V_h^{k-1}$. To address this discrepancy, we introduce a transfer operator $\mathcal{I}^k: V_h^{k-1} \rightarrow V_h^k$, which facilitates the transfer of information such that $\mathcal{I}^k u_h^{k-1}\in V_h^{k}$. This transfer operator, for instance, could be the $\L^2$ projection or an interpolation operator. The modified general fully discrete formulation is then expressed as follows:
\begin{align}
	\nonumber	\left(\frac{u_h^k-\mathcal{I}^ku_h^{k-1}}{\tau_k}, \chi\right) + \mathcal{A}_{DG}(u_h^k,\chi)+\eta \left(\sum_{j=1}^{k}\omega_{kj}\tau_ja_{DG}(\mathcal{I}^{j+1}u_h^{j}, \chi) \right)= (f^k,\chi), \quad \forall~~ \chi\in V_h^n.
\end{align}
We define the time indicator $\Xi_k$ at each time $t_k$, as
\begin{align}\label{3.timeeff}
	\nonumber	\Xi_k = \tau_k \Big(&\|u_h^k-\mathcal{I}^ku_h^{k-1}\|_{\H^1(\mathcal{T}_{h,k})}+ h_E(\tau_k)^{-2}\|~\![\![\mathcal{I}^ku_h^{k-1}-u_h^{k-1}]\!]~\|_{
		\L^2(E)} \\&+ h_E(\tau_k)^{-2}\|~\![\![u_h^k - \mathcal{I}^ku_h^{k-1}]\!]~\|_{\L^2(E)} \Big),
\end{align}
To account for the memory term, the time estimator includes an additional term reflecting the error in discretizing the memory term at the current time step ($t = t_k$). This term arises due to the oscillations introduced by the kernel function.
\begin{align}\label{3.dom}
	\mathcal{K}_k^2 = 	\eta^2\sum_{k=1}^N \int_{t_{k-1}}^{t_k} \Big\|\left( \sum_{j=1}^{k}\frac{1}{\tau_k\tau_j } \int_{t_{k-1}}^{t_k} \int_{t_{j-1}}^{\min \left(t, t_j\right)} K(t-s)\tau_j\mathrm{~d}s \mathrm{~d}t-\sum_{j=1}^k\int_{t_{j-1}}^{\min \left(t, t_j\right)}K(t-s)\right) \nabla\tilde{u}_h^{j} \mathrm{~d}s\Big\|^2
\end{align}
Subsequently, we establish the accumulated time, spatial error indicator, and data oscillation as follows:
\begin{align}
	\label{3.timeestimator1}	\Xi^2 &= \sum_{k=1}^N \Xi_k^2  , \quad \Upsilon^2 = \sum_{k=1}^N\tau_k \Big(\Upsilon_k^2(u_h^k)+ \Upsilon_k^2(\mathcal{I}^ku_h^{k-1})\Big),\\\label{3.timeestimator2}\mathcal{K}^2 &= \sum\limits_{k=1}^N\mathcal{K}_k^2,\quad  \mathcal{F}_1^2 = \sum\limits_{k=1}^N\int_{t_{k-1}}^{t_{k}}\|f-f^k\|_{\L^2}^2. 
\end{align}
where the expression $\Upsilon$ contains the terms that form the estimator defined for the steady-state condition at a specific time step $t_k$,
such that 
$$\Upsilon^2_k(u_h^k)= \Upsilon^2_{K,k} + \Upsilon^2_{E,k} + \Upsilon^2_{J,k},$$ with,
\begin{align}
	&\Upsilon^2_{K,k} := h_K^2\Big(\|R_K^k\|_{\L^2(K)}^2\Big), \quad \Upsilon^2_{E,k} = \sum_{E\in  \mathcal{E}_h} h_E\Big(\|R_{E_1}^k\|_{\L^2(E)}^2 +\|R_{E_2}^k\|_{\L^2(E)}^2\Big),\\& \Upsilon^2_{J,k} := \sum\limits_{E \in  \mathcal{E}_h} \mathfrak{K}_h\Big\|\nu[\![  u_h^k ]\!]\Big\|_{\L^2(E)}^2 + \sum\limits_{E \in  \mathcal{E}_h} \mathfrak{K}_h\Big\| \eta\sum\limits_{j=1}^{k}\omega_{kj}\tau_k [\![\nabla \mathcal{I}^ju^j_h]\!]\Big\|_{\L^2(E)}^2,
\end{align}
and the element-wise and edge-wise residual is given as:
\begin{align}
	R_K^k &:= \Big\{{\frac{u_h^k - \mathcal{I}^ku_h^{k-1}}{\tau_k}}+f_h^k+\nu\Delta u^k_h- \alpha (u^k_h)^{\delta}\sum\limits_{i=1}^d\frac{\partial u^k_h}{\partial x_i} +\eta \sum_{j=1}^{k}\omega_{kj}\tau_j\Delta \mathcal{I}^{j+1}u_h^{j}\\&\qquad+\beta u^k_h(1-(u^k_h)^{\delta})((u_h^k)^{\delta}-\gamma)\Big\}\Big|_K,\\
	R_{E_1}^k&:= \begin{cases}\frac{1}{2} [\![(\nu\nabla_h u^k_h)\cdot\n]\!] \quad \text { for } E \in \mathcal{E}_h^i , \\
		0 ~\hspace{23mm}\text {for } E \in  \mathcal{E}_h^{\partial}. \end{cases}, \quad R_{E_2}^k:= \begin{cases}\frac{1}{2}\eta\sum\limits_{j=1}^{k}\omega_{kj}\tau_j [\![( \nabla_h \mathcal{I}^{j+1}u^j_h )\cdot\n]\!] \quad\text { for } E \in \mathcal{E}^i_h, \\
		0 ~\hspace{42mm}\text {for }E \in  \mathcal{E}_h^{\partial}. \end{cases}
\end{align}
We now introduce a linear interpolation $u_h(t)$, for each $t\in (t_{k-1},t_{k})$, as
\begin{align}\label{liapp}
	u_h(t) = l_{k-1}  \mathcal{I}^ku_h^{k-1}  + l_{k}u_h^k = \frac{t_k-t}{\tau_k}\mathcal{I}^ku_h^{k-1} + \frac{t- t_{k-1}}{\tau_k}u_h^k. 
\end{align}
where ${l_{k-1},l_k}$ is the standard linear interpolation defined on $[t_{k-1},t_k]$. For each time step $k$ we split the discrete solution as, $u_h^k = u_{h,c}^{k} + u_{h,r}^k$ and their linear approximation $u_{h,c}(t)$ and  $u_{h,r}(t)$ are given in the similar manner as in \eqref{liapp}. Rewriting the error $e = u - u_h = e_c- u_{h,r}$,  where $e_{c}= u-u_{h,c}$ and $\tilde{e} = u - \tilde{u}_h$. For $t\in (t_{k-1}, t_k)$, we define 
\begin{align}\label{3.tauk}
	u_{\tau_k}(t) = l_{k-1} u_h^{k-1}  +  l_{k}u_h^k = \frac{t_k-t}{\tau_k}u_h^{k-1} + \frac{t- t_{k-1}}{\tau_k}u_h^k, 
\end{align}
so, we have that 
$$\partial_t u_{h}(t) = \frac{1}{\tau_k}(u_h^k-\mathcal{I}^ku_h^{k-1}), \quad\partial_t u_{\tau_k}(t) = \frac{1}{\tau_k}(u_h^k-u_h^{k-1}),$$
and we consider the problem of finding $\tilde{u}^k \in \H_0^1(\Omega),$ such that 
\begin{align}\label{3.fdapp}
	(\partial_t \tilde{u}_h(t), v) +  \tilde{\mathcal{A}}_{DG}(\tilde{u}^k,v)+ \eta \Big(\sum_{j=1}^{k}\omega_{kj}\tau_ja_{DG}^c(\tilde{u}_h^j, v) \Big) =(f^k,v), \quad \forall ~~v \in \H_0^1(\Omega), 
\end{align}

\begin{lemma}\label{3.Lemma5.1}
	The following estimates holds:
	\begin{align}
		&\|\tilde{e}_c(t_N)\|_{\L^2}^2 + \int_0^{t_N}|\!|\!| \tilde{e}_c(t)|\!|\!|^2 \mathrm{d}t \\&\leq C(\Xi^2 + \Upsilon^2+	\mathcal{F}_1^2+\mathcal{K}^2 ) +\|\tilde{e}_c(0)\|_{\L^2}^2 + \sum_{k=1}^{N-1}\Big(\|u(t_k)- \mathcal{I}^{k+1}\tilde{u}_{h,c} ^k\|_{\L^2}^2-\|\tilde{e}_c(t_k)\|_{\L^2}^2\Big),
	\end{align}
	where $\mathcal{F}_1^2$ and $\mathcal{K}^2$ are defined in \eqref{3.timeestimator2}.
\end{lemma}
\begin{proof}
	By subtracting \eqref{3.weakGBHE} from \eqref{3.fdapp}, we obtain:
	\begin{align}
		(\partial_t(u - \tilde{u}_h), v) +\tilde{\mathcal{A}}_{DG}(u,v) -  \tilde{\mathcal{A}}_{DG}(\tilde{u}^k,v) + \eta(( K*\nabla u),\nabla v) - \eta \Big(\sum_{j=1}^{k}\omega_{kj}\tau_ja_{DG}^c(\tilde{u}_h^{j}, v) \Big)  =(f -f^k,v),
	\end{align}
	With the notation established previously, let's choose $v = \tilde{e}_c = u- \tilde{u}_{h,c}$. We then have:
	\begin{align*}
		&(\partial_t\tilde{e}_c, \tilde{e}_c) + \tilde{\mathcal{A}}_{DG}(u,\tilde{e}_c) - \tilde{\mathcal{A}}_{DG}(\tilde{u}_{h,c},\tilde{e}_c)  + \eta(( K*\nabla \tilde{e}_c),\nabla  \tilde{e}_c) =   (f -f^{k},\tilde{e}_c)-(\partial_t\tilde{u}_{h,r}, \tilde{e}_c) \\&+ \eta \Big(\sum_{j=1}^{k}\omega_{kj}\tau_ja_{DG}^c(\tilde{u}_h^{j}, \tilde{e}_c) \Big)  - \eta(( K*\nabla\tilde{u}_{h,c}),\nabla  \tilde{e}_c) +  \tilde{\mathcal{A}}_{DG}(\tilde{u}^k,\tilde{e}_c) - \tilde{\mathcal{A}}_{DG}(\tilde{u}_{h,c},\tilde{e}_c).
	\end{align*}
	Leveraging Lemmas \ref{lemma3.2} and \ref{lemma3.6}, along with the Cauchy-Schwarz inequality, yields:
	\begin{align*}
		&\frac{1}{2}\frac{\mathrm{d}}{\mathrm{d}t}\|\tilde{e}_c\|_{\L^2}^2 +  |\!|\!| \tilde{e}_c|\!|\!|^2 + \eta(( K*\nabla \tilde{e}_c),\nabla  \tilde{e}_c)\\ & \leq C \Big(\|f-f^k\|_{\L^2} + \|\partial_t\tilde{u}_{h,r} \|_{\L^2} + |\!|\!|  \tilde{u}^k - \tilde{u}_{h,c}|\!|\!|  \Big)|\!|\!| \tilde{e}_c|\!|\!| +  \eta\| K*\nabla\tilde{u}_{h,r}\|_{\L^2}   |\!|\!| \tilde{e}_c|\!|\!|  \\&\quad	+ \eta \Big\|\sum_{j=1}^{k}\omega_{kj}\tau_j\nabla\tilde{u}_h^{j}-( K*\nabla\tilde{u}_{h})\Big\||\!|\!| \tilde{e}_c|\!|\!|.
	\end{align*}
	Note that the kernel function $K(t)$ is positive. By applying Young's inequality and integrating the resulting expression over the time interval $[0, t_N]$, we obtain the following inequality:
	\begin{align}\label{3.68}
		&\nonumber\|\tilde{e}_c(t_N)\|_{\L^2}^2 + \int_0^{t_N}|\!|\!| \tilde{e}_c(t)|\!|\!|^2 \mathrm{d}t \\&\nonumber\leq \|\tilde{e}_c(0)\|_{\L^2}^2 + \sum_{k=1}^N \int_{t_{k-1}}^{t_k} \Big( \|f-f^k\|^2_{\L^2}+ \|\partial_t\tilde{u}_{h,r} \|_{\L^2}^2  +|\!|\!|  \tilde{u}^k - \tilde{u}_{\tau_k} |\!|\!|^2 +   |\!|\!|  \tilde{u}_{\tau_k}- \tilde{u}_{h,c}|\!|\!|^2\Big) \\&\nonumber\quad +\frac{1}{2}\sum_{k=1}^{N-1}\Big(\|u(t_k)- \mathcal{I}^{k+1}\tilde{u}_{h,c} ^k\|_{\L^2}^2-\|\tilde{e}_c(t_k)\|_{\L^2}^2\Big)+ \eta^2\sum_{k=1}^N \int_{t_{k-1}}^{t_k} \| K*\nabla\tilde{u}_{h,r}\|_{\L^2}^2  \\&\quad+ \sum_{k=1}^N \int_{t_{k-1}}^{t_k} \eta^2 \Big\|\sum_{j=1}^{k}\omega_{kj}\tau_j\nabla\tilde{u}_h^{j}-( K*\nabla\tilde{u}_{h}^j)\Big\|^2+ \sum_{k=1}^N \int_{t_{k-1}}^{t_k} \eta^2 \Big\|( K*\nabla\tilde{u}_{h}^j)-( K*\nabla\tilde{u}_{h}(s)\Big\|^2.
	\end{align}
	Recalling the definition of $\omega_{kj}$, we find:
	\begin{align}\label{3.amem3}
		&\nonumber\eta^2\sum_{k=1}^N \int_{t_{k-1}}^{t_k}  \Big\|\sum_{j=1}^{k}\omega_{kj}\tau_j\nabla\tilde{u}_h^{j}-( K*\nabla\tilde{u}_{h}^j)\Big\|^2 \\&\nonumber= \eta^2\sum_{k=1}^N \int_{t_{k-1}}^{t_k} \Big\|\sum_{j=1}^{k}\frac{1}{\tau_k\tau_j } \int_{t_{k-1}}^{t_k} \int_{t_{j-1}}^{\min \left(t, t_j\right)} K(t-s)\tau_j\nabla\tilde{u}_h^{j} \mathrm{~d}s \mathrm{~d}t-\sum_{j=1}^k\int_{t_{j-1}}^{\min \left(t, t_j\right)}K(t-s)\nabla\tilde{u}_{h}^j \mathrm{~d}s\Big\|^2\\&\leq \eta^2\sum_{k=1}^N \int_{t_{k-1}}^{t_k} \Big\|\left( \sum_{j=1}^{k}\frac{1}{\tau_k\tau_j } \int_{t_{k-1}}^{t_k} \int_{t_{j-1}}^{\min \left(t, t_j\right)} K(t-s)\tau_j\mathrm{~d}s \mathrm{~d}t-\sum_{j=1}^k\int_{t_{j-1}}^{\min \left(t, t_j\right)}K(t-s)\right) \nabla\tilde{u}_h^{j} \mathrm{~d}s\Big\|^2.
	\end{align}
	Additionally, by employing the definition of $\tilde{u}_h$, applying Young's convolution inequality, and interchanging the summation, we achieve:
	\begin{align}\label{3.amem}
		\nonumber	\eta^2 \sum_{k=1}^N \int_{t_{k-1}}^{t_k} \Big\|( K*\nabla\tilde{u}_{h}^j)-( K*\nabla\tilde{u}_{h}(s))\Big\|^2&=\eta^2  \sum_{k=1}^N \int_{t_{k-1}}^{t_k} \Big\|\sum_{j=1}^k\int_{t_{j-1}}^{\min \left(t, t_j\right)}K(t-s)\left(\nabla\tilde{u}_{h}^j-\nabla\tilde{u}_{h}(s) \right) \mathrm{~d}s\Big\|^2
		\\\nonumber&\leq \eta^2  \sum_{k=1}^N \int_{t_{k-1}}^{t_k} \left(\int_0^TK(t-s)\mathrm{~d}s\right)^2\left( \sum_{j=1}^k\int^{t_j}_{t_{j-1}}\| \nabla\tilde{u}_{h}^j-\nabla\tilde{u}_{h}(s) \|\right) ^2\\\nonumber&\leq \eta^2  \sum_{k=1}^N \int_{t_{k-1}}^{t_k} \left(\int_0^TK(t-s)\mathrm{~d}s\right)^2\left( \sum_{j=1}^k\tau_j |\!|\!| \tilde{u}_{h}^j-\mathcal{I}^j\tilde{u}_{h}^{j-1} |\!|\!|\right) ^2 \\\nonumber&\leq C\eta^2  \sum_{k=1}^N \int_{t_{k-1}}^{t_k}\tau_j\sum_{j=1}^k|\!|\!| \tilde{u}_{h}^j-\mathcal{I}^j\tilde{u}_{h}^{j-1} |\!|\!|^2\\&\nonumber\leq C\eta^2  \sum_{k=1}^N \sum_{j=1}^k\tau_k\tau_j|\!|\!| \tilde{u}_{h}^j-\mathcal{I}^j\tilde{u}_{h}^{j-1} |\!|\!|^2\\&\leq C\eta^2  \sum_{k=1}^N \tau_k|\!|\!| \tilde{u}_{h}^k-\mathcal{I}^k\tilde{u}_{h}^{k-1} |\!|\!|^2,
	\end{align}
	and  finally the  term $\| K*\nabla\tilde{u}_{h,r}\|_{\L^2}^2$ can be estimated using the regularity $K\in \L^1(0,T)$ as
	\begin{align}\label{3.amem5}
		\eta^2\| K*\nabla\tilde{u}_{h,r}\|_{\L^2}^2 \leq \eta^2\left(\int_0^tK(t)\mathrm{~d}t\right)^2  |\!|\!| \tilde{u}_{h,r} |\!|\!|^2\leq C |\!|\!| \tilde{u}_{h,r} |\!|\!|^2. 
	\end{align}
	By combining the estimates \eqref{3.amem3}-\eqref{3.amem5} and substituting them into \eqref{3.71}, we obtain:
	\begin{align}\label{3.71}
		&\nonumber\|\tilde{e}_c(t_N)\|_{\L^2}^2 + \int_0^{t_N}|\!|\!| \tilde{e}_c(t)|\!|\!|^2 \mathrm{d}t \\&\nonumber\leq \|\tilde{e}_c(0)\|_{\L^2}^2 +C \sum_{k=1}^N \int_{t_{k-1}}^{t_k} \Big( \|f-f^k\|^2_{\L^2}+ \|\partial_t\tilde{u}_{h,r} \|_{\L^2}^2  +|\!|\!|  \tilde{u}^k - \tilde{u}_{\tau_k} |\!|\!|^2 + |\!|\!| \tilde{u}_{h,r} |\!|\!|^2+   |\!|\!|  \tilde{u}_{\tau_k}- \tilde{u}_{h,c}|\!|\!|^2\Big) \\&\nonumber\quad +\frac{1}{2}\sum_{k=1}^{N-1}\Big(\|u(t_k)- \mathcal{I}^{k+1}\tilde{u}_{h,c} ^k\|_{\L^2}^2-\|\tilde{e}_c(t_k)\|_{\L^2}^2\Big) + C\eta^2  \sum_{k=1}^N \tau_k|\!|\!| \tilde{u}_{h}^k-\mathcal{I}^k\tilde{u}_{h}^{k-1} |\!|\!|^2\\&\nonumber\quad +\eta^2\sum_{k=1}^N \int_{t_{k-1}}^{t_k} \Big\|\left( \sum_{j=1}^{k}\frac{1}{\tau_k\tau_j } \int_{t_{k-1}}^{t_k} \int_{t_{j-1}}^{\min \left(t, t_j\right)} K(t-s)\tau_j\mathrm{~d}s \mathrm{~d}t-\sum_{j=1}^k\int_{t_{j-1}}^{\min \left(t, t_j\right)}K(t-s)\right) \nabla\tilde{u}_h^{j} \mathrm{~d}s\Big\|^2.
	\end{align}
	Recalling the definition of $\tilde{u}_{\tau_k}$ and applying the triangle inequality, we get:
	\begin{align}
		\int_{t_{k-1}}^{t_k} |\!|\!|  \tilde{u}^k - \tilde{u}_{\tau_k} |\!|\!|^2 \leq \tau_k |\!|\!| \tilde{u}^k - \tilde{u}^{k-1} |\!|\!|^2 \leq  \Upsilon_k^2 + \tau_k |\!|\!| \tilde{u}^k - \tilde{u}_h^{k} |\!|\!|^2 + \tau_k |\!|\!| \tilde{u}^{k-1} - \mathcal{I}^k\tilde{u}^{k-1} |\!|\!|^2.
	\end{align}
	In light of Theorem \ref{3.thm3.7} and the previous findings, we can conclude that:
	\begin{align}
		\|\tilde{e}_c(t_N)\|_{\L^2}^2 + \int_0^{t_N}|\!|\!| \tilde{e}_c(t)|\!|\!|^2 \mathrm{d}t \leq &C(\Xi^2 + \Upsilon^2 + 	\mathcal{F}_1^2+\mathcal{K}^2 ) +\|\tilde{e}_c(0)\|_{\L^2}^2\\
		&+ \sum_{k=1}^{N-1}\Big(\|u(t_k)- \mathcal{I}^{k+1}\tilde{u}_{h,c} ^k\|_{\L^2}^2-\|\tilde{e}_c(t_k)\|_{\L^2}^2\Big).
	\end{align}
\end{proof}
\begin{theorem}\label{3.Theorem5.2}
	Let $u$ and $u_h$ represent the exact solution and the discrete solution, respectively, of the equation \eqref{3.GBHE}. Consider $\Xi$ and $\Upsilon$ as the a posteriori error estimators defined in \eqref{3.timeestimator1}. The following reliability estimator is established:
	\begin{align}
		\|e(t_N)\|_{\L^2}^2 + \int_0^{t_N}|\!|\!| e(t)|\!|\!|^2 \mathrm{d}t &\leq C\Big(\Xi^2 + \Upsilon^2+	\mathcal{F}_1^2 +\mathcal{K}^2 +\|\tilde{e}_c(0)\|_{\L^2}^2 + \sum_{k=1}^{N-1}\|u_{h,r}^k- \mathcal{I}^{k+1}\tilde{u}_{h,r} ^k\|_{\L^2}^2\Big), 
	\end{align}
	\bl{where 	$\mathcal{F}_1^2$ and $\mathcal{K}^2$ are defined in \eqref{3.timeestimator2}.}
\end{theorem}

\begin{proof}
	Utilizing the decomposition $u_h^k = u_{h,c}^k +u_{h,r}^k$ and the identity \cite[(5.59)-(5.60)]{EOJ} given by
	\begin{align}
		\|u(t_k)- \mathcal{I}^{k+1}\tilde{u}_{h,c} ^k\|_{\L^2}^2-\|\tilde{e}_c(t_k)\|_{\L^2}^2 = \|u_{h,r}^k - \mathcal{I}^{k+1}u^k_{h,r}\|_{\L^2}^2 + \langle u^k_{h,r} - \mathcal{I}^{k+1}u_{h,r}^k, \tilde{e}_c(t_k)\rangle.
	\end{align}
	and invoking Lemma \ref{3.Lemma5.1}, this proof follows a methodology analogous to that of Theorem \ref{3.theorem4.3}.
\end{proof}	
\subsection{Crank–Nicolson Scheme}
Expanding on the previously introduced notations, we extend our framework to include the Crank-Nicolson (CN) Scheme in the temporal domain. The following notations are utilized for a comprehensive representation:
\begin{align*}
	\overline{\partial}u^n = \frac{u^n - u^{n-1}}{\tau_k}, \quad	t_{n-\frac{1}{2}} = \frac{t_n + t_{n-1}}{2}, \quad u^{n-\frac{1}{2}} := \frac{u(t_{n})+u(t_{n-1})}{2}.
\end{align*}
The discretization of the memory term can be implemented using the CN scheme as follows:
\begin{align}\label{3.CNmem}
	J(\psi) &= \int_0^{t} K(t-s)\psi(s) \mathrm{d}s 
	= \frac{1}{2} \sum_{j=1}^k \omega_{kj} \tau_j\psi^{j-\frac{1}{2}} ,
\end{align}
where $\omega_{k j}=\frac{1}{\tau_j\tau_k} \int_{t_{k-1}}^{t_k} \int_{t_{j-1}}^{\min \left(t, t_j\right)} K(t-s) \mathrm{~d}s \mathrm{~d}t$, for $1\leq k\leq N$ and 
$\psi^{j-\frac{1}{2}} = \frac{\psi(t_{j})+ \psi(t_{j-1})}{2}  $  in $(t_{j-1},t_{j})$.
The fully discretized weak formulation of the system (\ref{3.GBHE}) in this context can be expressed as follows: Given $u_h^{k-1}$, seek $u_h^k \in V_h^k$ such that
\begin{align}\label{2.ncweakformfdCN}
	\nonumber	\left(\frac{u_h^k-u_h^{k-1}}{\tau_k}, \chi\right) + \mathcal{A}_{DG}(u_h^{k-\frac{1}{2}},\chi)+\eta \left(\sum_{j=1}^{k}\omega_{kj}\tau_ja_{DG}(u_h^{j-\frac{1}{2}}, \chi) \right)= (f^{k-\frac{1}{2}},\chi),\\
	(u_h(0),\chi)=(u_h^0,\chi), \quad \forall~~ \chi \in V^k_h.
\end{align}
where, $\omega_{k j}$ is as defined in \eqref{3.CNmem}, for $1\leq k\leq N$,  and 
\begin{align}
	\mathcal{A}_{DG}(u_h^{k-\frac{1}{2}},v) &= \frac{\nu}{2} a_{DG}\left(u_h^k+u_h^{k-1},v\right) +  \frac{\alpha }{2}\left(b_{DG}(u_h^k,u_h^k,v) + b_{DG}(u_h^{k-1},u_h^{k-1},v)\right)\\&\quad+ \frac{\beta}{2}\left(c(u_h^k)+ c(u_h^{k-1}),v\right)
\end{align}
Again, the partition $\mathcal{T}_{h,k}$ may undergo either coarsening or refinement compared to $\mathcal{T}_{h,k-1}$. To account for this, we define the transfer operator $\mathcal{I}^k: V_h^{k-1} \rightarrow V_h^k$. The general fully discrete formulation is then expressed as: 
\begin{align}
	\nonumber	&\left(\frac{u_h^k-\mathcal{I}^ku_h^{k-1}}{\tau_k}, \chi\right) + \mathcal{A}_{DG}(u_h^{k-\frac{1}{2}},\chi)	+\eta \left(\sum_{j=1}^{k}\omega_{kj}\tau_ja_{DG}(\mathcal{I}^{j+1}u_h^{j-\frac{1}{2}}, \chi) \right)= (f^{k-\frac{1}{2}},\chi), \quad \forall~~ \chi\in V_h^n.
\end{align}
Similar to the backward Euler discretization case, the time indicator $\Xi_k$ at each time $t_k$ is given by
$$\Xi_k = \tau_k \Big(\|u_h^k-\mathcal{I}^ku_h^{k-1}\|_{\H^1(\mathcal{T}_{h,k})}+ h_E(\tau_k)^{-2}\|~\![\![\mathcal{I}^ku_h^{k-1}-u_h^{k-1}]\!]~\|_{
	\L^2(E)} + h_E(\tau_k)^{-2}\|~\![\![u_h^k - \mathcal{I}^ku_h^{k-1}]\!]~\|_{\L^2(E)} \Big),$$
and the data oscillation term corresponding to memory is expressed as
\begin{align}\label{3.dom6}
	\mathcal{K}_k^2 = 	\eta^2\sum_{k=1}^N \int_{t_{k-1}}^{t_k} \Big\|\left( \sum_{j=1}^{k}\frac{1}{\tau_k\tau_j } \int_{t_{k-1}}^{t_k} \int_{t_{j-1}}^{\min \left(t, t_j\right)} K(t-s)\tau_j\mathrm{~d}s \mathrm{~d}t-\sum_{j=1}^k\int_{t_{j-1}}^{\min \left(t, t_j\right)}K(t-s)\right) \nabla\tilde{u}_h^{j-\frac{1}{2}} \mathrm{~d}s\Big\|^2.
\end{align}
Hence, the definitions of the accumulated time and spectral error indicators are as follows:
\begin{align}
	\label{3.timeestimatorCN1}	\Xi^2 &= \sum_{k=1}^n \Xi_k^2, \quad \Upsilon^2 = \sum_{k=1}^N\tau_k \Big(\Upsilon_k(u_h^k)+ \Upsilon_k(\mathcal{I}^ku_h^{k-1})\Big), \\ \label{3.timeestimatorCN2}\mathcal{K}^2 &= \sum\limits_{k=1}^N\mathcal{K}_k^2, \quad \mathcal{F}_2^2 = \sum\limits_{k=1}^N\int_{t_{k-1}}^{t_{k}}\|f-f^{k-\frac{1}{2}}\|_{\L^2}^2.
\end{align}
where the terms within $\Upsilon$ constitute the estimator defined in the steady case at a given time step $t_k$, such that
$$\Upsilon_k^2(u_h^k)= \Upsilon^2_{K,k} + \Upsilon^2_{E,k} + \Upsilon^2_{J,k},$$ where,
\begin{align}
	&\Upsilon^2_{K,k} := h_K^2\Big(\|R_K^k\|_{\L^2(K)}^2\Big), \quad \Upsilon^2_{E,k} = \sum_{E\in \partial K} h_E(\|R_{E_1}^k\|_{\L^2(E)}^2+\|R_{E_2}^k\|_{\L^2(E)}^2),\\&\quad \Upsilon^2_{J,k} := \sum\limits_{E \in  \mathcal{E}_h} \frac{\mathfrak{K}_h}{2}\Big\|[\![ \nu (u_h^k + u_h^{k-1})]\!]\Big\|_{\L^2(E)}^2 + \sum\limits_{E \in  \mathcal{E}_h}\mathfrak{K}_h\Big\| \eta\sum\limits_{j=1}^{k}\omega_{kj} \tau_k [\![ \nabla_h \mathcal{I}^{j+1}u^{j-\frac{1}{2}}_h]\!]\Big\|_{\L^2(E)}^2, 
\end{align}
and the element-wise and edge-wise residual is given as:
\begin{align}
	R_K^k &:= \Big\{{\frac{u_h^k - \mathcal{I}^ku_h^{k-1}}{\tau_k}}+f_h^{k-\frac{1}{2}}+\nu\Delta u^{k-\frac{1}{2}}_h- \frac{\alpha}{2} \left( (u^k_h)^{\delta}\sum\limits_{i=1}^d\frac{\partial u^k_h}{\partial x_i} + (u^{k-1}_h)^{\delta}\sum\limits_{i=1}^d\frac{\partial u^{k-1}_h}{\partial x_i} \right)\\&\hspace{7mm}+ \eta\sum_{j=1}^{k}\omega_{kj} \tau_ja_{DG}(\Delta \mathcal{I}^{j+1}u_h^{j}, \chi)\\&\hspace{7mm}+\frac{\beta }{2}\left(u^k_h(1-(u^k_h)^{\delta})((u_h^k)^{\delta}-\gamma) + u^{k-1}_h(1-(u^{k-1}_h)^{\delta})((u_h^{k-1})^{\delta}-\gamma)\right)\Big\}\Big|_K,\\
	R_{E_1}^k&:= \begin{cases}\frac{1}{4} \big[\!\big[\big(\nu\nabla_h (u^k_h + u^{k-1}_h) \cdot\n\big]\!\big] \quad \text { for } E \in \mathcal{E}^i_h, \\
		0 ~\hspace{37mm}\text {for } E \in  \mathcal{E}_h^{\partial}, \end{cases}
	\\ R_{E_2}^k&:= \begin{cases}\frac{1}{4} \eta\sum\limits_{j=1}^{k}\omega_{kj}\tau_j \big[\!\big[\big( \nabla_h \mathcal{I}^{j+1}u^{j-\frac{1}{2}}_h\big)\cdot\n\big]\!\big] \quad \text { for } E \in \mathcal{E}^i_h, \\
		0 ~\hspace{49mm}\text {for } E \in  \mathcal{E}_h^{\partial}. \end{cases}
\end{align}
For the Crank- Nicolson scheme, the interpolation \eqref{3.tauk} can be rewritten as,
for each $t\in (t_{k-1},t_{k})$, as
\begin{align}\label{3.CNliapp}
	u_{\tau_k}(t) = l_{k-1}  u_h^{k-1}  + l_{k}u_h^k &= u_h^{k-\frac{1}{2}} + (t-t_{n-\frac{1}{2}})\overline{\partial}u_h^k,
\end{align}
for $0\leq k\leq N$.
In the preceding scenario, we address the task of identifying $\tilde{u}^{k-\frac{1}{2}} \in \H_0^1(\Omega)$, aiming to accomplish the following:
\begin{align}\label{3.CNfdapp}
	(\partial_t \tilde{u}_h(t), v) + \tilde{\mathcal{A}}_{DG}(\tilde{u}^{k-\frac{1}{2}},v)+ \eta \Big(\sum_{j=1}^{k}\omega_{kj}\tau_j a_{DG}(\tilde{u}_h^{j-\frac{1}{2}}, v) \Big) =(f^{k-\frac{1}{2}},v), \quad \forall ~~ v\in \H_0^1(\Omega). 
\end{align}

\begin{lemma}\label{3.CNLemma5.1}
	The following estimates holds:
	\begin{align}
		&\|\tilde{e}_c(t_N)\|_{\L^2}^2 + \int_0^{t_N}|\!|\!| \tilde{e}_c(t)|\!|\!|^2 \mathrm{d}t \\&\leq C(\Xi^2 + \Upsilon^2 +\mathcal{F}_2^2+\mathcal{K}^2) +\|\tilde{e}_c(0)\|_{\L^2}^2 + \sum_{k=1}^{N-1}\Big(\|u(t_k)- \mathcal{I}^{k+1}\tilde{u}_{h,c} ^k\|_{\L^2}^2-\|\tilde{e}_c(t_k)\|_{\L^2}^2\Big),
	\end{align}
	where $\mathcal{F}_2^2$ and $\mathcal{K}^2$ are defined in \eqref{3.timeestimatorCN2}
\end{lemma}
\begin{proof}
	By deducting \eqref{3.weakGBHE} from \eqref{3.CNfdapp}, we obtain
	\begin{align}
		&\partial_t(u - \tilde{u}_h), v) +\tilde{ \mathcal{A}}_{DG}(u,v) -  \tilde{\mathcal{A}}_{DG}(\tilde{u}^{k-\frac{1}{2}},v) + \eta(( K*\nabla u)(s),\nabla v) \\&- \eta \Big(\sum_{j=1}^{k}\omega_{kj}\tau_ja_{DG}(\tilde{u}^{j-\frac{1}{2}}, v) \Big)  =(f -f^{k-\frac{1}{2}},v),
	\end{align} 
	Using the above notations, and choosing $v = \tilde{e}_c,$ we have
	\begin{align*}
		&(\partial_t\tilde{e}_c, \tilde{e}_c) + \tilde{\mathcal{A}}_{DG}(u,\tilde{e}_c) - \tilde{\mathcal{A}}_{DG}(\tilde{u}_{h,c},\tilde{e}_c)  + \eta(( K*\nabla \tilde{e}_c)(s),\nabla  \tilde{e}_c) =   (f -f^{k-\frac{1}{2}},\tilde{e}_c)-(\partial_t\tilde{u}_{h,r}, \tilde{e}_c) \\&+ \eta \Big(\sum_{j=1}^{k}\omega_{kj}\tau_ja_{DG}(\tilde{u}^{j-\frac{1}{2}}, \tilde{e}_c) \Big)  - \eta(( K*\nabla\tilde{u}_{h,c})(s),\nabla  \tilde{e}_c)  +  \tilde{\mathcal{A}}_{DG}(\tilde{u}^{k-\frac{1}{2}},\tilde{e}_c) - \tilde{\mathcal{A}_{DG}}(\tilde{u}_{h,c},\tilde{e}_c).
	\end{align*}
	Utilizing Lemma \ref{lemma3.2}, along with Lemma \ref{lemma3.6} and the Cauchy-Schwarz inequality, yields:
	\begin{align*}
		&\frac{1}{2}\frac{\mathrm{d}}{\mathrm{d}t}\|\tilde{e}_c\|_{\L^2}^2 +  |\!|\!| \tilde{e}_c|\!|\!|^2+ \eta(( K*\nabla \tilde{e}_c),\nabla  \tilde{e}_c)  \\&\leq C\Big(\|f-f^{k-\frac{1}{2}}\|_{\L^2} + \|\partial_t\tilde{u}_{h,r} \|_{\L^2} + |\!|\!|  \tilde{u}^{k-\frac{1}{2}} - \tilde{u}_{h,c}|\!|\!| \Big)|\!|\!| \tilde{e}_c|\!|\!| +  \eta\| K*\nabla\tilde{u}_{h,r}\|_{\L^2}   |\!|\!| \tilde{e}_c|\!|\!|  \\&\quad	+ \eta \Big\|\sum_{j=1}^{k}\omega_{kj}\tau_j\nabla\tilde{u}_h^{j-\frac{1}{2}}-( K*\nabla\tilde{u}_{h})\Big\||\!|\!| \tilde{e}_c|\!|\!|.
	\end{align*}
	Employing Young's inequality, integrating over the time interval $t\in[0,t_N]$, and leveraging the positivity of the kernel $K(t)$, we attain:
	\begin{align}
		&\|\tilde{e}_c(t_N)\|_{\L^2}^2 + \int_0^{t_N}|\!|\!| \tilde{e}_c(t)|\!|\!|^2 \mathrm{d}t \\&\leq \|\tilde{e}_c(0)\|_{\L^2}^2 + \sum_{k=1}^N \int_{t_{k-1}}^{t_k} \Big( \|f-f^{k-\frac{1}{2}}\|^2_{\L^2}+ \|\partial_t\tilde{u}_{h,r} \|_{\L^2}^2  +|\!|\!|  \tilde{u}^{k-\frac{1}{2}} - \tilde{u}_{\tau_k} |\!|\!|^2 + |\!|\!|   \tilde{u}_{\tau_k} -\tilde{u}_{h,c} |\!|\!|^2\Big) \\&\quad +\frac{1}{2}\sum_{k=1}^{N-1}\Big(\|u(t_k)- \mathcal{I}^{k+1}\tilde{u}_{h,c} ^k\|_{\L^2}^2-\|\tilde{e}_c(t_k)\|_{\L^2}^2\Big)+ \eta^2\sum_{k=1}^N \int_{t_{k-1}}^{t_k} \| K*\nabla\tilde{u}_{h,r}\|_{\L^2}^2  \\&\quad+ \sum_{k=1}^N \eta^2\int_{t_{k-1}}^{t_k}  \Big\|\sum_{j=1}^{k}\omega_{kj}\tau_j\nabla\tilde{u}_h^{j-\frac{1}{2}}-( K*\nabla\tilde{u}_{h}^{j-\frac{1}{2}})\Big\|^2+ \eta^2\sum_{k=1}^N \int_{t_{k-1}}^{t_k} \Big\|( K*\nabla\tilde{u}_{h}^{j-\frac{1}{2}})-( K*\nabla\tilde{u}_{h}(s))\Big\|^2.
	\end{align}
	Estimating the memory term in a manner analogous to \eqref{3.amem}, we achieve:
	\begin{align}
		&\eta^2\sum_{k=1}^N \int_{t_{k-1}}^{t_k} \Big\|( K*\nabla\tilde{u}_{h}^{j-\frac{1}{2}})-( K*\nabla\tilde{u}_{h})(s)\Big\|^2\\&\leq \eta^2  \sum_{k=1}^N \int_{t_{k-1}}^{t_k} \left(\int_0^TK(t-s)\mathrm{~d}s\right)^2\left( \sum_{j=1}^k\int^{t_j}_{t_{j-1}}\| \nabla\tilde{u}_{h}^{j-\frac{1}{2}}-\nabla\tilde{u}_{h}(s)\|\right) ^2\\\nonumber&\leq \eta^2  \sum_{k=1}^N \int_{t_{k-1}}^{t_k}\frac{\tau_j}{2}\sum_{j=1}^k|\!|\!| \tilde{u}_{h}^j-\mathcal{I}^j\tilde{u}_{h}^{j-1} |\!|\!|^2\\&\leq \frac{\eta^2}{2}  \sum_{k=1}^N \tau_k|\!|\!| \tilde{u}_{h}^k-\mathcal{I}^k\tilde{u}_{h}^{k-1} |\!|\!|^2,
	\end{align}
	Consolidating the obtained estimates, we get:
	\begin{align}
		&\nonumber\|\tilde{e}_c(t_N)\|_{\L^2}^2 + \int_0^{t_N}|\!|\!| \tilde{e}_c(t)|\!|\!|^2 \mathrm{d}t \\&\nonumber\leq \|\tilde{e}_c(0)\|_{\L^2}^2 + \sum_{k=1}^N \int_{t_{k-1}}^{t_k} \Big( \|f-f^{k-\frac{1}{2}}\|^2_{\L^2}+ \|\partial_t\tilde{u}_{h,r} \|_{\L^2}^2  +|\!|\!|  \tilde{u}^{k-\frac{1}{2}} - \tilde{u}_{\tau_k} |\!|\!|^2 + |\!|\!|   \tilde{u}_{\tau_k} -\tilde{u}_{h,c} |\!|\!|^2\Big) \\&\quad +\frac{1}{2}\sum_{k=1}^{N-1}\Big(\|u(t_k)- \mathcal{I}^{k+1}\tilde{u}_{h,c} ^k\|_{\L^2}^2-\|\tilde{e}_c(t_k)\|_{\L^2}^2\Big) + C\eta^2  \sum_{k=1}^N \tau_k|\!|\!| \tilde{u}_{h}^k-\mathcal{I}^k\tilde{u}_{h}^{k-1} |\!|\!|^2\\&\nonumber\quad +\frac{\eta^2}{2}\sum_{k=1}^N \int_{t_{k-1}}^{t_k} \Big\|\left( \sum_{j=1}^{k}\frac{1}{\tau_k\tau_j } \int_{t_{k-1}}^{t_k} \int_{t_{j-1}}^{\min \left(t, t_j\right)} K(t-s)\tau_j\mathrm{~d}s \mathrm{~d}t-\sum_{j=1}^k\int_{t_{j-1}}^{\min \left(t, t_j\right)}K(t-s)\right) \nabla\tilde{u}_h^{j-\frac{1}{2}} \mathrm{~d}s\Big\|^2.
	\end{align} 
	Employing the definition of $\tilde{u}_{\tau_k}$ as stipulated in \eqref{3.CNliapp}, and applying the triangle inequality, we derive:
	\begin{align}
		\int_{t_{k-1}}^{t_k} |\!|\!|  \tilde{u}^{k-\frac{1}{2}} - \tilde{u}_{\tau_k} |\!|\!|^2 \leq \frac{\tau_k}{2} |\!|\!| \tilde{u}^k - \tilde{u}^{k-1} |\!|\!|^2 \leq C\left(  \Upsilon_k^2 + \tau_k |\!|\!| \tilde{u}^k - \tilde{u}_h^{k} |\!|\!|^2 + \tau_k |\!|\!| \tilde{u}^{k-1} - \mathcal{I}^k\tilde{u_h}^{k-1} |\!|\!|^2\right).
	\end{align}
	Synthesizing the outcomes with the findings of Theorem \ref{3.thm3.7} leads to:
	\begin{align}
		\|\tilde{e}_c(t_N)\|_{\L^2}^2 + \int_0^{t_N}|\!|\!| \tilde{e}_c(t)|\!|\!|^2 \mathrm{d}t \leq &C(\Xi^2 + \Upsilon^2+ \mathcal{F}_2+\mathcal{K}^2+\|\tilde{e}_c(0)\|_{\L^2}^2\\
		&+ \sum_{k=1}^{N-1}\Big(\|u(t_k)- \mathcal{I}^{k+1}\tilde{u}_{h,c} ^k\|_{\L^2}^2-\|\tilde{e}_c(t_k)\|_{\L^2}^2\Big). 
	\end{align}
\end{proof}
\begin{theorem}\label{3.Theorem5.4}
	Let $u$ and $u_h$ represent the exact solution and the discrete solution, respectively, of the equation \eqref{3.GBHE}. Consider $\Xi$ and $\Upsilon$ as the a posteriori error estimators defined in \eqref{3.timeestimator1}. The following reliability estimator is established:
	\begin{align}
		\|e(t_N)\|_{\L^2}^2 + \int_0^{t_N}|\!|\!| e(t)|\!|\!|^2 \mathrm{d}t &\leq C\Big(\Xi^2 + \Upsilon^2+ \mathcal{F}_2^2+\mathcal{K}^2+\|\tilde{e}_c(0)\|_{\L^2}^2 + \sum_{k=1}^{N-1}\|u_{h,r}^k- \mathcal{I}^{k+1}\tilde{u}_{h,r} ^k\|_{\L^2}^2\Big),
	\end{align}
	where 	$\mathcal{F}_2^2$ and $\mathcal{K}^2$ are defined in \eqref{3.timeestimatorCN2}.
\end{theorem}

\begin{proof}
	Utilizing the decomposition $u_h^k = u_{h,c}^k +u_{h,r}^k$ and the identity \cite[(5.59)-(5.60)]{EOJ} given by
	\begin{align}
		\|u(t_k)- \mathcal{I}^{k+1}\tilde{u}_{h,c} ^k\|_{\L^2}^2-\|\tilde{e}_c(t_k)\|_{\L^2}^2 = \|u_{h,r}^k - \mathcal{I}^{k+1}u^k_{h,r}\|_{\L^2}^2 + \langle u^k_{h,r} - \mathcal{I}^{k+1}u_{h,r}^k, \tilde{e}_c(t_k)\rangle.
	\end{align}
	and invoking Lemma \ref{3.Lemma5.1}, this proof follows a methodology analogous to that of Theorem \ref{3.theorem4.3}.
\end{proof}	%
%
%
%
%
%
%
%
%
%
%
%
\section{Numerical Studies}\label{sec6}
The computational analysis plays a pivotal role in affirming the practical utility of the derived results, bolstering confidence in the theoretical framework expounded in the preceding sections. In this section, we validate the theoretical findings through numerical experiments. All computational procedures are executed using the open-source finite element library FEniCS \cite{ABJ}. Our primary focus involves investigating the convergence patterns on both uniform and adaptive meshes, facilitated by employing an a posteriori error estimator. To elaborate, we initiate the process with an initial mesh and iterate through a refinement loop, consisting of the following stpdf:
$$\text{\textbf{Solve}} \rightarrow  \text{\textbf{Estimate}} \rightarrow  \text{\textbf{Mark}} \rightarrow  \text{\textbf{Refine}}.$$
This iterative refinement loop enables the systematic construction of adaptively refined sequences, thereby improving the efficiency and reliability of our numerical simulations.\\

\textbf{Example 1:} \textbf{Accuracy verification:}  We initiate our numerical section by considering a smooth solution in benchmark problems, where the initial data $u_0$ and external forcing $f$ are manufactured from the exact solution $u$ in $d$ dimensions $(d = 2, 3)$ \bl{under uniform discretization of the mesh $\Omega = (0,1)^d\times[0,1]$.} The error plots in the energy norm along with the efficiency of the SGBHE and GBHE with weakly singular kernel under different cases, are presented. For this example, we set the parameters as $\alpha = \beta = \delta = \nu = 1$, and $\gamma = 0.5$. The choice of solutions are given by 
\begin{align}
	\underbrace{u =\prod_{i=1}^{d}\sin(\pi x_i)}_{SGBHE}, \quad \quad  \underbrace{u = (t^3 - t^2 + 1) \prod_{i=1}^{d}\sin(\pi x_i)}_{GBHE}
\end{align}
The calculated quantities includes:
\begin{itemize}
	\item The DG spatial error norm is given by $ 	|\!|\!| e|\!|\!|$ where 
	\begin{align}\label{3.dgnorm1}
		\text{Total error} =	|\!|\!| e |\!|\!|^2:= \sum_{K\in\mathcal{T}_h}\|\nabla_h e\|_{\L^2(\mathcal{T}_h)}^2 + \sum_{E\in\mathcal{E}_h}\mathfrak{K}_h\|[\![v]\!]\|_{\L^2(E)}^2.
	\end{align}
	\item The error indicator for SGBHE ($\zeta$):
	\begin{align}\label{3.errind}
		\zeta=\Big(\sum_{K \in \mathcal{T}_h} \zeta_K^2\Big)^{1 / 2}. 
	\end{align}
	where $\zeta_K$ are as defined in \eqref{3.sgbheest}.
	\item For the GBHE, the time dependent norm is given by:  
	
	$$ \text{Total error} = \|e(t_N)\|_{\L^{\infty}(0,T;\L^2(\Omega))}^2 +  \int_0^{t_N}|\!|\!| e(t)|\!|\!|^2 \mathrm{d}t .$$
	\item The error indicator for the time dependent case is the sum of the accumulated time and spectral error indicator as defined in \eqref{3.timeestimator1} for the backward Euler discretization in time and \eqref{3.timeestimatorCN1} for the CN discretization in time. So , the combined indicator is given as
	\begin{align}\label{3.timeind}
		\zeta = \left(\Xi^2 + \Upsilon^2\right)^{1 / 2}. 
	\end{align}
	\item The rate of convergence $r$ and the efficiency $eff$ are calculated using:
	\begin{align*}
		r = \log(|\!|\!|e_i|\!|\!|/|\!|\!|e_{i+1}|\!|\!|)/\log(h_i/h_{i+1}),\quad eff = \frac{error~indictor}{Total~ error},
	\end{align*}
	where $e_i$ represents errors corresponding to mesh discretization parameters $h_i$.
\end{itemize}
\begin{figure}[ht!]
	\begin{center}
		\includegraphics[width=0.490\textwidth]{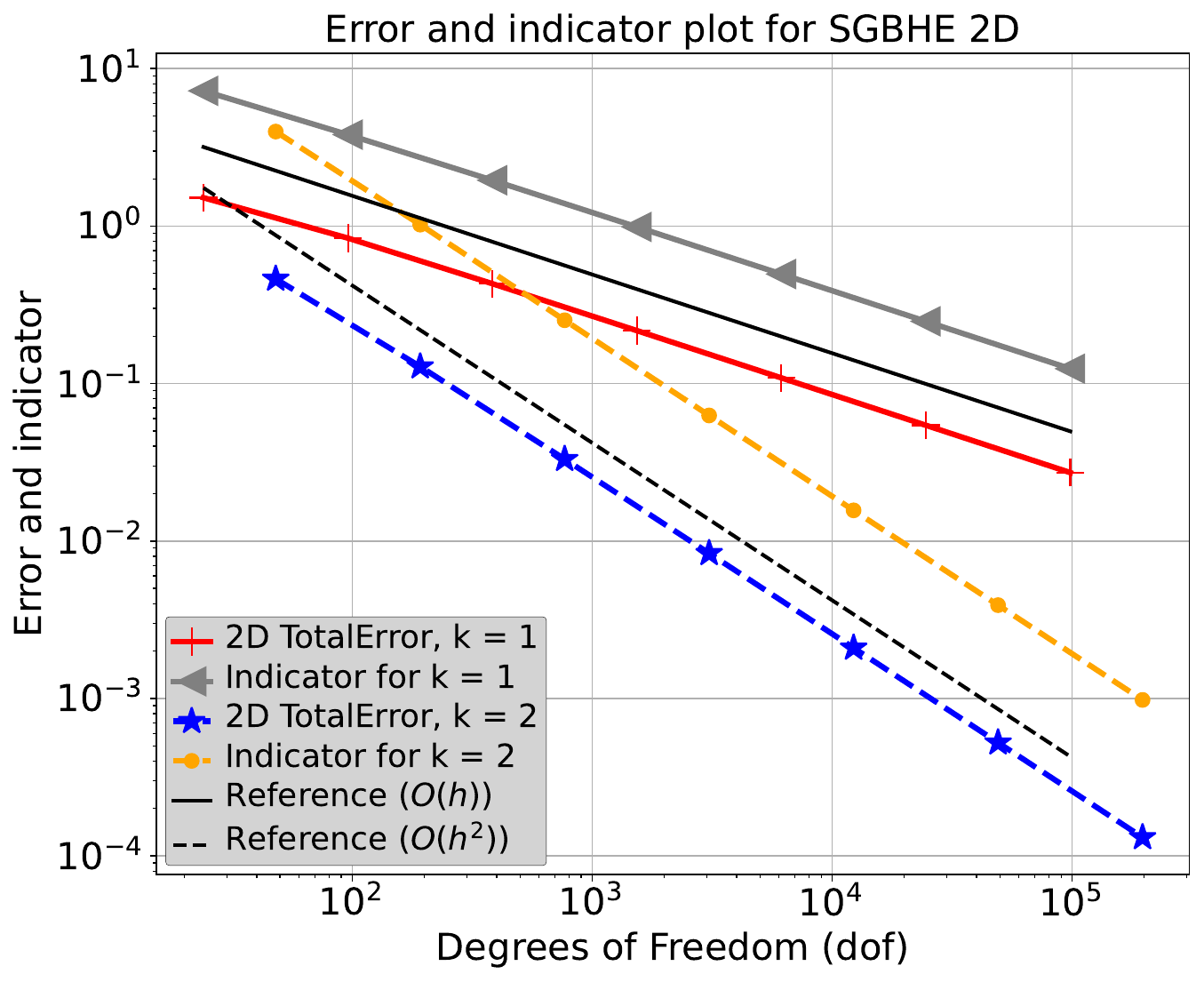}
		\includegraphics[width=0.490\textwidth]{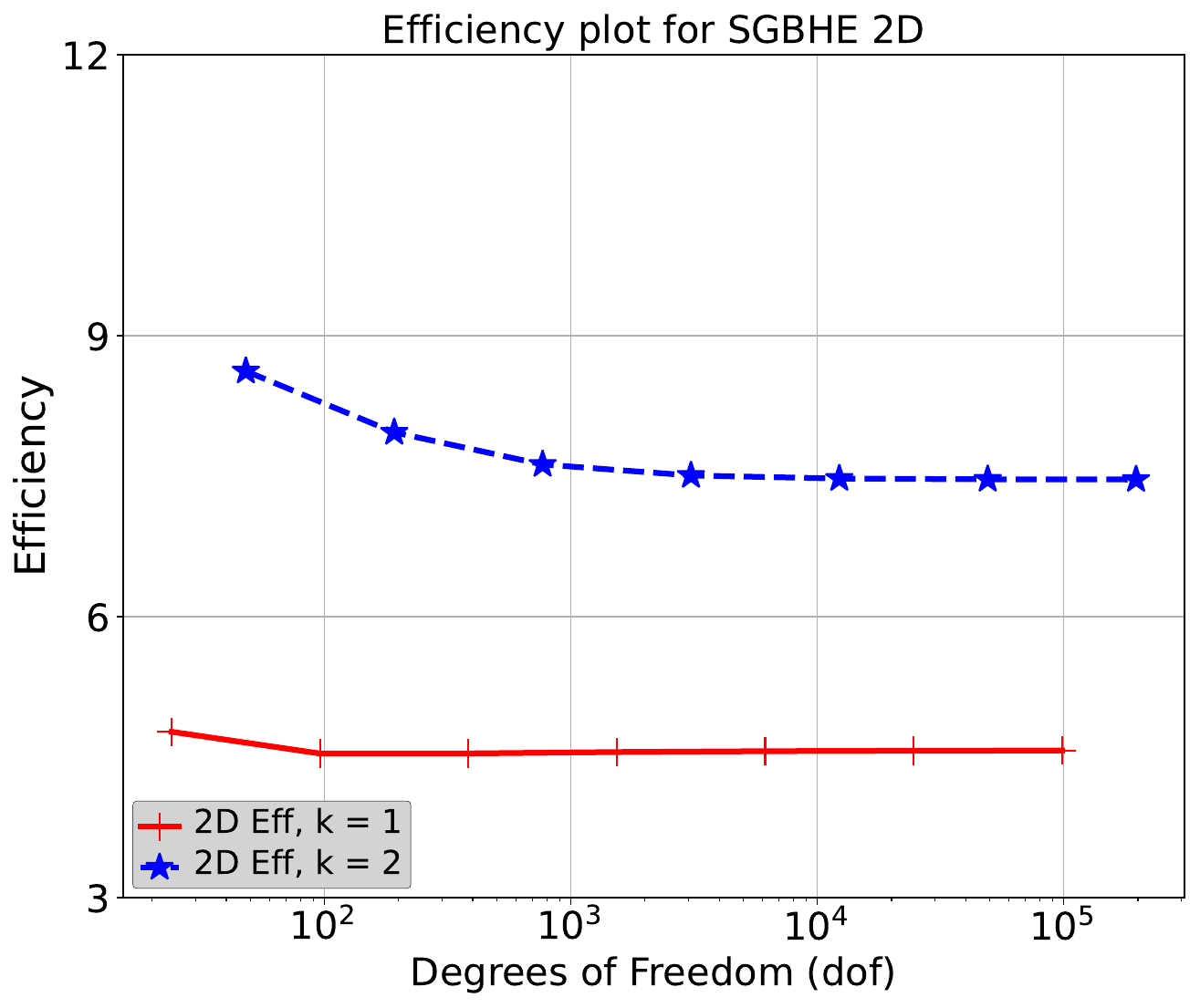}
		\includegraphics[width=0.490\textwidth]{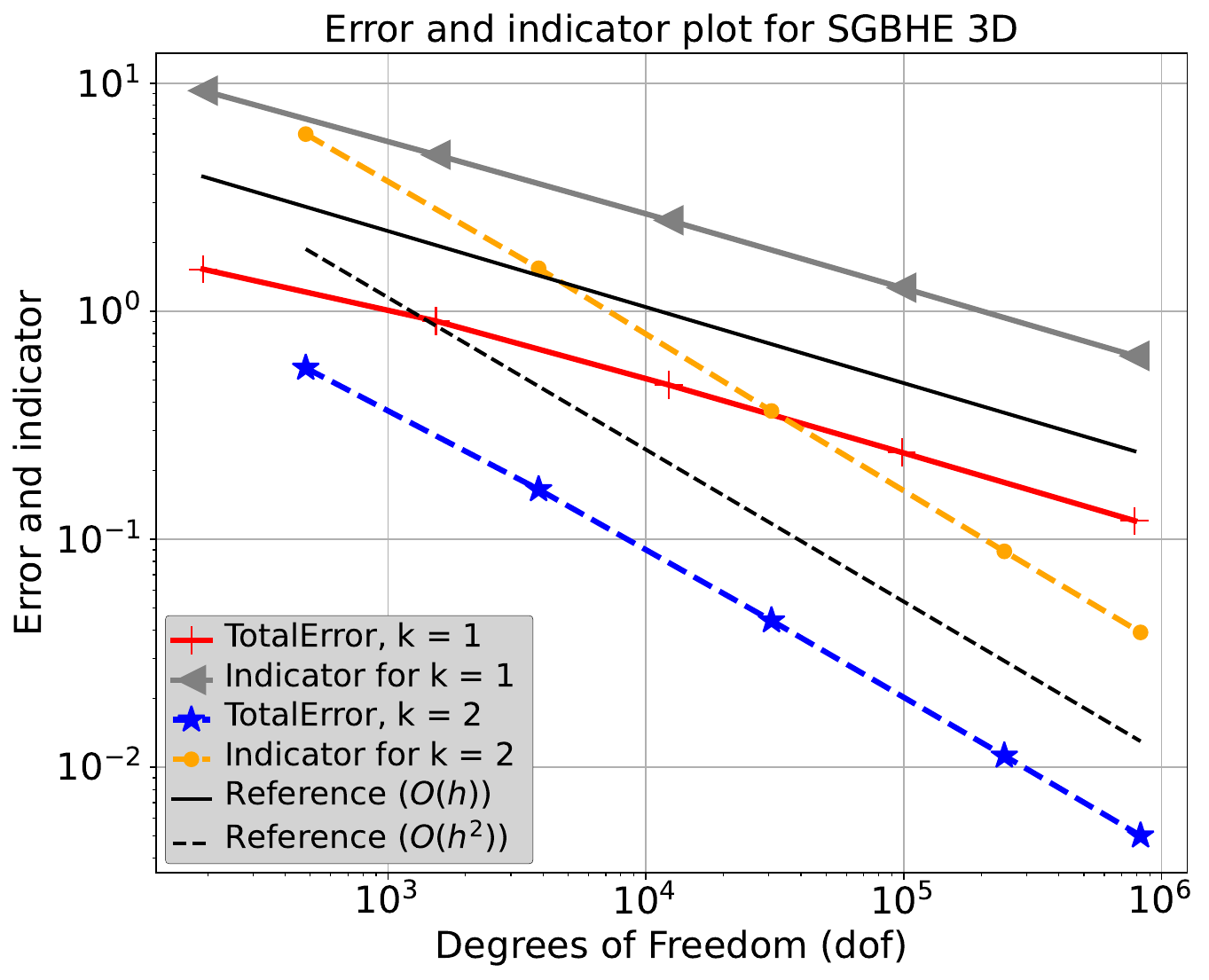}
		\includegraphics[width=0.490\textwidth]{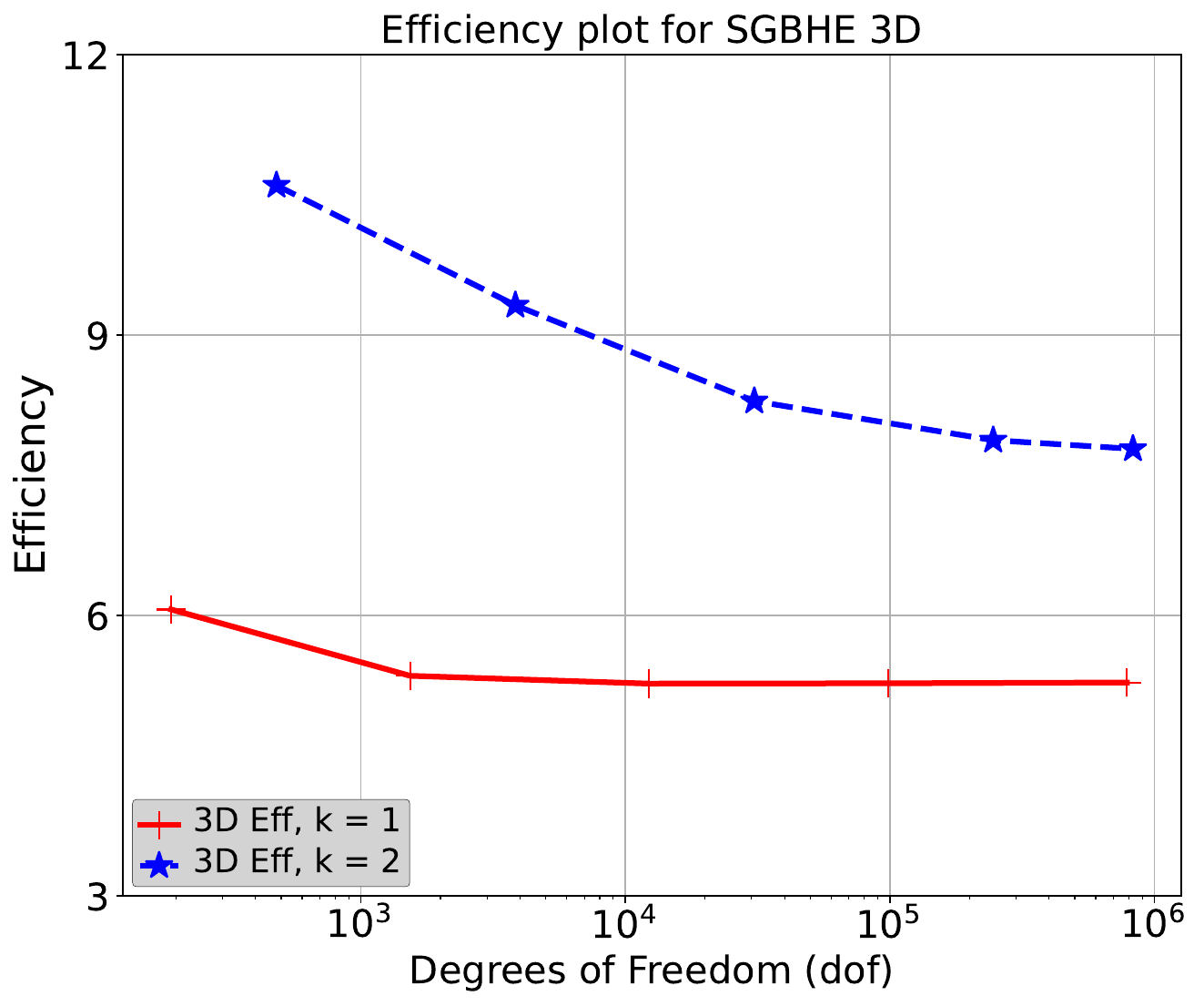}
	\end{center}\caption{ \footnotesize{Error in energy norm, error indicator and efficiency plot for the $2D$ and $3D$ SGBHE under uniform refinement for exact solution $u(\x) = \prod_{i=1}^d \sin(\pi x_i)$ with approximation degree $k = 1,2$ receptively.}}\label{fig:ex1}
\end{figure}

\begin{figure}[ht!]
	
	\begin{center}
		\includegraphics[width=0.490\textwidth]{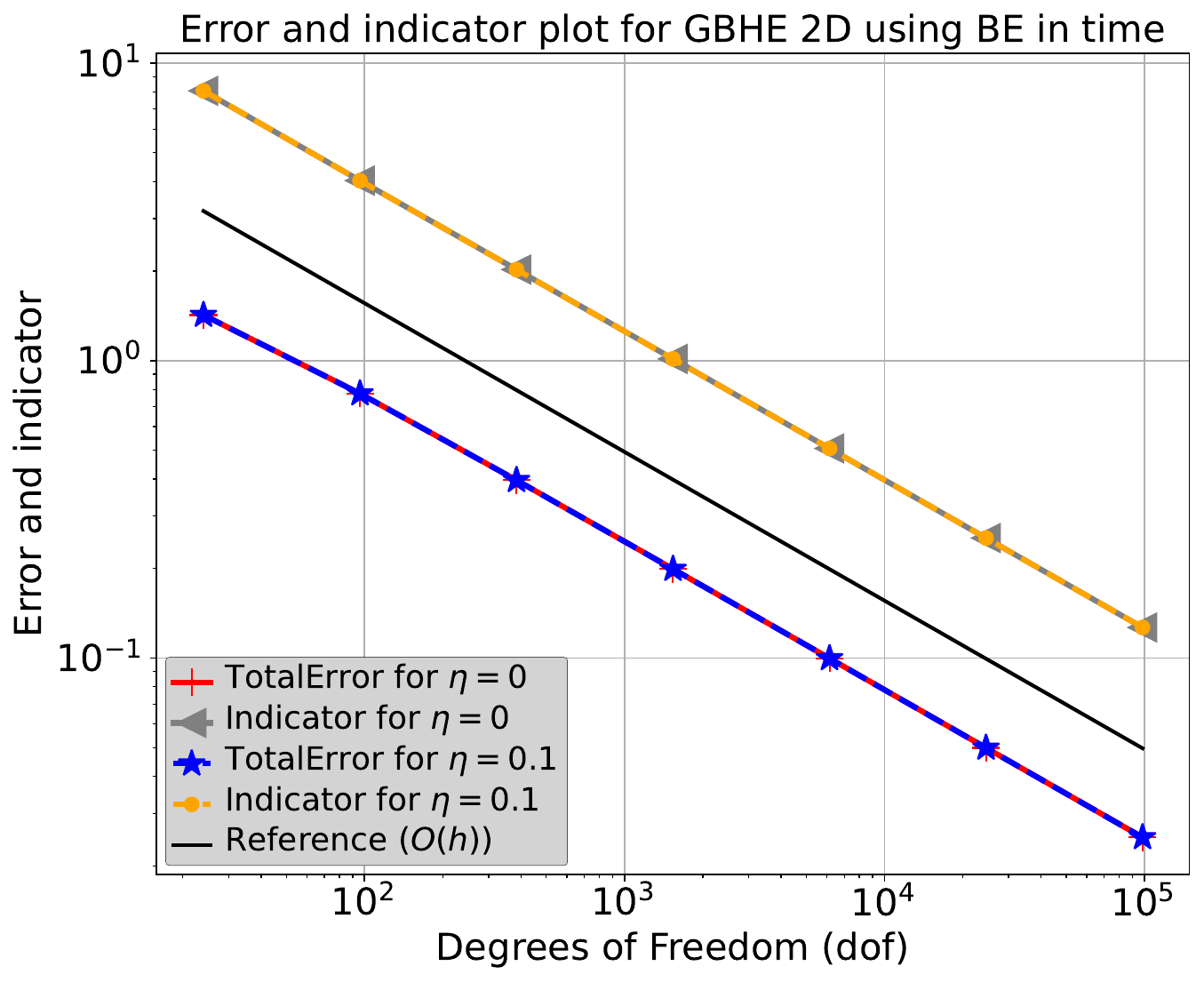}
		\includegraphics[width=0.490\textwidth]{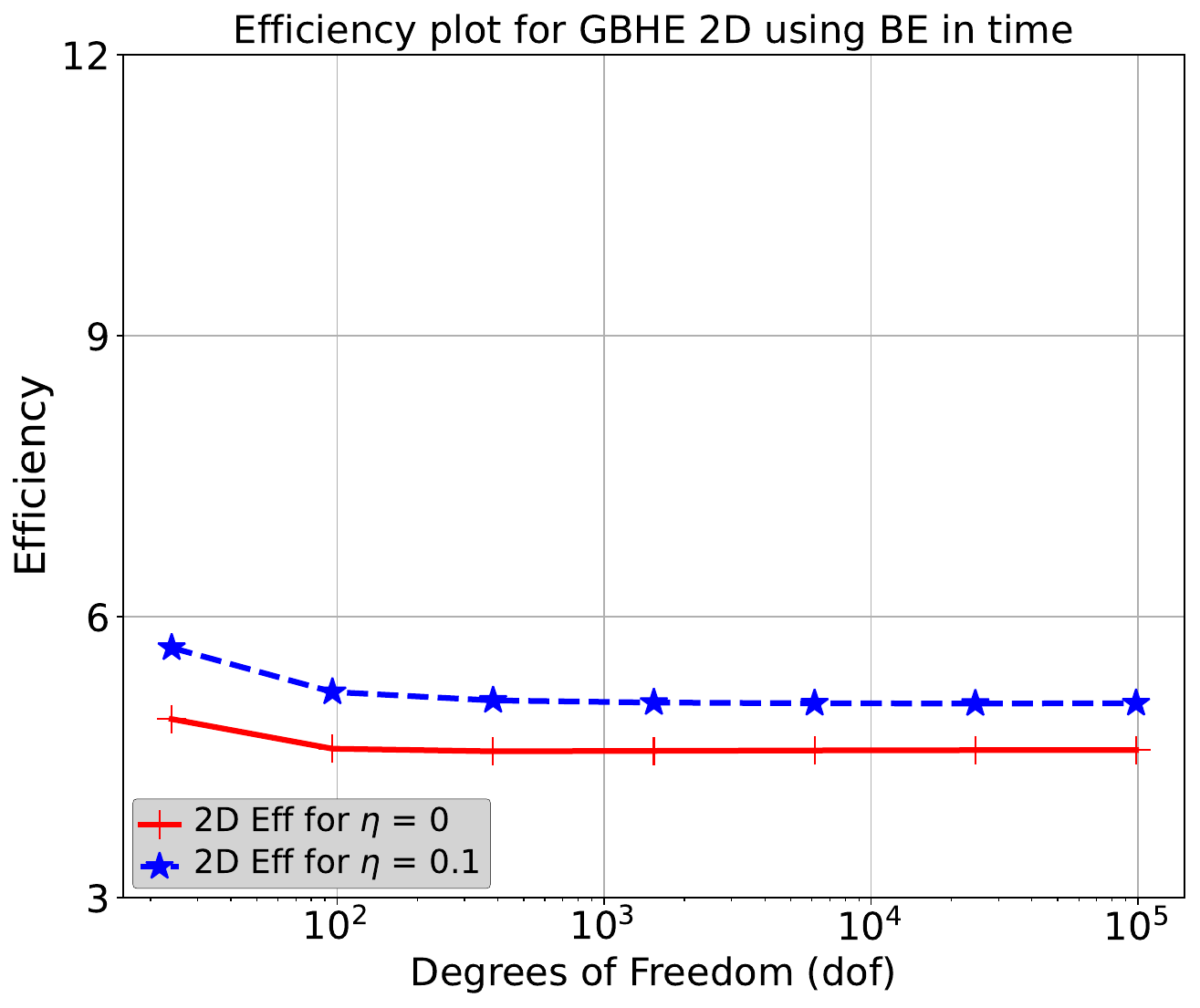}
		\includegraphics[width=0.490\textwidth]{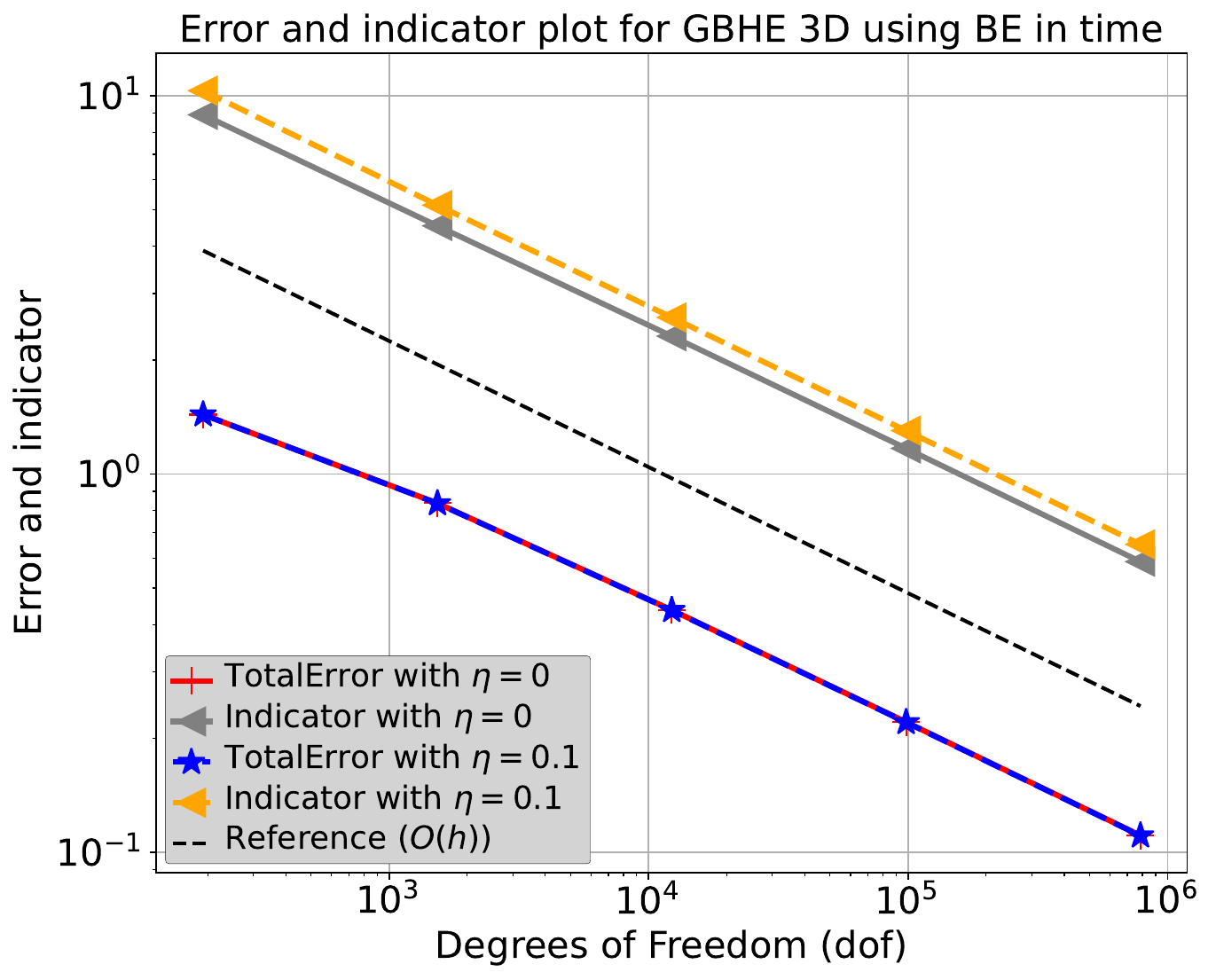}
		\includegraphics[width=0.490\textwidth]{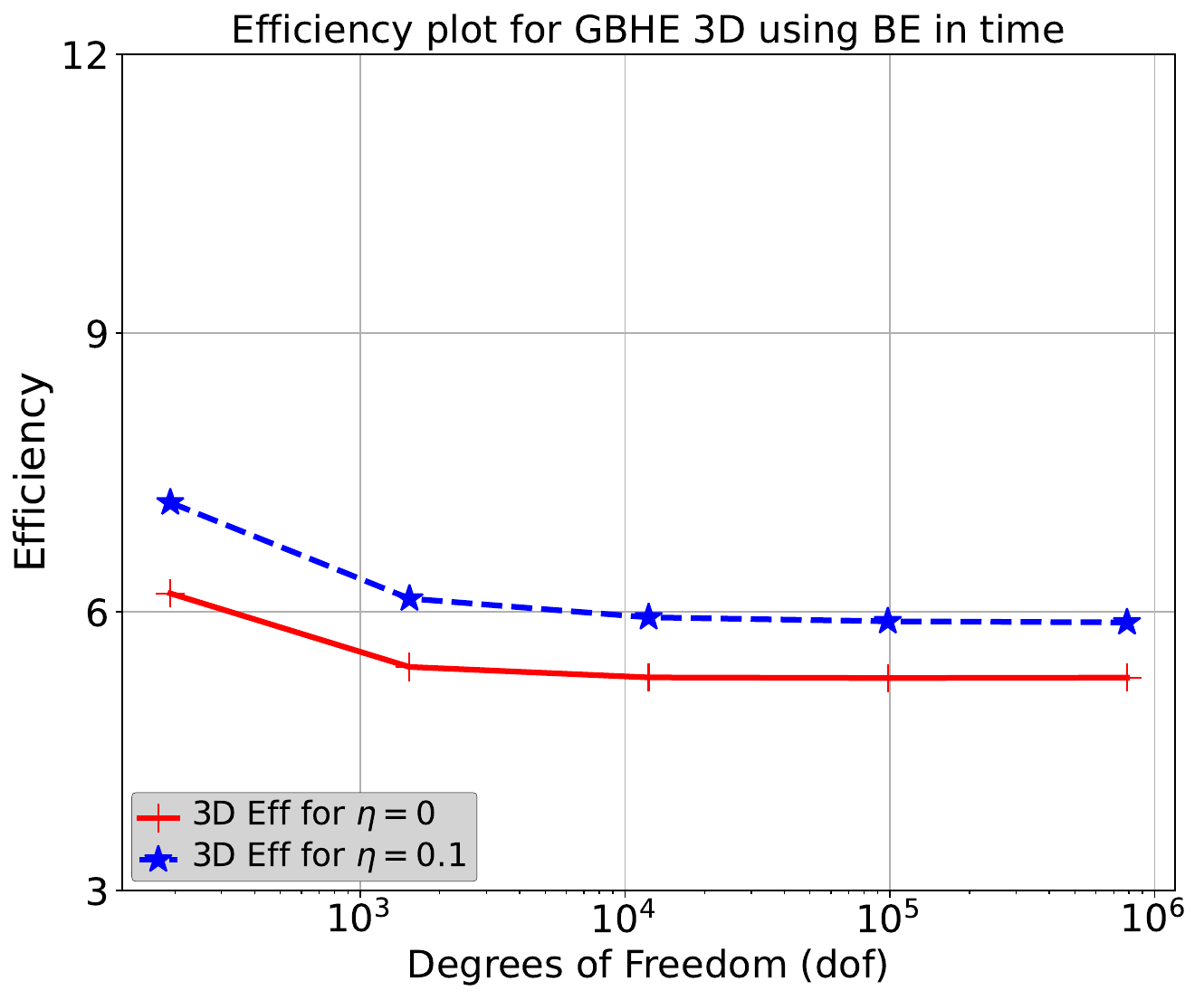}
	\end{center}
	\caption{\footnotesize{Error in energy norm, error indicator and efficiency plot for the $2D$ and $3D$ GBHE using backward Euler in time under uniform refinement for exact solution $u(\x,t) = \prod_{i=1}^d(t^3-t^2+1)\sin(\pi x_i)$ with ($\eta=0.1)$ and without $(\eta=0)$ memory.}}\label{fig:ex2}
\end{figure}
In Figure \ref{fig:ex1}, we depict the error estimates in the DG spatial norm \eqref{3.dgnorm1} for polynomial approximation degrees $k = 1, 2$ in $2D$ and $3D$, respectively. The achieved optimal convergence rate of $O(h^k)$ for $k = 1, 2$ affirm our theoretical assertions. Notably, the error indicator attains optimal convergence at $O(h)$ for $k= 1$ and $O(h^2)$ for $k = 2$, providing further support for the accuracy of our proposed estimator. The consistent efficiency observed across various mesh discretizations reinforces the reliability and effectiveness of our approach.

In the realm of the time-dependent scenario, we extend our analysis to the GBHE with a weakly singular kernel. Verification is conducted using both the backward Euler (BE) method \eqref{2.ncweakformfd} and the Crank-Nicolson (CN) scheme \eqref{2.ncweakformfdCN}. Employing the chosen weakly singular kernel $K(t) = \frac{1}{\sqrt{t}}$, Figures \ref{fig:ex2} showcase the error and the error indicator defined in \eqref{3.timeind}. These metrics exhibit optimal $O(h)$ convergence for both cases, with memory coefficients $\eta = 0.1$ and without memory $\eta = 0$ in $2D$ and $3D$, respectively. The efficiency plot consistently demonstrates the optimality of our estimator, maintaining a fixed ratio of the indicator to error. Lastly, Figure \ref{fig:ex3} illustrates the second-order convergence for the CN scheme.
\begin{figure}[ht!]
	\begin{center}
		\includegraphics[width=0.490\textwidth]{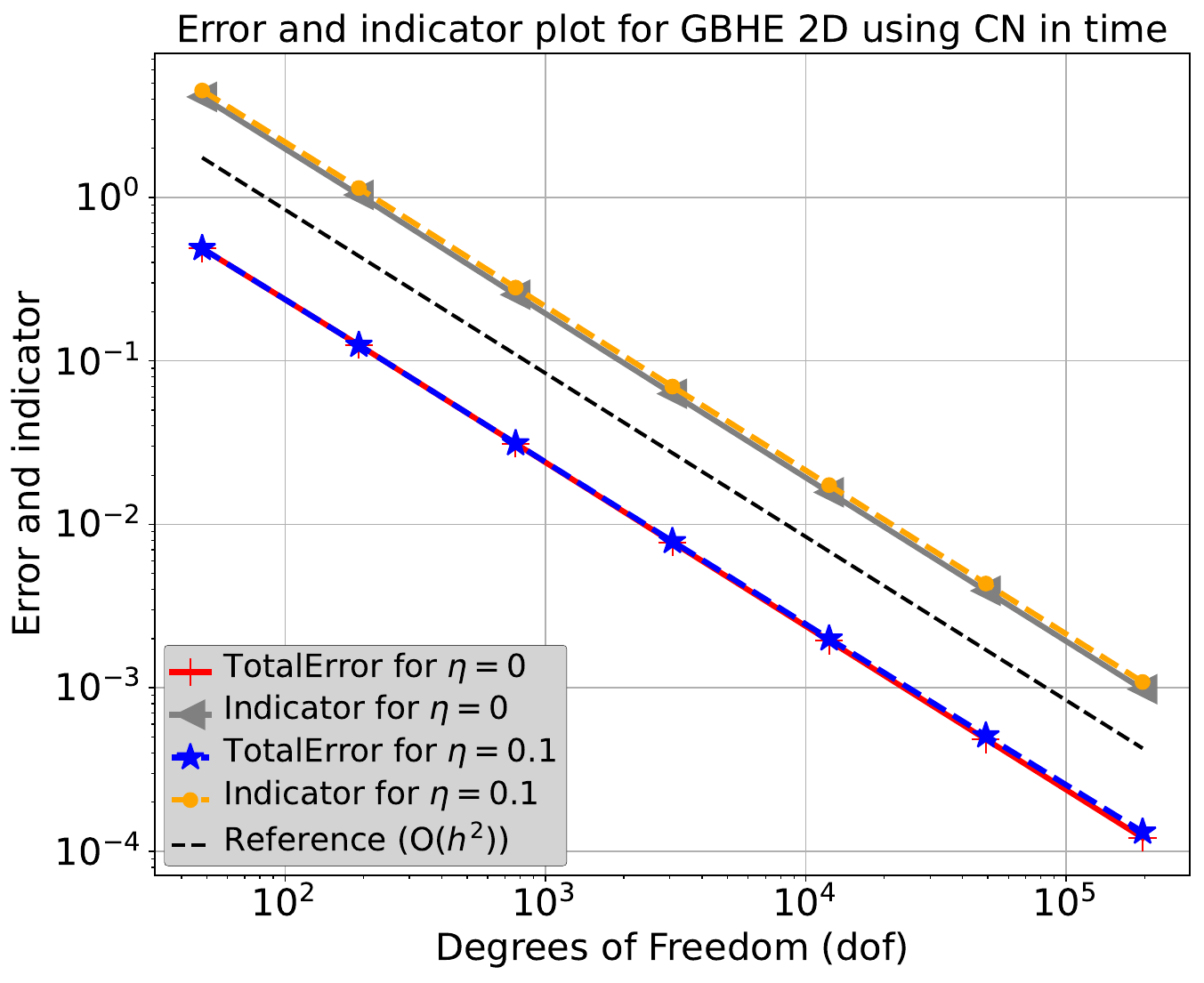}
		\includegraphics[width=0.490\textwidth]{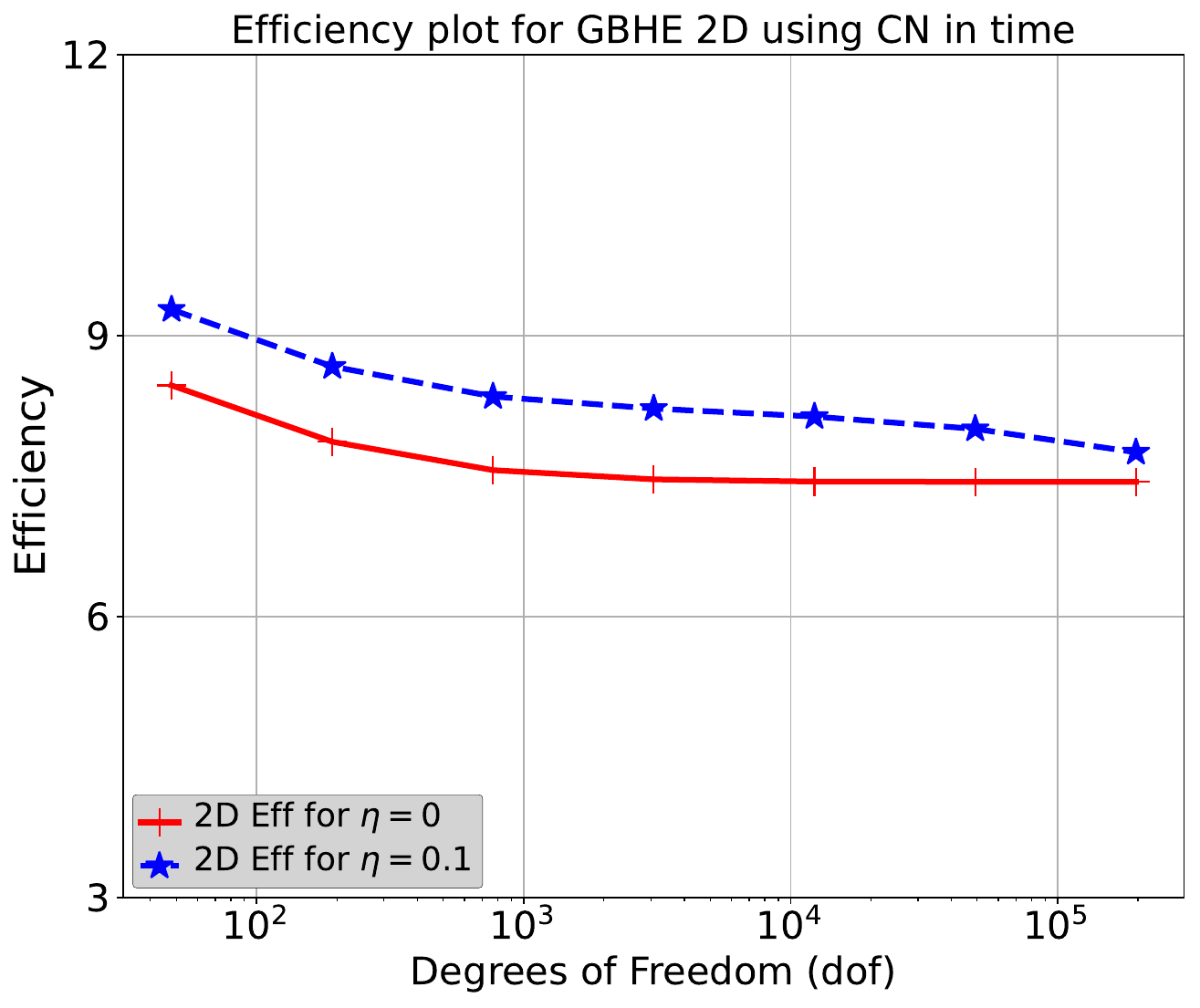}
		\includegraphics[width=0.490\textwidth]{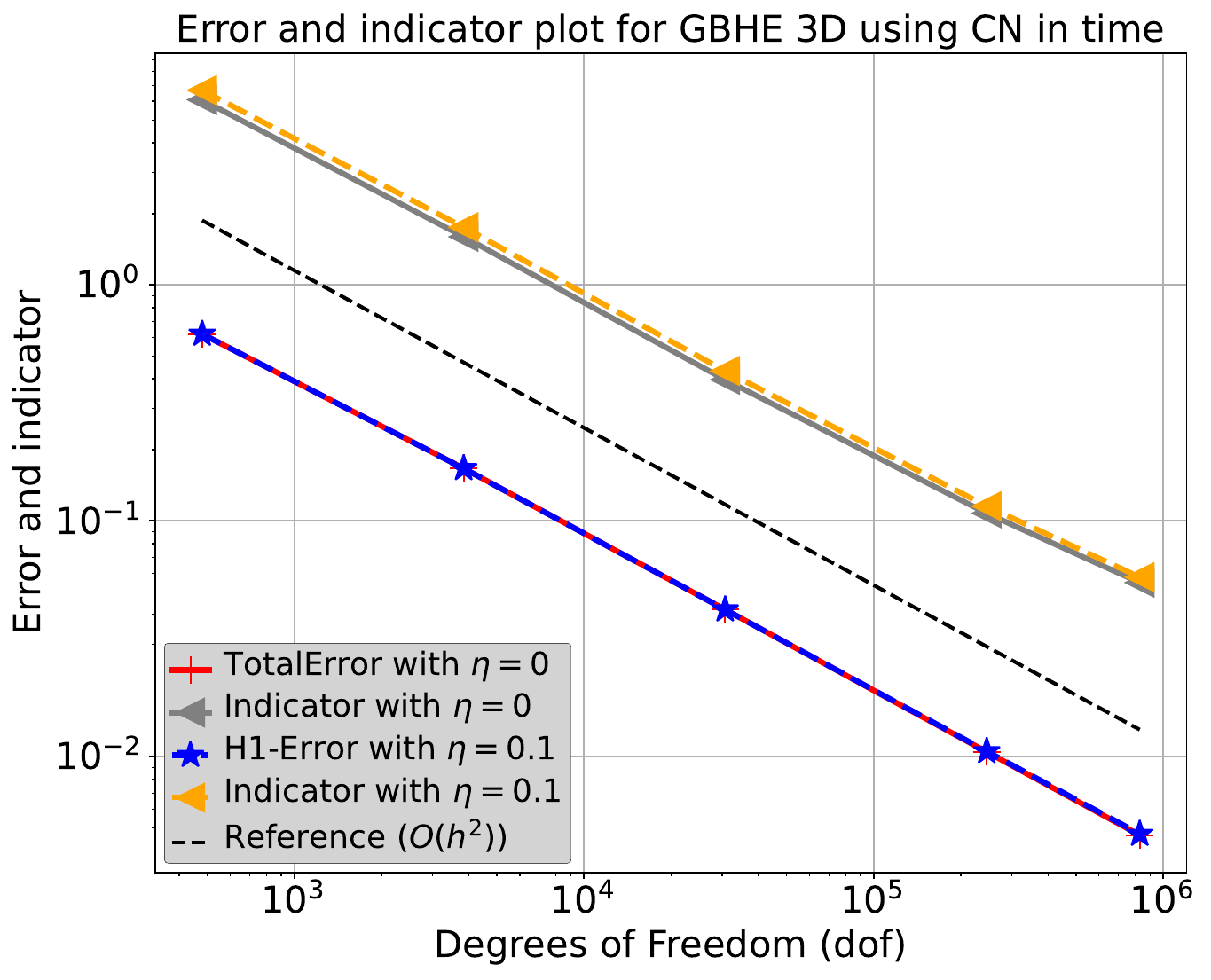}
		\includegraphics[width=0.490\textwidth]{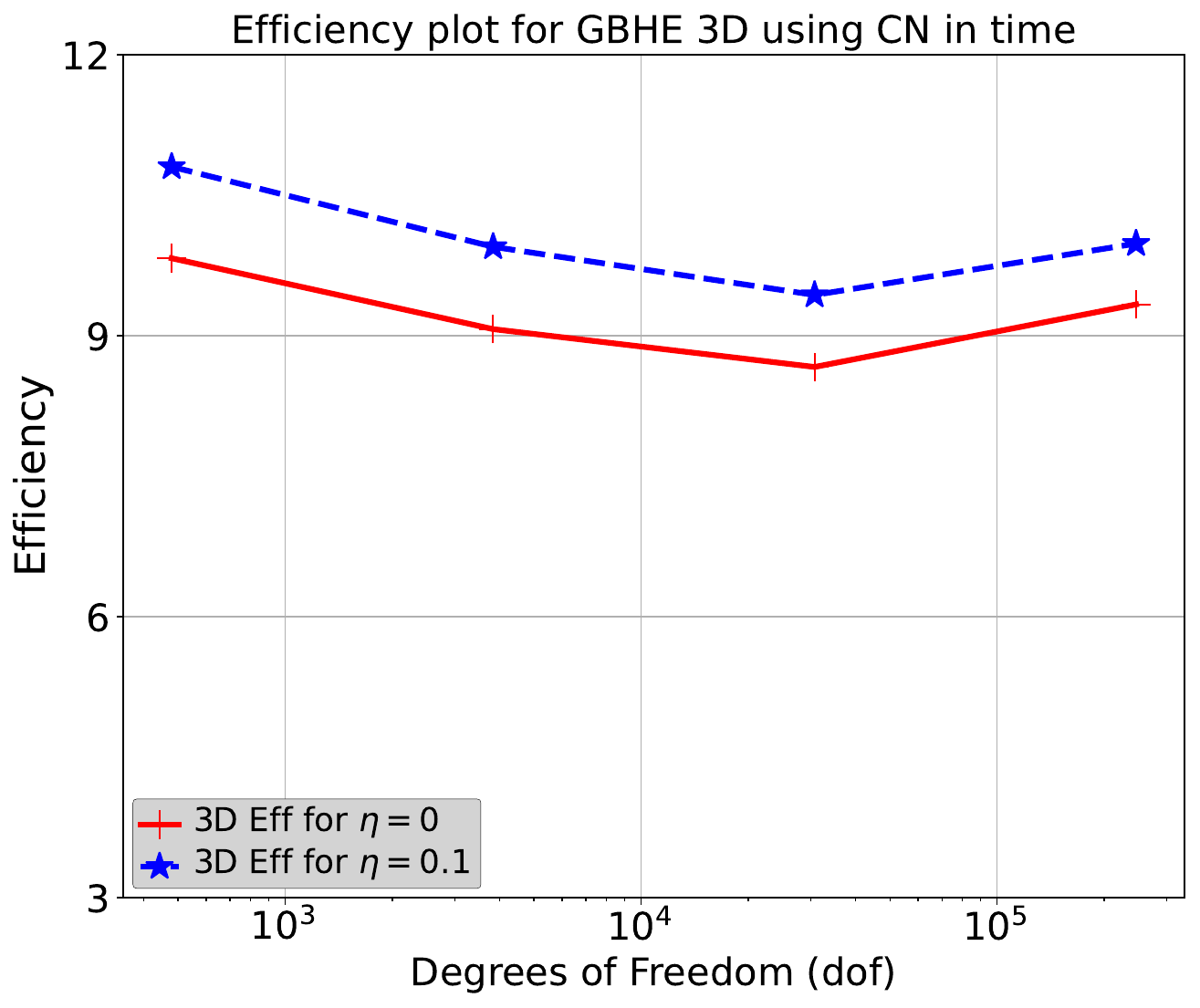}
	\end{center}
	\caption{\footnotesize{Error in energy norm, error indicator and efficiency plot for the $2D$ and $3D$ GBHE using CN in time under uniform refinement for exact solution $u(\x,t) = \prod_{i=1}^d(t^3-t^2+1)\sin(\pi x_i)$ with ($\eta=0.1)$ and without $(\eta=0)$ memory.}}\label{fig:ex3}
\end{figure}

\textbf{Example 2.} \textbf{L-shape domain.} Consider the SGBHE given by \eqref{3.SGBHEweak} defined on a non-convex L shape domain given by $\Omega = (-1,-1)^2\backslash(0,1)^2$. We consider the following two cases:
\begin{align*}
	\text{Case 1 : }u(x,y)&= xy(1-x)(1-y)\exp(-50((x-0.025)^2 + (y-0.025)^2)),\\	\text{Case 2 : } u(x,y)&= (x^2+y^2)^{\frac{1}{4}}.
\end{align*}
The forcing function $f$ and the Dirichlet boundary condition are prescribed according to the solution $u$ in both cases, with parameters chosen similarly to those in Example 1. The adaptive algorithm, comprising Solve, Estimate, Mark, and Refine steps, is employed. The error indicator is computed as defined in \eqref{3.errind}, and marking is conducted using the maximum criteria: a cell is refined if it satisfies
$$\zeta_K \geq \mu\max\limits_{L\in \mathcal{T}_h} \zeta_L,$$
where $0< \mu<1$. Mesh refinement is performed with equidistribution of the local error indicator in the updated mesh. The problem is then solved again on the refined mesh, re-estimated, and refined iteratively until the maximum residual value reaches the desired tolerance.
In the adaptive case, the convergence rate is computed as
$$ rate = -2 \log(|\!|\!|e_i|\!|\!|/|\!|\!|e_{i+1}|\!|\!|)/\log(DOF_i/DOF_{i+1}),$$
where $e_{i+1}$ and $\text{DOF}_{i+1}$ represent the error and degrees of freedom for the refined mesh. The constant $\mu$ is chosen such that the degrees of freedom for uniform and adaptive refinement are comparable.
\begin{figure}[ht!]
	\begin{center}
		\includegraphics[width=0.490\textwidth]{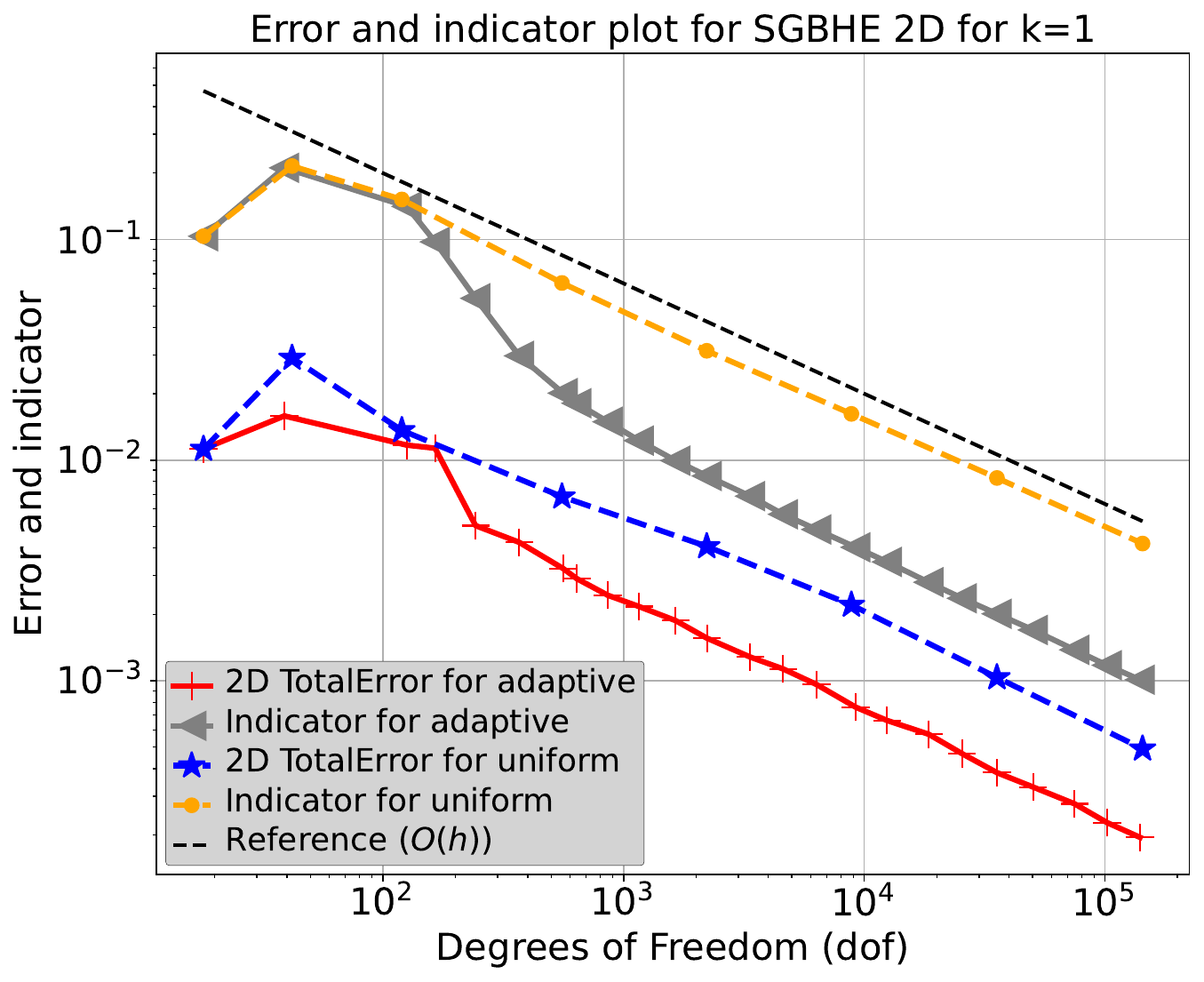}
		\includegraphics[width=0.490\textwidth]{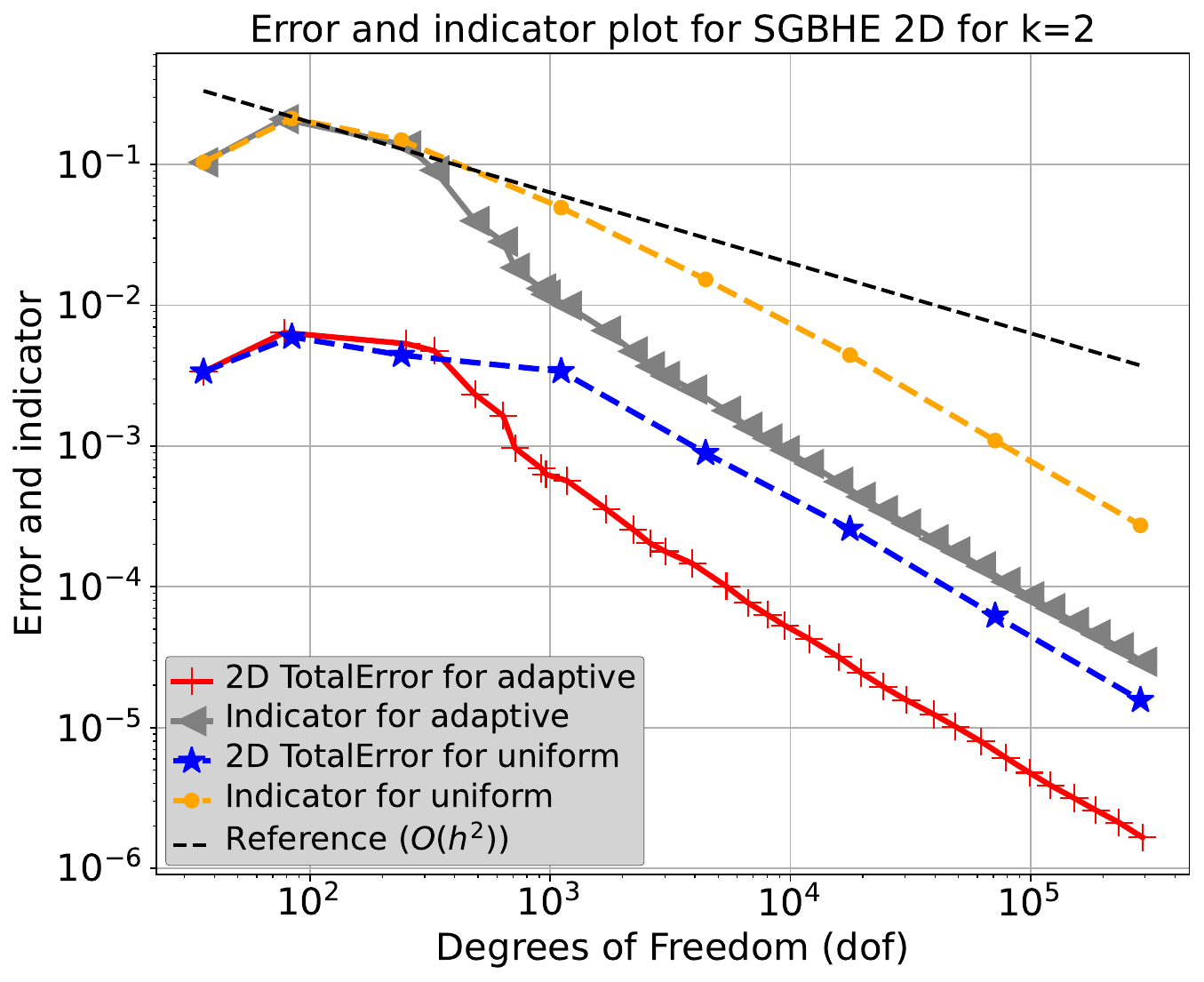}
	\end{center}
	\caption{\footnotesize{Error in energy norm, error indicator under adaptive and uniform refinement for the SGBHE with exact solution $u(x,y) = xy(1-x)(1-y)\exp(-50((x-0.025)^2 + (y-0.025)^2))$, and approximation degree, $k = 1,2$ respectively.}}	\label{fig:ex4}
\end{figure}
\begin{figure}[ht!]
	
	\begin{center}
		\includegraphics[width=0.325\textwidth]{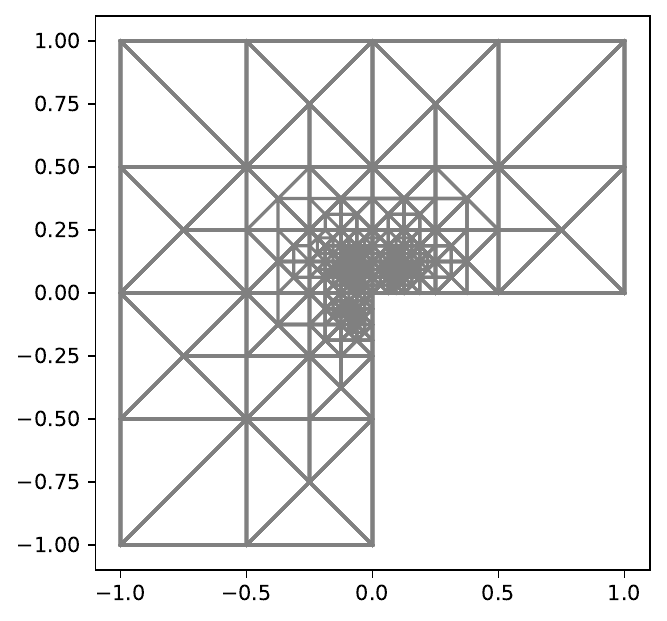}
		\includegraphics[width=0.325\textwidth]{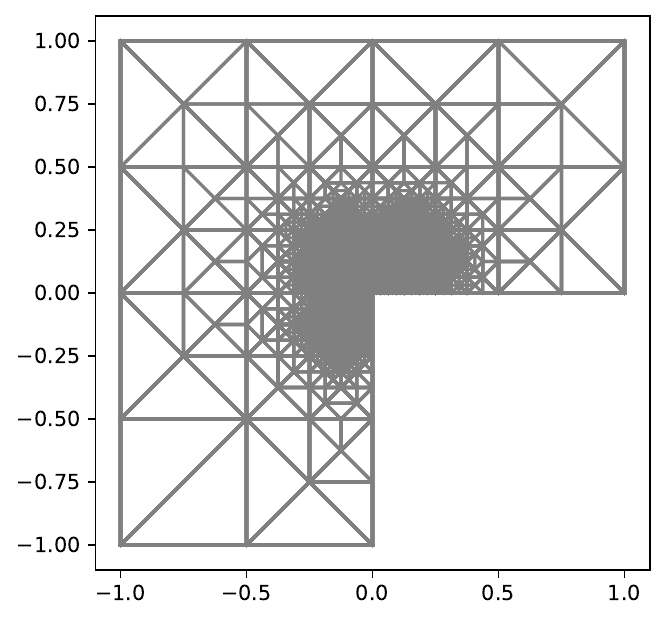}
		\includegraphics[width=0.325\textwidth]{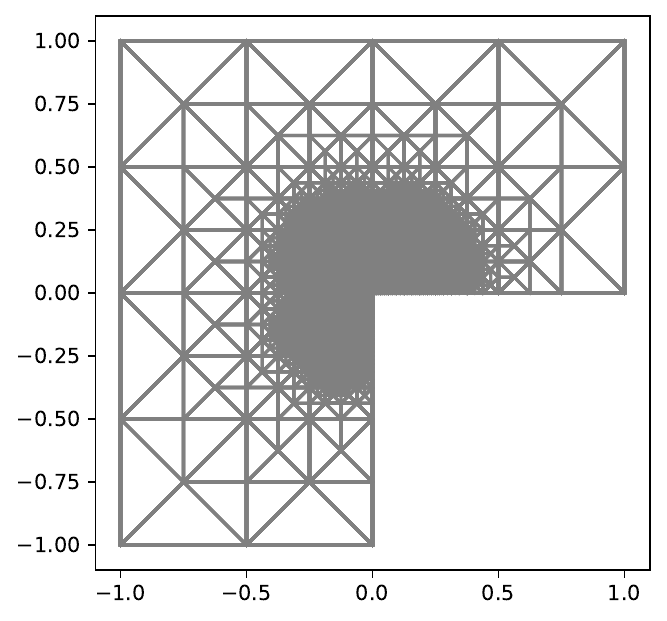}
	\end{center}
	\caption{\footnotesize{The refined meshes obtained after adaptive strategy with $\mathfrak{K}= 0.5$ and $1161, 12450,102084$ degrees of freedom respectively with approximation degree $k=1$ for Case 1.}}	\label{fig:ex6}
\end{figure}\\
In Case 1, the exact solution exhibits a high gradient around the point $(0,0)$. Despite this, the error plots in Figure \ref{fig:ex4} reveal that the error in the energy norm converges with $O(h)$ for both the uniform and adaptive refinement scenarios. Notably, due to the pronounced gradient and singularity at the point (0,0), more refinement is carried out in that specific region, as illustrated in Figure \ref{fig:ex6}.

Conversely, in Case 2, where there is a singularity at point $(0,0)$, the error in the energy norm converges suboptimally under uniform refinement, as depicted in Figure \ref{fig:ex5}. However, in the case of adaptive refinement, we achieve an optimal rate of convergence. The efficiency plot further assures the reliability of our estimator. The refined mesh obtained using the adaptive strategy at different degrees of freedom is presented in Figure \ref{fig:ex8}.
\begin{figure}[ht!]
	
	\begin{center}
		\includegraphics[width=0.490\textwidth]{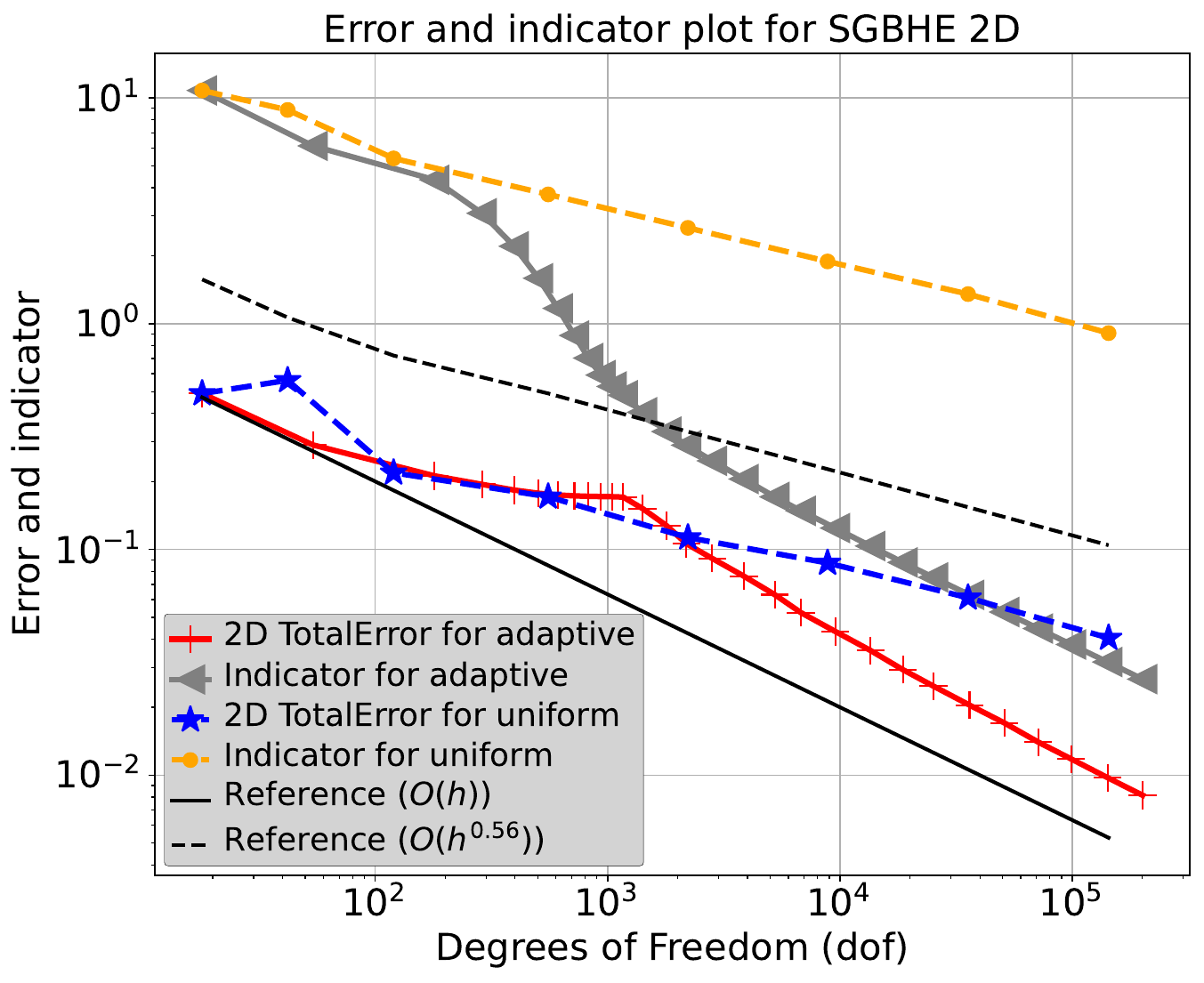}
		\includegraphics[width=0.490\textwidth]{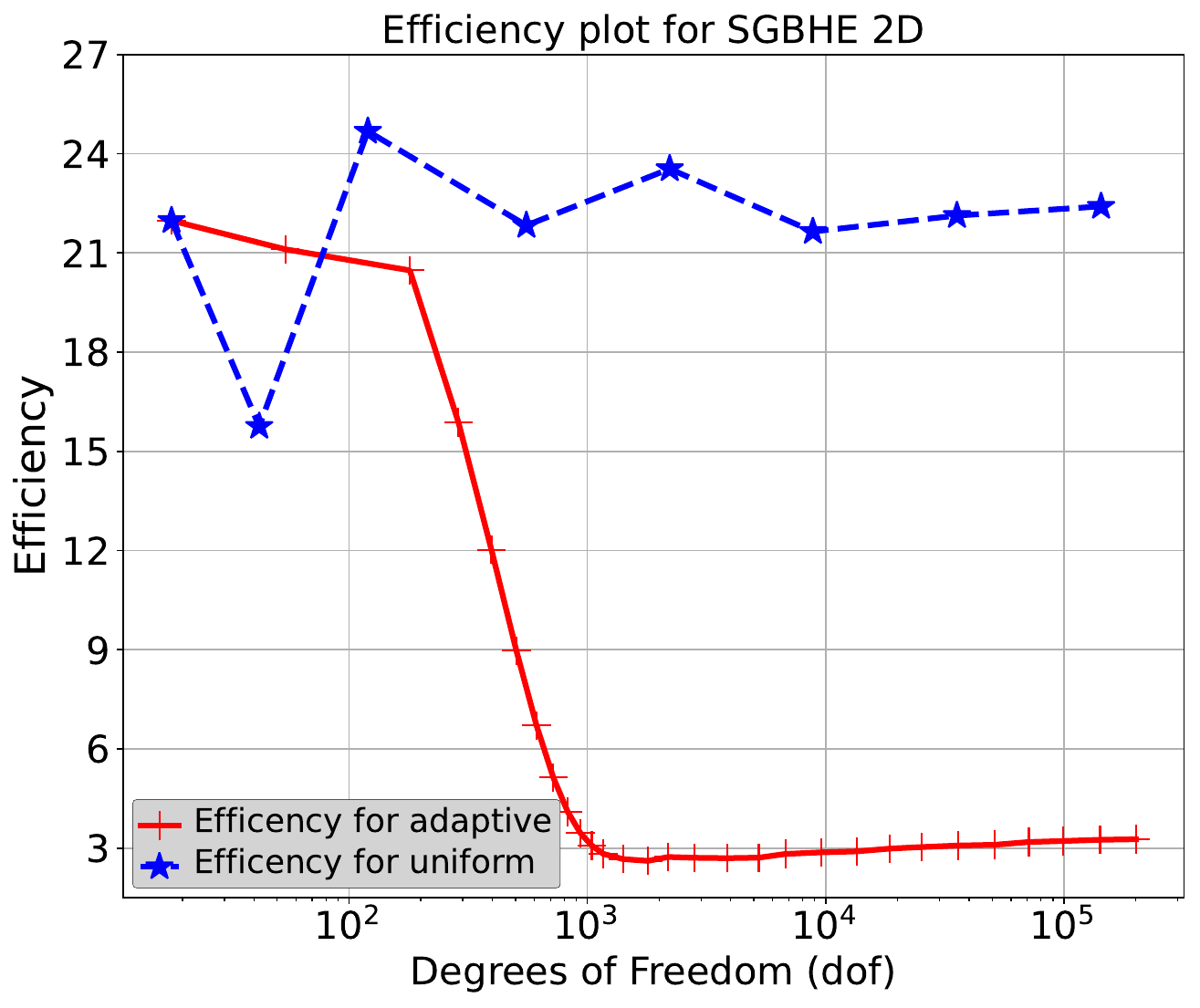}
	\end{center}
	\caption{\footnotesize{Error in energy norm, error indicator under adaptive and uniform refinement for the SGBHE with exact solution $u(x,y) = (x^2+y^2)^{\frac{1}{4}},$ and approximation degree, $k = 1$.}}\label{fig:ex5}
\end{figure}
\begin{figure}[ht!]
	
	\begin{center}
		\includegraphics[width=0.325\textwidth]{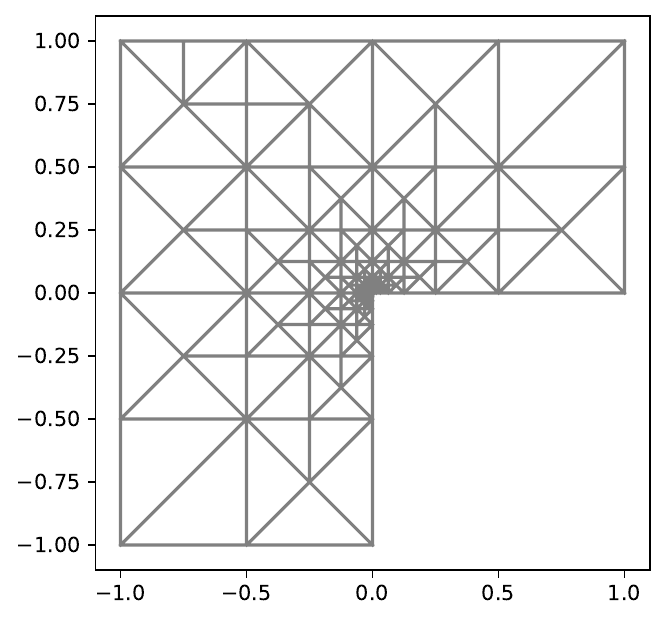}
		\includegraphics[width=0.325\textwidth]{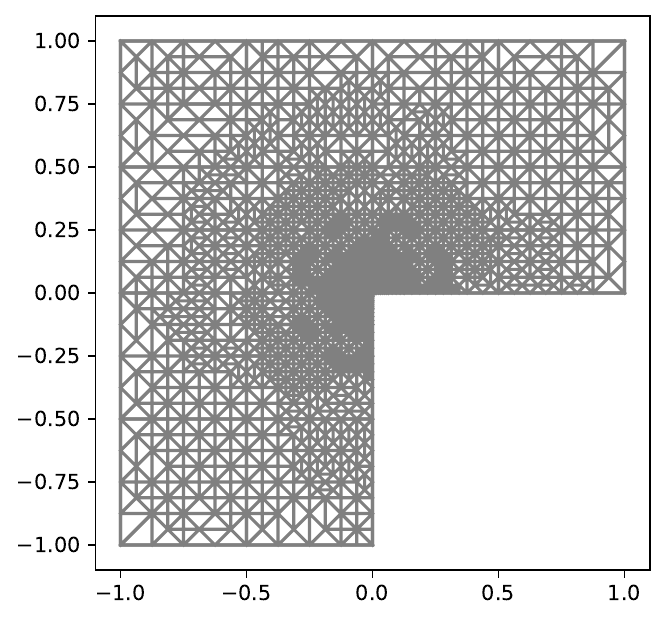}
		\includegraphics[width=0.325\textwidth]{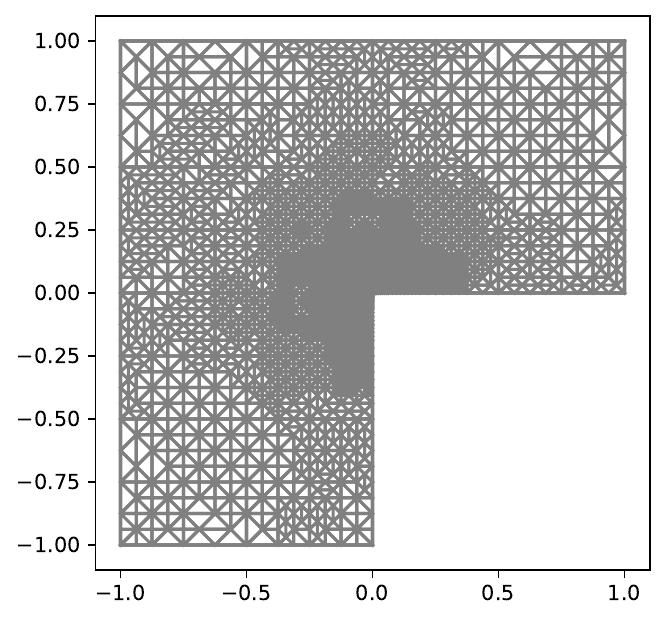}
	\end{center}
	\caption{\footnotesize{The refined meshes obtained after adaptive strategy with $1044, 9561$ and $13527$ degrees of freedom, respectively, with approximation degree $k=1$ for Case 2.}}	\label{fig:ex8}
\end{figure}\\
\textbf{Example 3:} \textbf{Time variable singularity.} In order to test our adaptive algorithm for the time adaptivity, we consider the GBHE ($\eta =0)$ defined in unit square with the exact solution \cite{NSo} given by 
$u = s(t)\times \exp(-50\times r^2(x,y,t))$ with $ r^2(x,y,t) = (x-0.4\times t-0.3)^2 + (y-0.4\times t-0.3)^2$ and 
$$
\begin{array}{cc}
	\bigg\{ & 
	\begin{array}{cc}
		s(t)= 1-\exp(-50\times(0.98\times t+0.01)^2) & \text{if }t <0.5, \\
		s(t) = 1-\exp(-50\times(1-0.98\times t+0.01)^2) & \text{else}. \\
	\end{array}
\end{array}
$$

Notice that the singularity undergoes a shift from $(0.3,0.3)$ at $t = 0s$ to $(0.7,0.7)$ at $t = 1s$ as time advances. For the domain $\Omega_t= [0,1]^2\times(0,T)$ with $T=1$ and a time discretization of $\tau = 0.1$, we employ an adaptive algorithm utilizing backward Euler (BE) in time and discontinuous Galerkin (DG) in space, allowing a maximum of $7$ refinements in space at each time step. Initiated at each time step with an initial unit square mesh featuring $4\times 4$ refinements, the refined mesh at each time level $t = i\times dt$ for $i= 1,2,\cdots,10$ is displayed in Figure \ref{fig:ex7}. The observed shift in the time singularity at each time step serves as validation for the effectiveness of our adaptive strategy.

\begin{figure}[ht!]
	
	\begin{center}
		\includegraphics[width=0.325\textwidth]{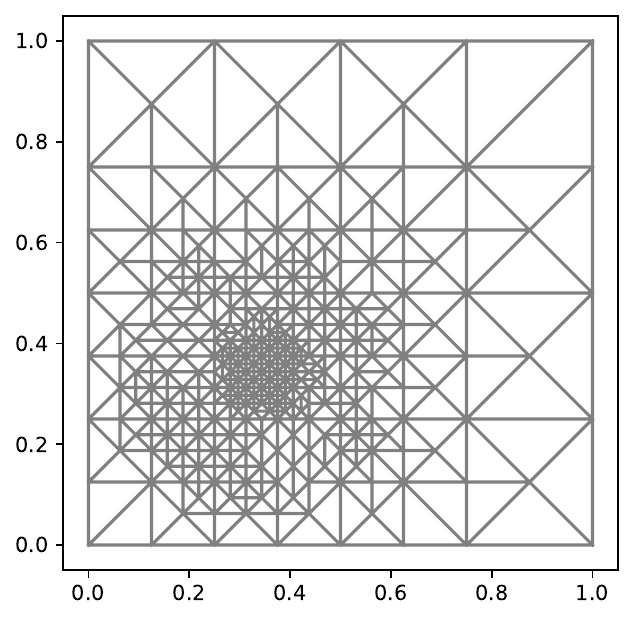}
		\includegraphics[width=0.325\textwidth]{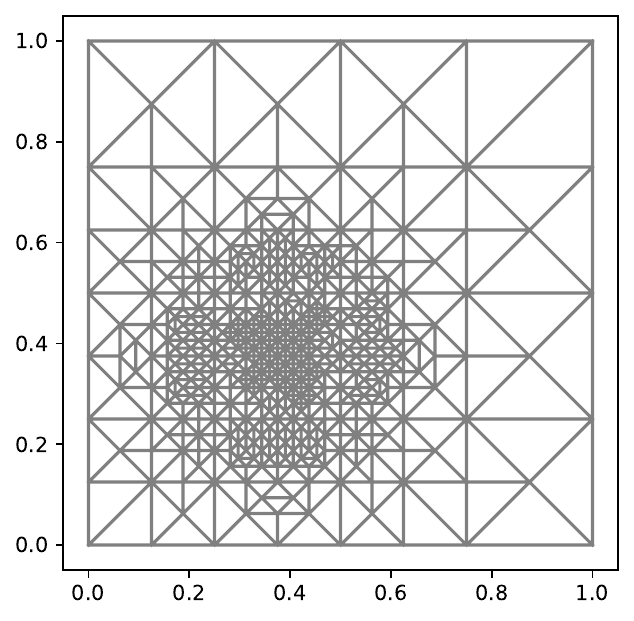}
		\includegraphics[width=0.325\textwidth]{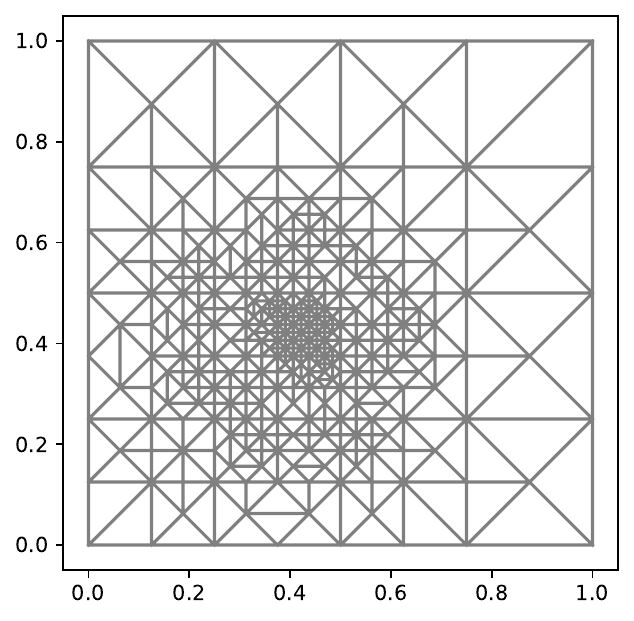}
		\includegraphics[width=0.325\textwidth]{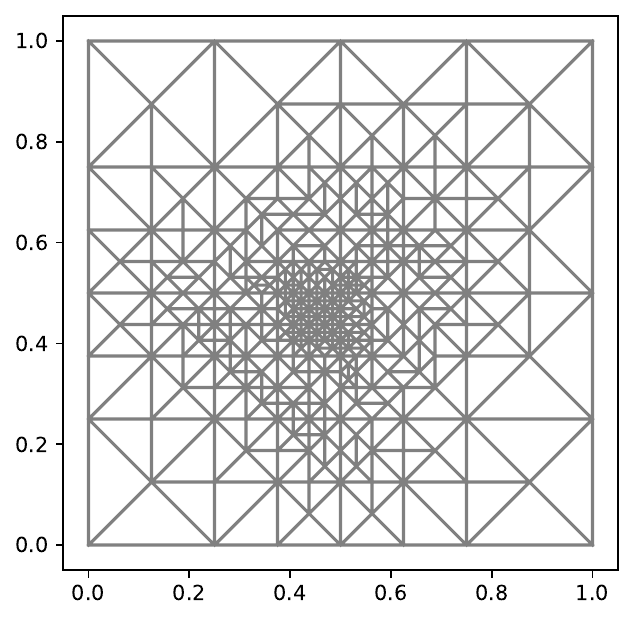}
		\includegraphics[width=0.325\textwidth]{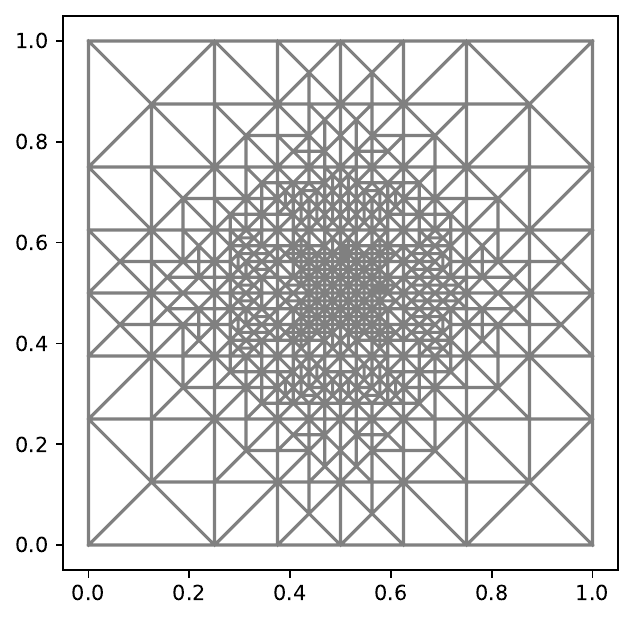}
		\includegraphics[width=0.325\textwidth]{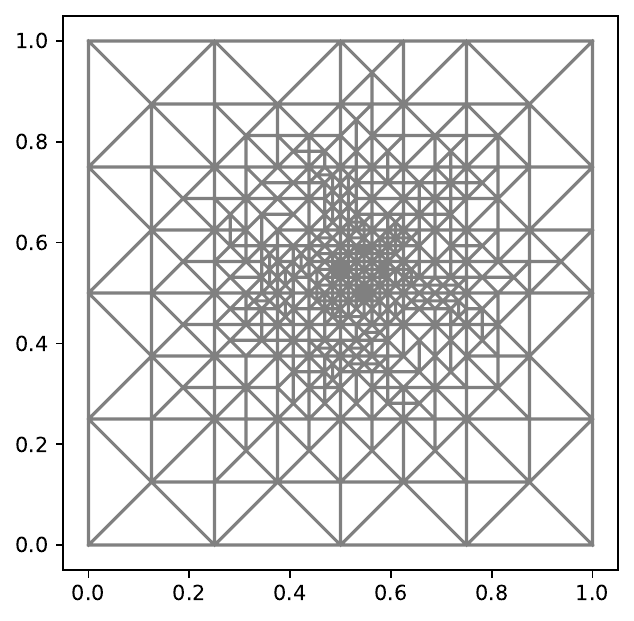}
		\includegraphics[width=0.325\textwidth]{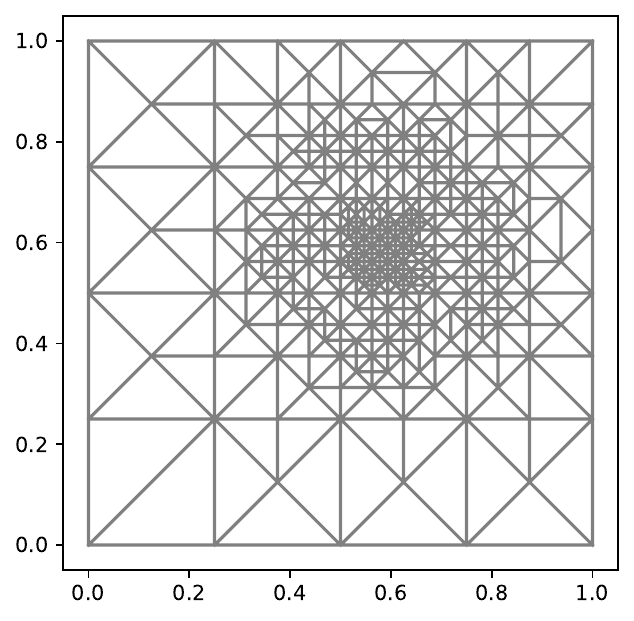}
		\includegraphics[width=0.325\textwidth]{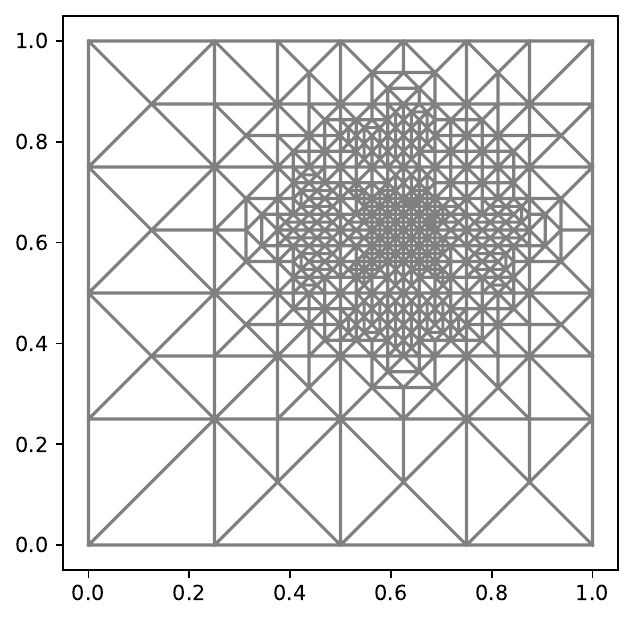}
		\includegraphics[width=0.325\textwidth]{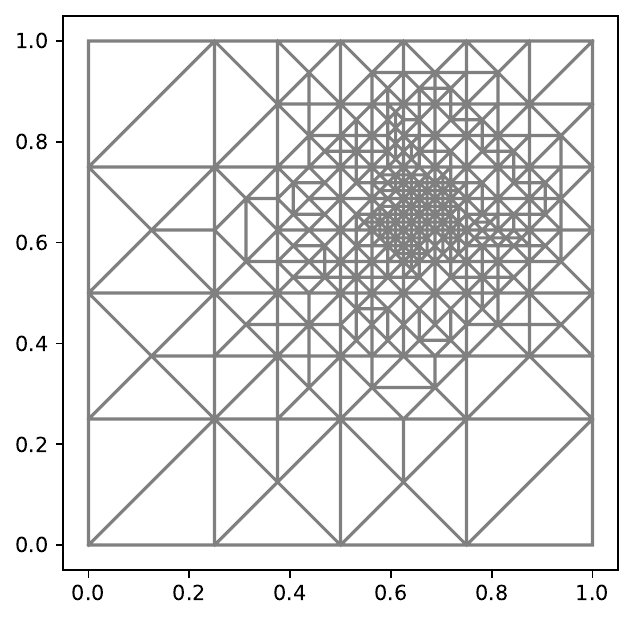}
		\includegraphics[width=0.325\textwidth]{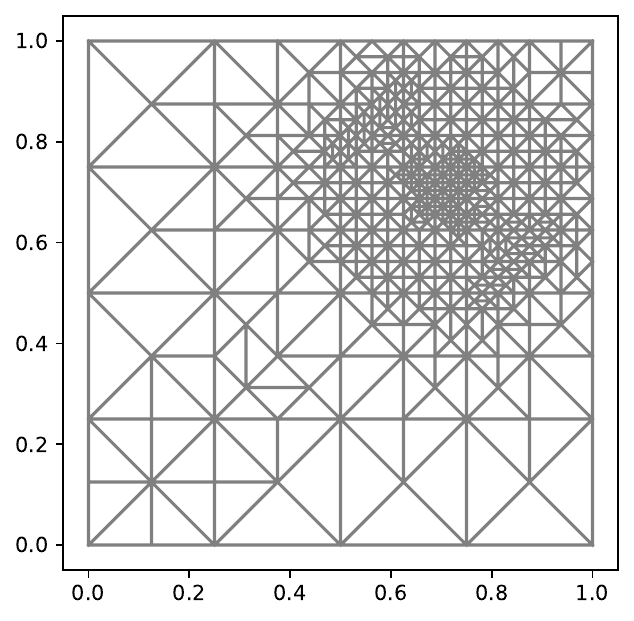}
	\end{center}
	\caption{\footnotesize{The adaptive mesh refinements after 7 iterations of the refined meshes obtained at each time step $t= n\times dt$ with $dt = 0.1$ and number of time steps, $n=10$.}}	\label{fig:ex7}
\end{figure}
\section{Conclusion Remark}
\bl{The current study takes into consideration the integro-differential equation, specifically focusing on the generalized Burgers-Huxley equation (GBHE) with weakly singular kernels. We conducted a comprehensive investigation into optimal $\L^2$ error estimates and a posteriori error estimates, establishing them for both steady and unsteady scenarios involving the GBHE. The study confirms the effectiveness of the proposed estimators in various settings, utilizing both backward Euler and Crank-Nicolson schemes. These findings enhance our understanding of convection-diffusion-reaction and partial integro-differential equation (PIDE) models, while also providing a solid foundation for future research. Furthermore, we identified promising avenues for further exploration, with $hp$-adaptivity for PIDEs using discontinuous Galerkin FEM being the primary focus of our ongoing research. Additionally, the a priori error estimates using $hp$-DGFEM have already been derived, as discussed in \cite{GBHE4}.}
\\
\textbf{Acknowledgement:} SM would like to express gratitude to the Ministry of Education, Government of India, for financial support through the Prime Minister Research Fellowship (PMRF ID: 2801816), enabling the conduct of his research work. AK was partially supported by SERB Core research grant CRG/2021/002569 and the faculty initiation grant MTD/FIG/100878. The authors also extend their thanks to Dr. Manil T. Mohan, Associate Professor, Department of Mathematics, IIT Roorkee  for their insightful discussions, which significantly enriched our work.\\
\textbf{Data availability}
In case of reasonable request, datasets generated during the research discussed in the paper will be available from the corresponding author.\\
\textbf{Declarations}
The authors declare that they have no conflict of interest.

\bibliographystyle{abbrv} 
\bibliography{3ref}

\begin{thebibliography}{10}

\bibitem{AJT}
M.~Ainsworth and J.~T. Oden.
\newblock {\em A posteriori error estimation in finite element analysis}.
\newblock Pure and Applied Mathematics (New York). Wiley-Interscience [John
  Wiley \& Sons], New York, 2000.

\bibitem{GPa}
G.~Akrivis and P.~Chatzipantelidis.
\newblock A posteriori error estimates for the two-step backward
  differentiation formula method for parabolic equations.
\newblock {\em SIAM J. Numer. Anal.}, 48(1):109--132, 2010.

\bibitem{ABJ}
M.~Aln{\ae}s, J.~Blechta, J.~Hake, A.~Johansson, B.~Kehlet, A.~Logg,
  C.~Richardson, J.~Ring, M.~E. Rognes, and G.~N. Wells.
\newblock The fenics project version 1.5.
\newblock {\em Archive of Numerical Software}, 3, 2015.

\bibitem{Arn}
D.~N. Arnold.
\newblock An interior penalty finite element method with discontinuous
  elements.
\newblock {\em SIAM J. Numer. Anal.}, 19:742--760, 1982.

\bibitem{Bah}
A.~R. Bahad{\i}r.
\newblock A fully implicit finite-difference scheme for two-dimensional
  {B}urgers' equations.
\newblock {\em Appl. Math. Comput.}, 137(1):131--137, 2003.

\bibitem{BSc}
S.~C. Brenner and L.~R. Scott.
\newblock {\em The mathematical theory of finite element methods}.
\newblock Springer, New York, 2008.

\bibitem{BOr}
A.~Buffa and C.~Ortner.
\newblock Compact embeddings of broken {S}obolev spaces and applications.
\newblock {\em IMA J. Numer. Anal.}, 29(4):827--855, 2009.

\bibitem{BKM}
R.~B{\"u}rger, A.~Khan, P.~E. M{\'e}ndez, and R.~Ruiz-Baier.
\newblock Divergence-conforming methods for transient doubly-diffusive flows: A
  priori and a posteriori error analysis.
\newblock {\em IMA J. Numer. Anal.}, 44(6):3520--3572, 2024.

\bibitem{CSG}
A.~Cangiani, E.~H. Georgoulis, and M.~Sabawi.
\newblock {\it {A} posteriori} error analysis for implicit-explicit
  {$hp$}-discontinuous {G}alerkin timestepping methods for semilinear parabolic
  problems.
\newblock {\em J. Sci. Comput.}, 82(2):Paper No. 26, 24, 2020.

\bibitem{CTs}
C.~Chen and T.~Shih.
\newblock {\em Finite element methods for integrodifferential equations},
  volume~9.
\newblock World Scientific, 1998.

\bibitem{CBD}
M.~Corti, F.~Bonizzoni, L.~Dede', A.~M. Quarteroni, and P.~F. Antonietti.
\newblock Discontinuous {G}alerkin methods for {F}isher-{K}olmogorov equation
  with application to {$\alpha$}-synuclein spreading in {P}arkinson's disease.
\newblock {\em Comput. Methods Appl. Mech. Engrg.}, 417:Paper No. 116450, 17,
  2023.

\bibitem{FNJ}
J.~de~Frutos and J.~Novo.
\newblock A posteriori error estimation with the {$p$}-version of the finite
  element method for nonlinear parabolic differential equations.
\newblock {\em Comput. Methods Appl. Mech. Engrg.}, 191(43):4893--4904, 2002.

\bibitem{DFK}
V.~Dolejsi, M.~Feistauer, V.~Kucera, and V.~Sobotikova.
\newblock An optimal {$L^\infty(L^2)$}-error estimate for the discontinuous
  {G}alerkin approximation of a nonlinear non-stationary convection-diffusion
  problem.
\newblock {\em IMA J. Numer. Anal.}, 28(3):496--521, 2008.

\bibitem{Elt}
B.~H. Elton.
\newblock Comparisons of lattice {B}oltzmann and finite difference methods for
  a two-dimensional viscous {B}urgers equation.
\newblock {\em SIAM J. Sci. Comput.}, 17(4):783--813, 1996.

\bibitem{KJC}
K.~Eriksson and C.~Johnson.
\newblock An adaptive finite element method for linear elliptic problems.
\newblock {\em Math. Comp.}, 50(182):361--383, 1988.

\bibitem{EVM}
A.~Ern and M.~Vohral\'{\i}k.
\newblock A posteriori error estimation based on potential and flux
  reconstruction for the heat equation.
\newblock {\em SIAM J. Numer. Anal.}, 48(1):198--223, 2010.

\bibitem{EMR}
V.~J. Ervin, J.~E. Mac\'{\i}as-D\'{\i}az, and J.~Ruiz-Ram\'{\i}rez.
\newblock A positive and bounded finite element approximation of the
  generalized {B}urgers'-{H}uxley equation.
\newblock {\em J. Math. Anal. Appl.}, 424:1143--1160, 2015.

\bibitem{Fle}
C.~A.~J. Fletcher.
\newblock A comparison of finite element and finite difference solutions of the
  one- and two-dimensional {B}urgers' equations.
\newblock {\em J. Comput. Phys.}, 51(1):159--188, 1983.

\bibitem{GOC}
E.~H. Georgoulis, O.~Lakkis, C.~G. Makridakis, and J.~M. Virtanen.
\newblock A posteriori error estimates for leap-frog and cosine methods for
  second order evolution problems.
\newblock {\em SIAM J. Numer. Anal.}, 54(1):120--136, 2016.

\bibitem{EOJ}
E.~H. Georgoulis, O.~Lakkis, and J.~M. Virtanen.
\newblock A posteriori error control for discontinuous {G}alerkin methods for
  parabolic problems.
\newblock {\em SIAM J. Numer. Anal.}, 49(2):427--458, 2011.

\bibitem{HNS}
I.~Hashim, M.~S.~M. Noorani, and M.~R. Said Al-Hadidi.
\newblock Solving the generalized {B}urgers'-{H}uxley equation using the
  {A}domian decomposition method.
\newblock {\em Math. Comput. Modelling}, 43:1404--1411, 2006.

\bibitem{HRT}
F.~He${\ss}$e, F.~A. Radu, M.~Thullner, and S.~Attinger.
\newblock Upscaling of the advection--diffusion--reaction equation with {Monod}
  reaction.
\newblock {\em Adv. Water Resour.}, 32(8):1336--1351, 2009.

\bibitem{JCA}
J.~Huang, Z.~Cen, A.~Xu, and L.-B. Liu.
\newblock A posteriori error estimation for a singularly perturbed {V}olterra
  integro-differential equation.
\newblock {\em Numer. Algorithms}, 83(2):549--563, 2020.

\bibitem{HAb}
P.~Huang and A.~Abduwali.
\newblock The modified local {C}rank-{N}icolson method for one- and
  two-dimensional {B}urgers' equations.
\newblock {\em Comput. Math. Appl.}, 59(8):2452--2463, 2010.

\bibitem{HSI}
A.~Hussain, S.~Bano, I.~Khan, D.~Baleanu, and K.~Sooppy~Nisar.
\newblock Lie symmetry analysis, explicit solutions and conservation laws of a
  spatially two-dimensional burgers--huxley equation.
\newblock {\em Symmetry}, 12(1):170, 2020.

\bibitem{KPa}
O.~A. Karakashian and F.~Pascal.
\newblock A posteriori error estimates for a discontinuous {G}alerkin
  approximation of second-order elliptic problems.
\newblock {\em SIAM J. Numer. Anal.}, 41(6):2374--2399, 2003.

\bibitem{KC}
W.~E. Kastenberg and P.~L. Chambr{\'e}.
\newblock On the stability of nonlinear space-dependent reactor kinetics.
\newblock {\em Nucl. Sci. Eng.}, 31(1):67--79, 1968.

\bibitem{KCP}
W.~E. Kastenberg and P.~L. Chambr{\'e}.
\newblock On the stability of nonlinear space-dependent reactor kinetics.
\newblock {\em Nuclear Science and Engineering}, 31(1):67--79, 1968.

\bibitem{KMR}
A.~Khan, M.~T. Mohan, and R.~Ruiz-Baier.
\newblock Conforming, nonconforming and {DG} methods for the stationary
  generalized {B}urgers'-{H}uxley equation.
\newblock {\em J. Sci. Comput.}, 88:1--26, 2021.

\bibitem{KRS}
A.~Khater, R.~Temsah, and M.~Hassan.
\newblock A chebyshev spectral collocation method for solving burgers’-type
  equations.
\newblock {\em Journal of computational and applied mathematics},
  222(2):333--350, 2008.

\bibitem{KHA}
A.~J. Khattak.
\newblock A computational meshless method for the generalized
  {B}urger's-{H}uxley equation.
\newblock {\em Appl. Math. Model.}, 33:3718--3729, 2009.

\bibitem{KSM}
B.~V.~R. Kumar, V.~Sangwan, S.~V. S. S. N. V. G.~K. Murthy, and M.~Nigam.
\newblock A numerical study of singularly perturbed generalized
  {B}urgers'-{H}uxley equation using three-step {T}aylor-{G}alerkin method.
\newblock {\em Comput. Math. Appl.}, 62:776--786, 2011.

\bibitem{VAC}
V.~Ku\v{c}era.
\newblock Optimal {$L^\infty(L^2)$}-error estimates for the {DG} method applied
  to nonlinear convection-diffusion problems with nonlinear diffusion.
\newblock {\em Numer. Funct. Anal. Optim.}, 31(1-3):285--312, 2010.

\bibitem{LMC}
O.~Lakkis and C.~Makridakis.
\newblock Elliptic reconstruction and a posteriori error estimates for fully
  discrete linear parabolic problems.
\newblock {\em Math. Comp.}, 75(256):1627--1658, 2006.

\bibitem{LTW}
S.~Larsson, V.~Thom\'{e}e, and L.~B. Wahlbin.
\newblock Numerical solution of parabolic integro-differential equations by the
  discontinuous {G}alerkin method.
\newblock {\em Math. Comp.}, 67(221):45--71, 1998.

\bibitem{LWH}
T.~D. Leta, W.~Liu, H.~Rezazadeh, J.~Ding, and A.~E. Achab.
\newblock Analytical traveling wave and soliton solutions of the (2+ 1)
  dimensional generalized burgers--huxley equation.
\newblock {\em Qualitative Theory of Dynamical Systems}, 20:1--23, 2021.

\bibitem{LCS}
Q.~Li, Z.~Chai, and B.~Shi.
\newblock Lattice {B}oltzmann models for two-dimensional coupled {B}urgers'
  equations.
\newblock {\em Comput. Math. Appl.}, 75(3):864--875, 2018.

\bibitem{LWZ}
X.~Liu, J.~Wang, and Y.~Zhou.
\newblock A space-time fully decoupled wavelet {G}alerkin method for solving
  two-dimensional {B}urgers' equations.
\newblock {\em Comput. Math. Appl.}, 72(12):2908--2919, 2016.

\bibitem{GBHE2}
S.~Mahajan and A.~Khan.
\newblock Finite element approximation for the delayed generalized
  {B}urgers-{H}uxley equation with weakly singular kernel: {P}art {II}
  {N}onconforming and {DG} approximation.
\newblock {\em SIAM J. Sci. Comput.}, 46(5):A2972--A2998, 2024.

\bibitem{GBHE4}
S.~Mahajan and A.~Khan.
\newblock hp-discontinuous galerkin method for the generalized burgers-huxley
  equation with weakly singular kernels.
\newblock {\em arXiv preprint arXiv:2409.00818}, 2024.

\bibitem{GBHE}
S.~Mahajan, A.~Khan, and M.~T. Mohan.
\newblock Finite element approximation for a delayed generalized
  {B}urgers-{H}uxley equation with weakly singular kernels: {P}art {I}
  well-posedness, regularity and conforming approximation.
\newblock {\em Comput. Math. Appl.}, 174:261--286, 2024.

\bibitem{MNo}
C.~Makridakis and R.~H. Nochetto.
\newblock Elliptic reconstruction and a posteriori error estimates for
  parabolic problems.
\newblock {\em SIAM J. Numer. Anal.}, 41(4):1585--1594, 2003.

\bibitem{MTW}
W.~McLean, V.~Thom\'{e}e, and L.~B. Wahlbin.
\newblock Discretization with variable time steps of an evolution equation with
  a positive-type memory term.
\newblock {\em J. Comput. Appl. Math.}, 69(1):49--69, 1996.

\bibitem{SNP}
S.~Memon, N.~Nataraj, and A.~K. Pani.
\newblock An a posteriori error analysis of mixed finite element {G}alerkin
  approximations to second order linear parabolic problems.
\newblock {\em SIAM J. Numer. Anal.}, 50(3):1367--1393, 2012.

\bibitem{MKH}
M.~T. Mohan and A.~Khan.
\newblock On the generalized {B}urgers'-{H}uxley equation: {E}xistence,
  uniqueness, regularity, global attractors and numerical studies.
\newblock {\em Discrete Contin. Dyn. Syst. Ser. B}, 26:3943--3988, 2021.

\bibitem{GMMRIMA}
G.~Murali Mohan~Reddy and R.~K. Sinha.
\newblock Ritz-{V}olterra reconstructions and {\it a posteriori} error analysis
  of finite element method for parabolic integro-differential equations.
\newblock {\em IMA J. Numer. Anal.}, 35(1):341--371, 2015.

\bibitem{MBM}
K.~Mustapha, H.~Brunner, H.~Mustapha, and D.~Sch\"{o}tzau.
\newblock An {$hp$}-version discontinuous {G}alerkin method for
  integro-differential equations of parabolic type.
\newblock {\em SIAM J. Numer. Anal.}, 49(4):1369--1396, 2011.

\bibitem{MMh}
K.~Mustapha and H.~Mustapha.
\newblock A second-order accurate numerical method for a semilinear
  integro-differential equation with a weakly singular kernel.
\newblock {\em IMA J. Numer. Anal.}, 30(2):555--578, 2010.

\bibitem{NAG}
T.~Nagatani.
\newblock Density waves in traffic flow.
\newblock {\em Phys. Rev. E}, 61(4):3564, 2000.

\bibitem{NSo}
S.~Nicaise and N.~Soualem.
\newblock A posteriori error estimates for a nonconforming finite element
  discretization of the heat equation.
\newblock {\em M2AN Math. Model. Numer. Anal.}, 39(2):319--348, 2005.

\bibitem{PMa}
M.~Picasso.
\newblock Adaptive finite elements for a linear parabolic problem.
\newblock {\em Comput. Methods Appl. Mech. Engrg.}, 167(3-4):223--237, 1998.

\bibitem{Rad}
S.~F. Radwan.
\newblock Comparison of higher-order accurate schemes for solving the
  two-dimensional unsteady {B}urgers' equation.
\newblock {\em J. Comput. Appl. Math.}, 174(2):383--397, 2005.

\bibitem{GMMRMC}
G.~M.~M. Reddy and R.~K. Sinha.
\newblock On the {C}rank-{N}icolson anisotropic a posteriori error analysis for
  parabolic integro-differential equations.
\newblock {\em Math. Comp.}, 85(301):2365--2390, 2016.

\bibitem{GMMRJSc}
G.~M.~M. Reddy, R.~K. Sinha, and J.~A. Cuminato.
\newblock A posteriori error analysis of the {C}rank-{N}icolson finite element
  method for parabolic integro-differential equations.
\newblock {\em J. Sci. Comput.}, 79(1):414--441, 2019.

\bibitem{SGZ}
M.~Sari, G.~Gurarslan, and A.~Zeytinougu.
\newblock High-order finite difference schemes for numerical solutions of the
  generalized burgers-huxley equation.
\newblock {\em Numer. Methods Partial Differential Equations}, 27:1313--1326,
  2011.

\bibitem{SZh}
D.~Sch\"{o}tzau and L.~Zhu.
\newblock A robust a-posteriori error estimator for discontinuous {G}alerkin
  methods for convection-diffusion equations.
\newblock {\em Appl. Numer. Math.}, 59(9):2236--2255, 2009.

\bibitem{WLD}
W.~Shen, L.~Ge, D.~Yang, and W.~Liu.
\newblock Sharp a posteriori error estimates for optimal control governed by
  parabolic integro-differential equations.
\newblock {\em J. Sci. Comput.}, 65(1):1--33, 2015.

\bibitem{HAR}
H.~Singh and A.~Khan.
\newblock Divergence conforming {DG} method for the optimal control of the
  {O}seen equation with variable viscosity.
\newblock {\em SIAM J. Sci. Comput.}, 47(2):A1251--A1278, 2025.

\bibitem{SSA}
A.~Sreelakshmi, V.~P. Shyaman, and A.~Awasthi.
\newblock An interwoven composite tailored finite point method for two
  dimensional unsteady {B}urgers' equation.
\newblock {\em Appl. Numer. Math.}, 197:71--96, 2024.

\bibitem{TKM}
J.~Tushar, A.~Khan, and M.~T. Mohan.
\newblock Optimal control of stationary doubly diffusive flows on two and three
  dimensional bounded lipschitz domains: Numerical analysis.
\newblock {\em arXiv preprint arXiv:2403.10282}, 2024.

\bibitem{Rcal}
R.~Verf\"{u}rth.
\newblock A posteriori error estimates for finite element discretizations of
  the heat equation.
\newblock {\em Calcolo}, 40(3):195--212, 2003.

\bibitem{VerS}
R.~Verf\"{u}rth.
\newblock Robust a posteriori error estimates for stationary
  convection-diffusion equations.
\newblock {\em SIAM J. Numer. Anal.}, 43(4):1766--1782, 2005.

\bibitem{VERB}
R.~Verf\"{u}rth.
\newblock {\em A posteriori error estimation techniques for finite element
  methods}.
\newblock Numerical Mathematics and Scientific Computation. Oxford University
  Press, Oxford, 2013.

\bibitem{XYW}
X.~Wang, Z.~Zhu, and Y.~Lu.
\newblock Solitary wave solutions of the generalised burgers-huxley equation.
\newblock {\em Journal of Physics A: Mathematical and General}, 23(3):271,
  1990.

\bibitem{YJi}
O.~P. Yadav and R.~Jiwari.
\newblock Finite element analysis and approximation of {B}urgers'-{F}isher
  equation.
\newblock {\em Numer. Methods Partial Differential Equations},
  33(5):1652--1677, 2017.

\bibitem{YK}
O.~Y. Yefimova and N.~A. Kudryashov.
\newblock Exact solutions of the {Burgers--Huxley} equation.
\newblock {\em J. Appl. Math. Mech.}, 68(3):413--420, 2004.

\bibitem{ZYa}
J.~Zhang and G.~Yan.
\newblock Lattice {B}oltzmann method for one and two-dimensional {B}urgers
  equation.
\newblock {\em Phys. A}, 387(19-20):4771--4786, 2008.

\bibitem{ZOW}
L.~Zhang, J.~Ouyang, X.~Wang, and X.~Zhang.
\newblock Variational multiscale element-free {G}alerkin method for 2{D}
  {B}urgers' equation.
\newblock {\em J. Comput. Phys.}, 229(19):7147--7161, 2010.

\bibitem{ZNy}
N.~Y. Zhang.
\newblock On fully discrete {G}alerkin approximations for partial
  integro-differential equations of parabolic type.
\newblock {\em Math. Comp.}, 60(201):133--166, 1993.

\bibitem{ZLi}
Z.~Zhao and H.~Li.
\newblock Numerical study of two-dimensional {B}urgers' equation by using a
  continuous {G}alerkin method.
\newblock {\em Comput. Math. Appl.}, 149:38--48, 2023.

\end{thebibliography}
\clearpage

				\end{document}